\documentclass[10pt]{article}
\usepackage{amsfonts}
\usepackage{epsfig,epic}

\usepackage{amsmath}
\usepackage{amsfonts}
\usepackage{amssymb,amsthm}

\newcommand{\calA}{{\mathcal A}}
\newcommand{\calZ}{{\mathcal Z}}

\newcommand{\calM}{{\mathcal M}}
\newcommand{\calH}{{\mathcal H}}
\newcommand{\calG}{{\mathcal G}}
\newcommand{\calC}{{\mathcal C}}

\newcommand{\calU}{{\mathcal U}}

\newcommand{\N}{{\mathbb N}}
\renewcommand{\P}{\N_{>0}}
\newcommand{\Z}{{\mathbb Z}}
\newcommand{\Q}{{\mathbb Q}}
\newcommand{\R}{{\mathbb R}}
\newcommand{\C}{{\mathbb C}}
\newcommand{\K}{{\mathbb K}}

\newcommand{\Frac}[2]{\displaystyle \frac{#1}{#2}}
\newcommand{\Sum}[2]{\displaystyle{\sum_{#1}^{#2}}}
\newcommand{\Prod}[2]{\displaystyle{\prod_{#1}^{#2}}}

\newcommand{\Lim}[1]{\displaystyle{\lim_{#1}\ }}

\newcommand{\eps}{{\varepsilon}}
\newcommand{\e}{{\epsilon}}

\def\pol#1{\langle #1 \rangle}
\def\Lyn{{\mathcal Lyn}}
\def\Lirr{{\mathcal L_{irr}}}
\def\llm#1{\mathcal #1}

\def\CX{\C \langle X \rangle}

\def\abs#1{|#1|}
\def\absv#1{\|#1\|}
\def\Conv{{\rm CV}}
\def\AXX{\serie{A}{X}}
 \def\shuffle{\mathop{_{^{\sqcup\!\sqcup}}}} 

\gdef\stuffle{\;%
  \setlength{\unitlength}{0.0125cm}%
  \begin{picture}(20,10)(220,580) 
  \thinlines 
  \put(220,592){\line( 0,-1){ 10}} 
  \put(220,582){\line( 1, 0){ 20}} 
  \put(240,582){\line( 0, 1){ 10}} 
  \put(230,592){\line( 0,-1){ 10}} 
  \put(225,587){\line( 1, 0){ 10}} 
  \end{picture}\; 
}

\newtheorem{corollary}{Corollary}
\newtheorem{proposition}{Proposition}
\newtheorem{theorem}{Theorem}
\newtheorem{lemma}{Lemma}
\newtheorem{definition}{Definition}

\newtheorem{example}{Example}

\newtheorem{remark}{Remark}

\newcommand{\Li}{\operatorname{Li}}
\newcommand{\ad}{\operatorname{ad}}

\def\LI{\mathrm{LI}}
\def\L{\mathrm{L}}
\def\H{\mathrm{H}}

\def\P{\mathrm{P}}

\def\A{\mathrm{A}}
\def\Ann{\mathrm{Ann}}
\def\F{\mathrm{F}}

\def\deg{\mathrm{deg}}
\def\wgt{\mathrm{wgt}}

\newcommand{\poly}[2]{#1 \langle #2 \rangle}

\def\QX{\poly{\Q}{X}}

\def\CX{\poly{\C}{X}}

\def\QY{\poly{\Q}{Y}}
\def\QYb{\poly{\Q}{\bar Y}}

\newcommand{\serie}[2]{#1 \langle \! \langle #2 \rangle \! \rangle}
\def\AXX{\serie{A}{X}}
\def\AYY{\serie{A}{Y}}

\def\QYY{\serie{\Q}{Y}}

\def\KXX{\serie{\K}{X}}

\def\QXX{\serie{\Q}{X}}
\def\CcXX{\serie{\C^{\mathrm{cont}}}{X}}

\def\QratXX{\serie{\Q^{\mathrm{rat}}}{X}}

\def\CXX{\serie{\C}{X}}

\def\pol#1{\langle #1 \rangle}
\def\ser#1{\langle\!\langle #1 \rangle\!\rangle}
\def\AX{A \langle X \rangle}
\def\AY{A \langle Y \rangle}
\def\AYb{A \langle \bar Y \rangle}
\def\KX{\K \langle X \rangle}

\newcommand{\calT}{\mathcal{T}}

\def\Lie{{\cal L}ie}

\def\LAXX{\Lie_{A} \langle\!\langle X \rangle \!\rangle}
\def\LQX{\Lie_{\Q} \langle X \rangle}
\def\LQY{\Lie_{\Q} \langle Y \rangle}
\def\LQYb{\Lie_{\Q} \langle\bar Y \rangle}
\def\LKX{\Lie_{\K} \langle X \rangle}

\def\LXX{\Lie_{\C} \langle\!\langle X \rangle \!\rangle}

\def\CX{\C \langle X \rangle}
\def\CXX{\serie{\C}{X}}

\newcommand{\calR}{{\cal R}}

\def\Lim{\displaystyle\lim}
\def\Sum{\displaystyle\sum}
\def\Prod{\displaystyle\prod}
\def\Frac{\displaystyle\frac}
\def\path{\rightsquigarrow}
\def\reg{\mathop\mathrm{reg}\nolimits}

\def\resg{\triangleleft}
\def\resd{\triangleright}
\def\bv{\mid}
\def\bbv{\bv\!\bv}
\def\abs#1{\bv\!#1\!\bv}

\gdef\minishuffle{{\scriptstyle \shuffle}}  
\gdef\ministuffle{{\scriptstyle \stuffle}}  

\newcommand{\myover}[2]{\genfrac{}{}{0pt}{}{#1}{#2}}
\newcommand{\mychoose}[2]{\genfrac{(}{)}{0pt}{}{#1}{#2}}

\def\scal#1#2{\langle #1\bv#2 \rangle}

\def\deg{\mathop\mathrm{deg}\nolimits}

\def\supp{\mathop\mathrm{supp}\nolimits}

\def\binom#1#2{{#1\choose#2}}

\begingroup
\def\2#1{\ifnum#1<10 0\fi\the#1}
\xdef\isodayandtime{
{\2\day-\2\month-\the\year\space\2{\count0}:%
\2{\count2}}}
\endgroup

\newcounter{per1}
\begin{document}

\begin{center}
{\Large On a conjecture by Pierre Cartier\\
about a group of associators}

\bigskip{ \Large Hoang Ngoc Minh}

\medskip \noindent

LIPN - UMR 7030, CNRS, 93430 Villetaneuse, France.\\
Universit\'e Lille II, 1, Place D\'eliot, 59024 Lille, France.
\end{center}

\begin{abstract}
In \cite{cartier2}, Pierre Cartier conjectured that for any non commutative formal
power series $\Phi$ on $X=\{x_0,x_1\}$ with coefficients in a
$\Q$-extension, $A$, subjected to some suitable conditions, there
exists an unique algebra homomorphism $\varphi$ from the
$\Q$-algebra generated by the convergent polyz\^etas to $A$ such
that $\Phi$ is computed from $\Phi_{KZ}$ Drinfel'd associator by
applying $\varphi$ to each coefficient. We prove $\varphi$
exists and it is a free Lie exponential over $X$. Moreover,
we give a complete description of the kernel of
polyz\^eta and draw some consequences about a structure of the
algebra of convergent polyz\^etas and about the arithmetical nature of the Euler
constant.
\end{abstract}

\tableofcontents

\section{Introduction}
\subsection{Drinfel'd associator and polyz\^etas}
In 1986, in order to study the linear representation of the braid group $B_n$
coming from the monodromy of the Knizhnik-Zamolodchi\-kov differential equations over
$\C_*^n=\{\underline{z}=(z_1,\ldots,z_n)\in\C^n\bv z_i\neq z_j\mbox{ for }i\neq j\}$ \cite{drinfeld}~:
\begin{eqnarray}\label{KZn}
dF(\underline{z})=\Omega_n(\underline{z})F(\underline{z})
&\mbox{with}&
\Omega_n(\underline{z})=\frac{1}{2{\rm i}\pi}\sum_{1\le i<j\le n}t_{i,j}\frac{d(z_i-z_j)}{z_i-z_j},
\end{eqnarray}
and $\{t_{i,j}\}_{i,j\ge1}$ are noncommutative variables,
Drinfel'd introduced a class of formal power series $\Phi$
on noncommutative variables over the finite alphabet $X=\{x_0,x_1\}$.
Such a power series $\Phi$ is called an {\it associator}.

Since the system (\ref{KZn}) is completely integrable then $d\Omega_n-\Omega_n\wedge\Omega_n=0$ \cite{cartier,drinfeld}.
It is equivalent to the fact the $\{t_{i,j}\}_{i,j\ge1}$ satisfy the infinitesimal braid relations~:
\begin{eqnarray}
t_{i,j}=0\phantom{{}_{j,i}}&\mbox{for}&i=j,\\
t_{i,j}=t_{j,i}&\mbox{for}&i\neq j,\\[0pt]
[t_{i,j},t_{i,k}+t_{j,k}]=0\phantom{{}_{j,i}}&\mbox{for distinct}&i,j,k,\\[0pt]
[t_{i,j},t_{k,l}]=0\phantom{{}_{j,i}}&\mbox{for distinct}&i,j,k,l.
\end{eqnarray}

\begin{example}\label{KZ}
\begin{itemize}
\item $\calT_2=\{t_{1,2}\}$.
\begin{eqnarray*}
\Omega_2(z_1,z_2)=\frac{t_{1,2}}{2{\rm i}\pi}\frac{d(z_1-z_2)}{z_1-z_2}
&\mbox{with}&
F(z_1,z_2)=(z_1-z_2)^{t_{1,2}/2{\rm i}\pi}.
\end{eqnarray*}

\item $\calT_3=\{t_{1,2},t_{1,3},t_{2,3}\},\ [t_{1,3},t_{1,2}+t_{2,3}]=[t_{2,3},t_{1,2}+t_{1,3}]=0$.
\begin{eqnarray*}
\Omega_3(z_1,z_2,z_3)&=&\frac{1}{2\mathrm{i}\pi}
\biggl[t_{1,2}\frac{d(z_1-z_2)}{z_1-z_2}+t_{1,3}\frac{d(z_1-z_3)}{z_1-z_3}
+t_{2,3}\frac{d(z_2-z_3)}{z_2-z_3}\biggr].\\
F(z_1,z_2,z_3)&=&G\biggl(\frac{z_1-z_2}{z_1-z_3}\biggr)
(z_1-z_3)^{(t_{1,2}+t_{1,3}+t_{2,3})/2\mathrm{i}\pi},
\end{eqnarray*}
where $G$ satisfies the following fuchsian differential equation with
three regular singularities at $0,1$ and $\infty$~:
\begin{eqnarray*}
\hskip-3cm(DE)\hskip35mm dG(z)&=&[x_0\;\omega_0(z)+x_1\;\omega_1(z)]G(z),
\end{eqnarray*}
with
\begin{eqnarray*}
x_0:=\phantom{-}\frac{t_{1,2}}{2\mathrm{i}\pi}&\mbox{and}&\omega_0(z):=\Frac{dz}{z},\\
x_1:=-\frac{t_{2,3}}{2\mathrm{i}\pi}&\mbox{and}&\omega_1(z):=\Frac{dz}{1-z}.
\end{eqnarray*}
\end{itemize}
\end{example}
As already shown by Drinfel'd, the equation $(DE)$ admits, on the simply connected domain
$\C-(]-\infty,0]\cup[1,+\infty[)$, two specific solutions
\begin{eqnarray}\label{particularsolution}
G_0(z){}_{\widetilde{z\path0}}\exp[x_0\log(z)]&\mbox{and}&
G_1(z){}_{\widetilde{z\path1}}\exp[-x_1\log(1-z)].
\end{eqnarray}
He also proved  there exists the associator $\Phi_{KZ}$ such that $G_1^{-1}(z)G_0(z)=\Phi_{KZ}$.

After that, L\^e and Murakami expressed the coefficients of the Drinfel'd associator
$\Phi_{KZ}$ in terms of {\it convergent} polyz\^etas \cite{lemurakami},
{\em i.e.} for $r_1>1$,
\begin{eqnarray}
\zeta(r_1,\ldots,r_k)&=&\sum_{n_1>\ldots>n_k>0}\Frac1{n_1^{r_1}\ldots n_k^{r_k}}.
\end{eqnarray}
In \cite{lemurakami}, the authors also expressed the {\it divergent} coefficients
as {\it linear} combinations of convergent polyz\^etas via a {\it regularization process}
(see also \cite{orlando}). This process is one of many ways to regularize the divergent terms.

\subsection{Group of associators and regularized Chen generating series}

The algebraic aspects of our regularization process based essentially on various products\footnote{First source of ambiguity leading to the problem of rewriting expressions of polyz\^etas in a canonical form using irreducible Lyndon words (see \cite{FPSAC97,SLC43}).} among polyz\^etas (see \cite{SLC44}) and its analytical aspects will be described, in Section \ref{regularization}, as the {\it finite part}, of the asymptotic expansions in different scales of comparison\footnote{Second source of ambiguity leading to the problem to determine the value of regularized polyz\^etas and its analytical meaning (see \cite{SLC44,words03}).} \cite{bourbaki}. It will be seen also, in Section \ref{galois}, as the action of the differential Galois group of the polylogarithms\footnote{Third source of ambiguity leading to the problem of fixing the integration path to solve $(DE)$ and its monodromy group (see \cite{FPSAC98}) or its differential Galois group (see \cite{orlando}).} (recalled in Section \ref{polylogarithm})
\begin{eqnarray}
\Li_{r_1,\ldots,r_k}(z)&=&\sum_{n_1>\ldots>n_k>0}\Frac{z^{n_1}}{n_1^{r_1}\ldots n_k^{r_k}}
\end{eqnarray}
on the asymptotic expansion of polylogarithms, at $z=1$ and  in the comparison scale
$\{(1-z)^{a}\log^{b}(1-z)\}_{a\in\Z,b\in\N}$,  and the same action on the asymptotic expansions, at $+\infty$ and  in the comparison scales
$\{n^{a}\log^{b}(n)\}_{a\in\Z,b\in\N}$ and $\{n^{a}\H_1^{b}(n)\}_{a\in\Z,b\in\N}$,
of the harmonic sums (recalled in Section \ref{harmonic})
\begin{eqnarray}
\H_{r_1,\ldots,r_k}(N)&=&\sum_{n_1>\ldots>n_k>0}^N\Frac1{n_1^{r_1}\ldots n_k^{r_k}}.
\end{eqnarray}

This action leads then to a conjecture by Pierre Cartier (\cite{cartier2}, conjecture C3) and to the description of the group of associators yielding the ideal of polynomial relations among coefficients of associators (theorems \ref{associatorgp} and \ref{alginpdt}).
This group is in fact, closely linked to the group of the Chen generating series studied by K.T.~Chen to describe the solutions of differential equations \cite{chen} and it turns out that each associator regularizes a Chen generating series of the differential forms $\omega_0$ and $\omega_1$ along the integration path  on the simply connected domain $\C-(]-\infty,0]\cup[1,+\infty[)$. 

\subsection{Global renormalization and global regularization}
In fact, our regularization process based essentially on two noncommutative
generating series over $Y=\{y_i\}_{i\ge1}$,
which encodes the multi-indices $(r_1,\ldots,r_k)$ by the words
$y_{r_1}\ldots y_{r_k}$ over the monoid generated by $Y$, denoted by $Y^*$,
of polylogarithms and of harmonic sums  (recalled in Section \ref{generatingseries})
\begin{eqnarray}
\Lambda(z)=\sum_{w\in Y^*}\Li_w(z)\;w
&\mbox{and}&
\H(N)=\sum_{w\in Y^*}\H_w(N)\;w.
\end{eqnarray}

Through the algebraic combinatorial aspects\footnote{See \cite{reutenauer} to get an idea of these aspects of combinatorial Hopf algebra of the shuffle product, denoted by $\shuffle$, and its co-product, denoted by $\Delta_{\minishuffle}$. For the quasi-shuffle product, denoted by $\stuffle$, and its co-product, denoted by $\Delta_{\ministuffle}$, see Annexe A.

In our works, recalled in Annexe B, these algebraic combinatorial aspects were explored systematically to expand the outputs of nonlinear controlled dynamical system with singular inputs (Corollary \ref{output}) on polylogarithmic functional basis \cite{IMACS0,hoangjacoboussous,FPSAC96}. In this way \cite{cade}, polyz\^etas do appear then as fundamental arithmetical constant for the asymptotic analysis and for the renormalization of the outputs and their successive derivations (Corollary \ref{asymptoticofoutput}) via the extended Fliess fundamental formula (Theorem \ref{fondamentalformula}).} \cite{reutenauer} and the topological aspects \cite{berstel} of formal power series in noncommutative variables, we have already showed the existence of noncommutative formal series over $Y$, $Z_1$ and $Z_2$ with constant terms, such that \cite{cade}
\begin{eqnarray}
\Lim_{z\rightarrow1}\exp\biggl(y_1\log\frac1{1-z}\biggr)\Lambda(z)&=&Z_1,\\
\Lim_{N\rightarrow\infty}
\exp\biggl(\Sum_{k\ge1}\H_{y_k}(N)\frac{(-y_1)^k}{k}\biggr)\H(N)&=&Z_2.
\end{eqnarray}

Moreover, $Z_1$ and $Z_2$ are equal and stand for the noncommutative generating series
of all convergent polyz\^etas $\{\zeta(w)\}_{w\in Y^*- y_1Y^*}$ as shown by the factorized form
indexed by Lyndon words  (recalled in Section \ref{ngs}).
This theorem enables, in particular, to explicit the counter-terms eliminating
the divergence of the polylogarithms $\{\Li_w(z)\}_{w\in y_1Y^*}$,
for $z\rightarrow1$, and of the harmonic sums $\{\H_w(N)\}_{w\in y_1Y^*}$,
for $N\rightarrow\infty$, and to calculate the Euler-Mac Laurin constants
associated to the divergent polyz\^etas $\{\zeta(w)\}_{w\in y_1Y^*}$ (see Corollary \ref{generalizedgamma}).
It allows also to give, in Section \ref{kerzeta} and via identification of locale coordinates
in infinite dimension, a {\it complete} description of the kernel by its generators, of the following
algebra homomorphism\footnote{Here, $\epsilon$ stands for the empty word over $Y$.}
\begin{eqnarray}
\zeta:(A\e\oplus(Y-y_1)\AY,\stuffle)&\longrightarrow&(\R,.)\\
y_{r_1}\ldots y_{r_k}&\longmapsto&\sum_{n_1>\ldots>n_k>0}\Frac1{n_1^{r_1}\ldots n_k^{r_k}},
\end{eqnarray}
and the set of {\it $A$-irreducible} polyz\^etas forming a transcendence basis of the image of $\zeta$, with
$A=\Q[{\mathrm i}\pi]$ (see Corollary \ref{ideal}).

Finally, via the {\it indiscernability} (recalled in Section \ref{sectionindiscernability}) over the group of associators, this study makes precise the structure of the $A$-algebra generated by the convergent polyz\^etas (see Theorem \ref{polyzetastructure}) and concludes the main challenge of the {\it polynomial} relations among polyz\^etas indexed by convergent Lyndon words which are algebraicly independant on the Euler constant and motivated \cite{SLC43,FPSAC97,bigotte,elwardi}.
In particular, the $A$-algebra generated by the convergent polyz\^etas was conjectured to be {\it free} \cite{SLC43,FPSAC97} and it will be proved, thanks to the propositions \ref{injectivite}, \ref{noyeaudephi} and  \ref{generator}. Moreover, this free $A$-algebra is {\it graded by weight}  meaning there is no {\it linear} relation among convergent polyz\^etas of different weight (see Theorem \ref{polyzetastructure}).

\section{Background~: structures and analytical studies of harmonic sums and of polylogarithms}

\subsection{Structures of harmonic sums and of polylogarithms}

\subsubsection{Quasi-symmetric functions and harmonic sums}\label{harmonic}

Let $\{t_i\}_{i\in\N_+}$ be an infinite set of variables.
The elementary symmetric functions $\eta_k$ and the power sums $\psi_k$ are defined by (see \cite{reutenauer})
\begin{eqnarray}
\eta_k(\underline t)=\sum_{n_1>\ldots>n_k>0}t_{n_1}\ldots t_{n_k}
&\mbox{and}&\psi_k(\underline t)=\sum_{n>0}t_n^k.
\end{eqnarray}
They are respectively coefficients of the following generating functions
\begin{eqnarray}
\eta(\underline t\bv z)=\prod_{i\ge1}(1+t_iz)
&\mbox{and}&
\psi(\underline t\bv z)=\sum_{i\ge1}\frac{t_iz}{1-t_iz}.
\end{eqnarray}
These generating functions satisfy a Newton identity
\begin{eqnarray}\label{eqNewton}
z\frac{d}{dz}\log\eta(\underline t\bv z)&=&\psi(\underline t\bv -z).
\end{eqnarray}
The fundamental theorem from symmetric functions theory asserts that
$\{\eta_k\}_{k\ge0}$ are linearly independent,
and provides remarkable identities like (with $\eta_0=1$)~:
\begin{eqnarray}\label{det}
\eta_k&=&\frac{(-1)^{k}}{k!}\sum_{\myover{s_1,\ldots,s_k\ge0}{s_1+\ldots+ks_k=k+1}}
{\mychoose{k}{s_1,\ldots,s_k}}
\biggl(-\frac{\psi_1}{1}\biggr)^{s_1}\ldots\biggl(-\frac{\psi_k}{k}\biggr)^{s_k}.\label{i1}
\end{eqnarray}

Let $Y$ be the infinite alphabet $\{y_i\}_{i\ge1}$ equipped with the order
$y_1>y_2>y_3>\ldots$ and let $\Lyn Y$ be the set of Lyndon words over $Y$.
The length of $w=y_{s_1}\ldots y_{s_r}\in Y^*$ is denoted by $\bv w\bv$ and its degree equals to $s_1+\ldots+s_r$.

The quasi-symmetric function $\F_w$,
of depth $r=\bv w\bv $ and of degree (or weight) $s_1+\ldots+s_r$, is defined by
\begin{eqnarray}\label{weight}
\F_w(\underline t)&=&\sum_{n_1>\ldots>n_r>0}t_{n_1}^{s_1}\ldots t_{n_r}^{s_r}.
\end{eqnarray}
In particular, $\F_{y_1^k}=\eta_k$ and $\F_{y_k}=\psi_k$.
The functions $\{\F_{y_1^k}\}_{k\ge0}$ are linearly independent
and integrating differential equation (\ref{eqNewton}) shows that functions
$\F_{y_1^k}$ and $\F_{y_k}$ are linked by the formula
\begin{eqnarray}\label{Newton}
\sum_{k\ge0}\F_{y_1^k}z^k&=&\exp\biggl(-\sum_{k\ge1}\F_{y_k}\frac{(-z)^k}{k}\biggr).
\end{eqnarray}

Every $\H_w(N)$ can be obtained by specializing, in the quasi-symmetric function $\F_w$,
the variables $\{t_i\}_{i\ge1}$ as follows \cite{hoffman}
\begin{eqnarray}
\forall N\ge i\ge1,t_i=1/i&\mbox{and}&\forall i>N,t_i=0.
\end{eqnarray}
In the same way, for $w\in Y^*-y_1Y^*$, the convergent polyz\^eta
$\zeta(w)$ can be obtained by specializing, in $\F_w$, the variables
$\{t_i\}_{i\ge1}$ as follows \cite{hoffman}
\begin{eqnarray}
\forall N\ge i\ge1,&&t_i=1/i.
\end{eqnarray}
The notation $\F_w$ is extended by linearity to all polynomials over $Y$.

If $u,v\in Y^*$, of length $r,s$ and of weight\footnote{The weight is as in Equation (\ref{weight}).}
$p,q$ respectively, $\F_{u\ministuffle v}$ is a quasi-symmetric function of depth $r+s$ and of weight $p+q$, and
$\F_{u\ministuffle v}=\F_u\;\F_v$,
where $\ministuffle$ is the quasi-shuffle product\footnote{See Annexe A, for the study of the Hopf algebra of $\stuffle$ which is not included in \cite{reutenauer}.} \cite{hoffman}. Hence,
\begin{eqnarray}
\forall u,v\in Y^*,&&\H_{u\ministuffle v}=\H_u\;\H_v,\\
\Rightarrow\quad
\forall u,v\in Y^*-y_1Y^*,&&\zeta(u\ministuffle v)=\zeta(u)\;\zeta(v).
\end{eqnarray}
Remarkable identity (\ref{i1}) can be then seen as
\begin{eqnarray}
y_1^k&=&\frac{(-1)^{k}}{k!}\sum_{\myover{s_1,\ldots,s_k\ge0}{s_1+\ldots+ks_k=k+1}}
{\genfrac{(}{)}{0pt}{}{k}{s_1,\ldots,s_k}}\frac{(-y_1)^{\ministuffle s_1}}{1^{s_1}}
\ministuffle\ldots\ministuffle\frac{(-y_k)^{\ministuffle s_k}}{k^{s_k}}.\label{i3}
\end{eqnarray}

\subsubsection{Iterated integrals and polylogarithms}\label{polylogarithm}

Let $X$ be the finite alphabet $\{x_0,x_1\}$ equipped with the order $x_0<x_1$\footnote{In all the sequel, we  follow the notations of \cite{berstel,reutenauer}.}. Let
\begin{eqnarray}
\calC:=\C\biggl[z,\Frac1z,\Frac1{1-z}\biggr]&\mbox{and}&\calG:=\left\{z,\Frac1z,\Frac{z-1}z,\Frac z{z-1},\Frac1{1-z},1-z\right\}.
\end{eqnarray}
This ring $\calC$ is invariant under differentiation and under the homographic transformations belonging to the group $\calG$ whose elements commute the singularities $\{0,1,+\infty\}$.

The iterated integral over $\omega_0,\omega_1$ associated to the word
$w=x_{i_1}\cdots x_{i_{k}}$ over $X^*$ (the monoid generated by $X$)
and along the integration path $z_0\path z$ is the following multiple integral defined by
\begin{eqnarray}\label{intint}
\int\limits_{z_0\path z}\omega_{i_1}\cdots\omega_{i_k}=
\int_{z_0}^{z}\omega_{i_1}(t_1)\int_{z_0}^{t_1}\omega_{i_2}(t_2)\ldots
\int_{z_0}^{t_{r-2}}\omega_{i_r}(t_{r-1})\int_{z_0}^{t_{r-1}}\omega_{i_r}(t_{r}),
\end{eqnarray}
where $t_1\cdots t_{r-1}$ is a subdivision of the path $z_0\path z$.
In a shortened notation, we denote this integral by
$\alpha_{z_0}^z(w)$ and\footnote{Here, $\e$ stands for the empty word over $X$.}
$\alpha_{z_0}^z(\e)=1$.
One can check that the polylogarithm $\Li_{s_1,\ldots,s_r}$
is also the value of the iterated integral over
$\omega_0,\omega_1$ and along the integration path $0\path z$ \cite{FPSAC95,FPSAC96}~:
\begin{eqnarray}
\Li_w(z)&=&\alpha_0^z(x_0^{s_1-1}x_1\ldots x_0^{s_r-1}x_1).
\end{eqnarray}
The definition of polylogarithms is extended over the words $w\in X^*$ by putting
$\Li_{x_0}(z):=\log z$.
The $\{\Li_w\}_{w\in X^*}$ are $\calC$-linearly independent \cite{FPSAC98,SLC43}.
The functions $\P_w(z):=(1-z)^{-1}\Li_w(z), w\in X^*,$ are also $\C$-linearly independent.
Since, for $w\in Y^*,\P_w$ is the ordinary generating function of
$\{\H_w(N)\}_{N\ge0}$ \cite{words03}~:
\begin{eqnarray}
\P_w(z)&=&\sum_{N\ge0}\H_w(N)\;z^N
\end{eqnarray}
then, as a consequence of the classical isomorphism between convergent Taylor series
and their associated sums, the harmonic sums $\{\H_w\}_{w\in Y^*}$ are also
$\C$-linearly independent. Firstly, $\ker\P=\{0\}$ and $\ker\H=\{0\}$,
and secondly, $\P$ is a morphism for the Hadamard product~:
\begin{eqnarray}
\qquad
\P_u(z)\odot\P_v(z)
=\sum_{N\ge0}\H_u(N)\H_v(N)z^N
=\sum_{N\ge0}\H_{u\ministuffle v}(N)z^N
=\P_{u\ministuffle v}(z).
\end{eqnarray}

\begin{proposition}[\cite{words03}]\label{isomorphisms}
Extended by linearity, the following maps are isomorphism of algebras
\begin{eqnarray*}
\P:(\C\pol{Y},\stuffle)&\longrightarrow&\left(\C\{\P_w\}_{w\in Y^*},\odot\right),\\
u&\longmapsto&\P_u,\\
\H:(\C\pol{Y},\stuffle)&\longrightarrow&\left(\C\{\H_w\}_{w\in Y^*},.\right),\\
u&\longmapsto&\H_u=\{\H_u(N)\}_{N \geq 0}.
\end{eqnarray*}
\end{proposition}

Studying the equivalence between action of $\{(1-z)^l\}_{l\in\Z}$
over $\{\P_w(z)\}_{w\in Y^*}$ and that of $\{N^k\}_{k\in\Z}$
over $\{\H_w(N)\}_{w\in Y^*}$ (see \cite{AofA}), we have

\begin{theorem}[\cite{cade}]
The Hadamard $\calC$-algebra of $\{\P_w\}_{w\in Y^*}$
can be identified with that of $\{\P_l\}_{l\in\Lyn Y}$.
In the same way, the algebra of harmonic sums $\{\H_w\}_{w\in Y^*}$
with polynomial coefficients can be identified with that of
$\{\H_l\}_{l\in\Lyn Y}$.
\end{theorem}

By Identity (\ref{i3}) and by applying the isomorphism $\H$
on the set of Lyndon words $\{y_r\}_{1\le r\le k}$, we obtain $\H_{y_1^k}$
as polynomials in $\{\H_{y_r}\}_{1\le r\le k}$ (which are algebraically independent), and
\begin{eqnarray}
\H_{y_1^k}&=&\sum_{\myover{s_1,\ldots,s_k\ge0}{s_1+\ldots+ks_k=k+1}}\frac{(-1)^k}{s_1!\ldots s_k!}
\biggl(-\frac{\H_{y_1}}{1}\biggr)^{s_1}\ldots\biggl(-\frac{\H_{y_k}}{k}\biggr)^{s_k}.
\end{eqnarray}

\subsection{Results {\em \`a la Abel} for generating series
of harmonic sums and of polylogarithms}\label{ngs}

\subsubsection{Generating series of harmonic sums and of polylogarithms}\label{generatingseries}

Let $\H(N)$ be the noncommutative generating series of $\{\H_w(N)\}_{w\in Y^*}$ \cite{words03}~:
\begin{eqnarray}
\H(N)&:=&\sum_{w\in Y^*}\H_w(N)\;w.
\end{eqnarray}

Let $\{\Sigma_w\}_{w\in Y^*},\{\check\Sigma_w\}_{w\in Y^*}$ be respectively a PBW basis of the envelopping algebra $\calU(\LQY)$ and the quasi-shuffle algebra $(\QY,\stuffle)$ (viewed as a $\Q$-module) on duality such that $\{\Sigma_l\}_{l\in\Lyn X},\{\check\Sigma_l\}_{l\in\Lyn X}$ are respectively a basis of $\LQY$ and a transcendence basis of $(\QY,\stuffle)$ (see Annexe A).

\begin{theorem}[Factorization of $\H$]\label{Hfactorization}
Let
\begin{eqnarray*}
\H_{\reg}(N)&:=&\Prod_{l\in\Lyn Y-\{y_1\}}^{\searrow}e^{\H_{\check\Sigma_l}(N)\;\Sigma_l}.
\end{eqnarray*}
Then\ $\H(N)=e^{\H_{y_1}(N)\;y_1}\H_{\reg}(N)$.
\end{theorem}

\begin{proof}
See Annexe A.
\end{proof}

For $l\in\Lyn Y-\{y_1\}$, the polynomial $\Sigma_l$ is a finite combination of words in $Y^*-y_1Y^*$. Then we can state the following
\begin{definition}\label{psikz}
We set $Z_{\ministuffle}:=\H_{\reg}(\infty)$.
\end{definition}

The noncommutative generating series of polylogarithms \cite{FPSAC98,SLC43}
\begin{eqnarray}
\L&:=&\sum_{w\in X^*}\Li_{w}\;w
\end{eqnarray}
satisfies Drinfel'd's differential equation $(DE)$ of Example \ref{KZ}
\begin{eqnarray}\label{drinfeld}
d\L&=&(x_0\omega_0+x_1\omega_1)\L
\end{eqnarray}
with boundary condition  \cite{drinfeld1,drinfeld2}
\begin{eqnarray}\label{boundary}
\L(\eps)&{}_{\widetilde{\eps\to0^+}}&e^{x_0\log\eps}.
\end{eqnarray}
This enables us to prove that $\L$ is the exponential of a Lie series\footnote{{\it i.e.}, $\L$  is goup-like for the co-product $\Delta_{\minishuffle}$~: $\Delta_{\minishuffle}(\L)=\L\otimes\L$.}
\cite{FPSAC98,SLC43}. Hence,
\begin{proposition}[Logarithm of $\L$, \cite{orlando}]\label{logL}
Let $\pi_1(w)$ is the following Lie series
\begin{eqnarray*}
\pi_1(w)&=&\sum_{k\ge1}\frac{(-1)^{k-1}}k\sum_{u_1,\ldots,u_k\in X^+}
\langle w\bv u_1\shuffle\ldots\shuffle u_k\rangle\;u_1\ldots u_k.
\end{eqnarray*}
Then
\begin{eqnarray*}
\log\L(z)
&=&\sum_{k\ge1}\frac{(-1)^{k-1}}k\sum_{u_1,\ldots,u_k\in X^+}
\Li_{u_1\shuffle\ldots\shuffle u_k}(z)\;u_1\ldots u_k\\
&=&\sum_{w\in X^*}\Li_w(z)\;\pi_1(w).
\end{eqnarray*} 
\end{proposition}
Applying a theorem of Ree \cite{ree,reutenauer}, $\L$ satisfies Friedrichs
criterion \cite{FPSAC98,SLC43}~:
\begin{eqnarray}\label{shuffle}
\forall u,v\in X^*,&&\Li_{u\minishuffle v}=\Li_u\;\Li_v,\\
\Rightarrow\quad
\forall u,v\in x_0X^*x_1,&&\zeta(u\minishuffle v)=\zeta(u)\;\zeta(v).
\end{eqnarray}

\begin{proposition}[Successive differentiation of $\L$, \cite{orlando}]\label{lem:derivL}
For any $l\in\N$, let
\begin{eqnarray*}
P_l(z)=\sum_{\wgt({\bf r})=l}\sum_{w\in X^{\deg({\bf r})}}\prod_{i=1}^{\deg({\bf r})}
\binom{\sum_{j=1}^ir_i+j-1}{r_i}\tau_{\bf r}(w)\in\calC\langle X\rangle,
\end{eqnarray*}
where, for any $w=x_{i_1}\cdots x_{i_k}$ and ${\bf r}=(r_1,\ldots,r_k)$ of degree $\deg({\bf r})=k$
and of weight $\wgt({\bf r})=k+r_1+\cdots +r_k$, the polynomial
$\tau_{\bf r}(w)=\tau_{r_1}(x_{i_1})\cdots\tau_{r_k}(x_{i_k})$ is defined by
\begin{eqnarray*}
\forall r\in\N,\quad
\tau_r(x_0)=\partial^r\frac{x_0}{z}=\frac{-r!x_0}{(-z)^{r+1}}&\hbox{and}&  
\tau_r(x_1)=\partial^r\frac{x_1}{1-z}=\frac{r!x_1}{(1-z)^{r+1}}.
\end{eqnarray*}
Denoting $\partial=d/dz$, we have $\partial^l\L(z)=P_l(z)\L(z)$.
\end{proposition}

Let $\{\check S_l\}_{l\in\Lyn X}$ be the transcendence basis of the shuffle algebra $(\QX,\shuffle)$ and $\{\check S_w\}_{w\in X^*}$ be the associated completed basis of  the shuffle algebra $(\QX,\shuffle)$ (viewed as a $\Q$-module). They are defined as follows \cite{reutenauer}
\begin{eqnarray}
\check S_{1_{X^*}}=&1&\mbox{for }l=1_{X^*}\\
\check S_l=&x\check S_u,&\mbox{for }l=xu\in\Lyn X,\label{basedeSchzut}\\
\check S_w=&
\Frac{\check S_{l_1}^{\shuffle i_1}\shuffle\ldots\shuffle\check S_{l_k}^{\shuffle i_k}}{i_1!\ldots i_k!}
&\mbox{for }w=l_1^{i_1}\ldots l_k^{i_k},l_1>\ldots>l_k.
\end{eqnarray}
Let  $\{S_w\}_{w\in Y^*}$ be the PBW basis of the envelopping algebra $\calU(\LQX)$
in duality with the basis $\{\check S_w\}_{w\in Y^*}$ and $\{S_l\}_{l\in\Lyn X}$ is then the basis of the Lie algebra $\LQX$ \cite{reutenauer}.

\begin{theorem}[Factorization of $\L$, \cite{FPSAC98,SLC43}]\label{factorisationL}
Let
\begin{eqnarray*}
\L_{\reg}&:=&\Prod_{l\in\Lyn X-X}^{\searrow}e^{\Li_{S_l}\check S_l}.
\end{eqnarray*}
Then $\L(z)=e^{-x_1\log(1-z)}\L_{\reg}(z)e^{x_0\log z}$.
\end{theorem}
For $l\in\Lyn X-X$, the polynomial $S_l$ is a finite combination of words in $x_0X^*x_1$.
Then we can state the following

\begin{definition}[\cite{FPSAC98,SLC43}]\label{phikz}
We set $Z_{\minishuffle}:=\L_{\reg}(1)$.
\end{definition}
In the definitions \ref{psikz} and \ref{phikz} only {\it convergent} polyz\^etas arise
and these noncommutative generating series will induce, in Section \ref{regularization},
two algebra morphisms of regularization as shown in the theorems \ref{reg1} and \ref{reg2} respectively.
Hence, these power series are quite different of those given in \cite{lemurakami}
or in \cite{racinet} (the last is based on \cite{boutet},
see \cite{cartier2}) needing a regularization process.

\subsubsection{Asymptotic expansions by noncommutative generating series
and regularized Chen generating series}\label{Asymptotic}

Let $\rho_{1-z},\rho_{1-\frac1z}$ and $\rho_{\frac1z}$ \cite{FPSAC99,SLC43}
be three monoid morphisms verifying
\begin{eqnarray}
\rho_{1-z}(x_0)=-x_1&\mbox{and}&\rho_{1-z}(x_1)=-x_0,\\
\rho_{1-1/z}(x_0)=-x_0+x_1&\mbox{and}&\rho_{1-1/z}(x_1)=-x_0\\
\rho_{1/z}(x_0)=-x_0+x_1&\mbox{and}&\rho_{1/z}(x_1)=x_1.
\end{eqnarray}
Using homographic transformations belonging to the group $\calG$, one has \cite{FPSAC99,SLC43}
\begin{eqnarray}
\L(1-z)&=&e^{x_0\log(1-z)}\rho_{1-z}[\L_{\reg}(z)]e^{-x_1\log z}Z_{\minishuffle},\label{sigma1}\\
\L(1-1/z)&=&e^{x_0\log(1-z)}\rho_{1-\frac1z}[\L_{\reg}(z)]e^{-x_1\log z}
\rho_{1-1/z}(Z_{\minishuffle}^{-1})e^{\mathrm{i}\pi x_0}\label{sigma2}\\
\L(1/z)&=&e^{-x_1\log(1-z)}\rho_{1/z}[\L_{\reg}(z)]e^{(-x_0+x_1)\log z}
\rho_{1/z}(Z_{\minishuffle}^{-1})e^{\mathrm{i}\pi x_1}Z_{\minishuffle}\label{sigma3}.
\end{eqnarray}
Thus, (\ref{boundary}) and (\ref{sigma1}) yield \cite{FPSAC99,SLC43}
\begin{eqnarray}\label{asymptoticbehaviour}
\L(z)\;{}_{\widetilde{z\rightarrow0}}\;\exp(x_0\log z)&\mbox{and}&
\L(z)\;{}_{\widetilde{z\rightarrow1}}\;\exp(-x_1\log(1-z))\;Z_{\minishuffle}.
\end{eqnarray}

Let us call $\LI_{\calC}$ the smallest algebra containing $\calC$,
closed under derivation and under integration with respect to $\omega_0$ and $\omega_1$.
It is the $\calC$-module generated by the polylogarithms $\{\Li_w\}_{w\in X^*}$.

Let $\pi_Y:\LI_{\calC}\langle\!\langle X\rangle\!\rangle
\longrightarrow\LI_{\calC}\langle\!\langle Y\rangle\!\rangle$ be a projector such that
for any $f\in\LI_{\calC}$ and $w\in X^*$, $\pi_Y(f\;wx_0)=0$. Then \cite{cade}
\begin{eqnarray}\label{piYL}
\Lambda(z)=\pi_Y\L(z)
{}_{\widetilde{z\rightarrow1}}\exp\biggl(y_1\log\frac1{1-z}\biggr)\pi_YZ_{\minishuffle}.
\end{eqnarray}

Since the coefficient of $z^N$ in the ordinary Taylor expansion
of $\P_{y_1^k}$ is $\H_{y_1^k}(N)$ then let
\begin{eqnarray}
&&\mathrm{Mono}(z):=e^{-(x_1+1)\log(1-z)}
=\Sum_{k\ge0}\P_{y_1^k}(z)\;y_1^k\label{Mono}\\
&&\mathrm{Const}:=\sum_{k\ge0}\H_{y_1^k}\;y_1^k
=\exp\biggl(-\Sum_{k\ge1}\H_{y_k}\Frac{(-y_1)^k}{k}\biggr)\label{Const}.
\end{eqnarray}

\begin{proposition}[\cite{cade}]\label{C3}
We have
\begin{eqnarray*}
\pi_Y\P(z)\;{}_{\widetilde{z\rightarrow1}}\;\mathrm{Mono}(z)\pi_YZ_{\minishuffle}
&\mbox{and}&
\H(N)\;{}_{\widetilde{N\rightarrow\infty}}\;\mathrm{Const}(N)\pi_YZ_{\minishuffle}.
\end{eqnarray*}
\end{proposition}

\begin{proof}
Let $\mu$ be the morphism verifying $\mu(x_0)=x_1$ and $\mu(x_1)=x_0$.
Then, by Theorem \ref{factorisationL}, the noncommutative generating series of
$\{\P_w\}_{w\in X^*}$ is given by
\begin{eqnarray*}\label{factorisationP}
\P(z)&=&(1-z)^{-1}\L(z)\\
&=&e^{-(x_1+1)\log(1-z)}\L_{\reg}(z)e^{x_0\log z}\\
&=&e^{x_0\log z}\mu[\L_{\reg}(1-z)]e^{-(x_1+1)\log(1-z)}Z_{\minishuffle}\\
&=&e^{x_0\log z}\mu[\L_{\reg}(1-z)]\mathrm{Mono}(z)Z_{\minishuffle}.
\end{eqnarray*}
Thus,
$\P(z){}_{\widetilde{z\rightarrow0}}e^{x_0\log z}$ and
$\P(z){}_{\widetilde{z\rightarrow1}}\mathrm{Mono}(z)Z_{\minishuffle}$
yielding the expected results.
\end{proof}

As consequence of (\ref{piYL})-(\ref{Const}) and of Proposition \ref{C3}, one gets

\begin{theorem}[{\em\`a la Abel}, \cite{cade}]
$$\Lim_{z\rightarrow1}\exp\biggl(y_1\log\frac1{1-z}\biggr)\Lambda(z)=
\Lim_{N\rightarrow\infty}\exp\biggl(\Sum_{k\ge1}\H_{y_k}(N)\frac{(-y_1)^k}{k}\biggr)\H(N)=\pi_YZ_{\minishuffle}.$$
\end{theorem}

Therefore, the knowledge of the ordinary Taylor expansion at $0$ of the polylogarithmic functions $\{\P_w(1-z)\}_{w\in X^*}$ gives

\begin{theorem}[\cite{AofA}]
For all $g\in\calC\{\P_w\}_{w \in Y^*}$,
there exists algorithmically computable $c_j\in\C,\alpha_j\in\Z,\beta_j\in\N$
and $b_i\in\C,\eta_i\in\Z,\kappa_i\in\N$ such that
\begin{eqnarray*}
g(z)\;{}_{\widetilde{z\rightarrow1}}\;\Sum_{j=0}^{+\infty}c_j(1-z)^{\alpha_j}\log^{\beta_j}(1-z)
&\mbox{and}&
[z^n]g(z)\;{}_{\widetilde{N\rightarrow+\infty}}\;\Sum_{i=0}^{+\infty}b_in^{\eta_i}\log^{\kappa_i}(n).
\end{eqnarray*}
\end{theorem}

\begin{definition}\label{calZ}
Let $\calZ$ be the $\Q$-algebra generated by convergent polyz\^etas
and let $\calZ'$ be the\footnote{Here, $\gamma$ stands for the Euler constant
$\gamma=.57721 56649 01532 86060 65120 90082 40243 \ldots$}
$\Q[\gamma]$-algebra generated by $\calZ$.
\end{definition}

\begin{corollary}[\cite{AofA}]\label{asymptotic}
There exists algorithmically computable $c_j\in\calZ,\alpha_j\in\Z,\beta_j\in\N$
and $b_i\in\calZ',\kappa_i\in\N,\eta_i\in\Z$ such that
$$\begin{array}{rrlllll}
&\forall w\in Y^*,&\P_w(z)&\sim&\Sum_{j=0}^{+\infty}c_j(1-z)^{\alpha_j}\log^{\beta_j}(1-z)&\mbox{for}&z\rightarrow1,\\
&\forall w\in Y^*,&\H_w(N)&\sim&\Sum_{i=0}^{+\infty}b_iN^{\eta_i}\log^{\kappa_i}(N)&\mbox{for}&N\rightarrow+\infty.
\end{array}$$
\end{corollary}

The Chen generating series along the path $z_0\path z$, associated to $\omega_0,\omega_1$ is the following
\begin{eqnarray}\label{chen}
S_{z_0\path z}:=\sum_{w\in X^*}\langle S\bv w\rangle\;w&\mbox{with}&\langle S\bv w\rangle=\alpha_{z_0}^z(w)
\end{eqnarray}
which solves the differential equation (\ref{drinfeld})
with the initial condition $S_{z_0\path z_0}=1$.
Thus, $S_{z_0\path z}$ and $\L(z)\L(z_0)^{-1}$ satisfy the same differential equation taking the same value at $z_0$ and
\begin{eqnarray}\label{chenpolylog}
S_{z_0\path z}&=&\L(z)\L(z_0)^{-1}.
\end{eqnarray}

Any Chen generating series $S_{z_0\path z}$ is group like \cite{ree} and depends only on the homotopy class of $z_0\path z$ \cite{chen}.
The product of $S_{z_1\path z_2}$ and $S_{z_0\path z_1}$ is the Chen generating series
\begin{eqnarray}\label{concate}
S_{z_0\path z_2}&=&S_{z_1\path z_2}S_{z_0\path z_1}.
\end{eqnarray}

Let $\eps\in]0,1[$ and let $z_i=\eps\exp({\mathrm i}\theta_i)$, for $i=0$ or $1$. We set $\theta=\theta_1-\theta_0$. Let
$\Gamma_0(\eps,\theta)$ (resp. $\Gamma_1(\eps,\theta)$) be the path turning around $0$ (resp. $1$) in the positive direction from $z_0$ to $z_1$. By induction on the length of $w$, one has
\begin{eqnarray}
\abs{\langle S_{\Gamma_i(\eps,\theta)}\bv w\rangle}&=&
 (2\eps)^{\abs{w}_{x_i}}{\theta^{\abs{w}|}}/{\abs{w}!},
\end{eqnarray}
where, $\abs{w}$ denotes the lenghth of $w$ and $\abs{w}_{x_i}$, denotes the number of occurrences of letter $x_i$ in $w$, for $i=0,1$.
For $\eps\to0^+$, these estimations yield
\begin{eqnarray}
S_{\Gamma_i(\eps,\theta)}&=&e^{{\mathrm i}\theta x_i}+o(\eps).
\end{eqnarray}
In particular, if $\Gamma_0(\eps)$ (resp. $\Gamma_1(\eps)$) is a
circular path of radius $\eps$ turning around $0$ (resp. $1$) in
the positive direction, starting at $z=\eps$ (resp. $1-\eps$), then, by the
noncommutative residu theorem \cite{FPSAC98,SLC43}, we get
\begin{eqnarray}\label{majoration}
S_{\Gamma_0(\eps)}=e^{2{\mathrm i}\pi x_0}+o(\eps)
&\mbox{and}&
S_{\Gamma_1(\eps)}=e^{-2{\mathrm i}\pi x_1}+o(\eps).
\end{eqnarray}
Finally, the asymptotic behaviors of $\L$ on (\ref{asymptoticbehaviour}) give
\begin{proposition}[\cite{SLC43,FPSAC98}]\label{chenregularization}
We have $S_{\eps\path1-\eps}\;{}_{\widetilde{\eps\rightarrow0^+}}\;e^{-x_1\log\eps}Z_{\minishuffle}\;e^{-x_0\log\eps}$.
\end{proposition}
In other terms, $Z_{\minishuffle}$ is the regularized Chen generating series
$S_{\eps\path1-\eps}$ of diffferential forms $\omega_0$ and $\omega_1$~:
$Z_{\minishuffle}$ is the noncommutative generating series of the finite parts of the
coefficients of the Chen generating series $e^{x_1\log\eps}\;S_{\eps\path1-\eps}\;e^{x_0\log\eps}$~:
the concatenation of $e^{x_0\log\eps}$ and then $S_{\eps\path1-\eps}$ and finally, $e^{x_1\log\eps}$.

\begin{proposition}\label{atomiser}
Let $\rho_{1-1/z}$ be the morphism is given in Section \ref{Asymptotic}. We have
\begin{eqnarray*}
\Prod_{l\in\Lyn X\atop l\neq{x_0,x_1}}^{\searrow}e^{\zeta(\check l)l}
&=&e^{\mathrm{i}\pi x_0}\Prod_{l\in\Lyn X\atop l\neq{x_0,x_1}}^{\searrow}
e^{\zeta(\check l)\rho_{1-1/z}(l)}
e^{\mathrm{i}\pi(-x_0+x_1)}\Prod_{l\in\Lyn X\atop l\neq{x_0,x_1}}^{\searrow}
e^{\zeta(\check l)\rho_{1-1/z}^2(l)}e^{-\mathrm{i}\pi x_1}.
\end{eqnarray*}
\end{proposition}

\begin{proof}
Following the hexagonal path given
in Figure \ref{chemin3}, one has \cite{FPSAC99,SLC43}
\begin{eqnarray*} \label{hexa1}
  (S_{\eps \path 1-\eps}e^{\mathrm{i}\pi x_0})
  \rho_{1-1/z}(S_{\eps \path 1-\eps}e^{\mathrm{i}\pi x_0}) 
  \rho_{1-1/z}^2(S_{\eps \path 1-\eps}e^{\mathrm{i}\pi x_0})
  &=&1+O(\sqrt{\eps}).
\end{eqnarray*}
\begin{figure}
   \centering
   \input{chemin3.epic}
   \caption{Hexagonal path}
   \label{chemin3}
\end{figure}
By Proposition \ref{chenregularization}, it follows the hexagonal relation
\cite{drinfeld1,drinfeld2,FPSAC99,SLC43} which is
\begin{eqnarray*}
&&Z_{\minishuffle}e^{\mathrm{i}\pi x_0}\rho_{1-1/z}(Z_{\minishuffle})
e^{\mathrm{i}\pi(-x_0+x_1)}\rho_{1-1/z}^2(Z_{\minishuffle})e^{-\mathrm{i}\pi x_1}=1,\\
&\iff&
e^{\mathrm{i}\pi x_0}\rho_{1-1/z}(Z_{\minishuffle})
e^{\mathrm{i}\pi(-x_0+x_1)}
\rho_{1-1/z}^2(Z_{\minishuffle})e^{-\mathrm{i}\pi x_1}
=Z_{\minishuffle}^{-1}.
\end{eqnarray*}
It follows then the expected result.
\end{proof}

\subsection{Indiscernability over  a class of formal power series}\label{sectionindiscernability}

\subsubsection{Residual calculus and representative series}
\begin{definition}
Let $S\in\QXX$ and let $P\in\QX$.

The {\em left residual} (resp. {\em right residual}) of $S$ by $P$, is the
formal power series $P\resg S$ (resp. $S\resd P$) in $\QXX$
defined by~:
\begin{eqnarray*}
   \pol{P\resg S\bv w}=\pol{S\bv wP}&(resp.&\pol{S\resd P\bv w}=\pol{S\bv Pw}).
\end{eqnarray*}
\end{definition}

We straightforwardly get, for any $P,Q\in\QX$~:
\begin{eqnarray}
P\resg (Q\resg S)=PQ\resg S,\
(S\resd P)\resd Q=S\resd PQ,\
(P\resg S)\resd Q=P\resg(S\resd Q).
\end{eqnarray}
In case $x,y\in X$ and $w\in X^* $, we get
$x\resg (wy)=\delta_{x,y}w$ and $xw\resd y=\delta_{x,y}w$.

\begin{lemma}[Reconstruction lemma]
Let $S\in\QXX$. Then
\begin{eqnarray*}
S=\pol{S\bv\epsilon}+\sum_{x\in X}x(S\resd x)
=\pol{S\bv\epsilon}+\sum_{x\in X}(x\resg S)x.
\end{eqnarray*}
\end{lemma}

\begin{lemma}
The left and right residuals by a letter $x$ are
derivations in $(\QXX,\shuffle):$
\begin{eqnarray*}
x\resg(u\shuffle v)=(x\resg u)\shuffle v+u\shuffle(x\resg v),&&
(u\shuffle v)\resd x=(u\resd x)\shuffle v+u\shuffle(v\resd x).
\end{eqnarray*}
\end{lemma}
\begin{proof}
Use the recursive definitions of the shuffle product.
\end{proof}

\begin{lemma}\label{residuals}
For any Lie polynomial $Q\in\LQX$, the linear maps
``$Q\resg$'' and ``$\resd Q$'' are derivations on $(\Q[\Lyn X],\shuffle)$.
\end{lemma}
\begin{proof}
For any $l,l_1,l_2\in\Lyn X$, we have
\begin{eqnarray*}
\hat l\triangleleft (l_1\shuffle l_2)
=l_1\shuffle(\hat l\triangleleft l_2)+(\hat l\triangleleft l_1)\shuffle l_2
=l_1\delta_{l_2,\hat l}+\delta_{l_1,\hat l}l_2,\\
(l_1\shuffle l_2)\triangleright\hat l
=l_1\shuffle(l_2\triangleright\hat l)+(l_1\triangleright\hat l)\shuffle l_2
=l_1\delta_{l_2,\hat l}+\delta_{l_1,\hat l}l_2.
\end{eqnarray*}
\end{proof}

\begin{lemma}\label{reslettre}
For any Lyndon word $l\in\Lyn X$ and $\check S_l$ defined as in (\ref{basedeSchzut}), one has
\begin{eqnarray*}
x_1\triangleleft l=l\triangleright x_0=0&\mbox{and}&
x_1\triangleleft\check S_l=\check S_l\triangleright x_0=0.
\end{eqnarray*}
\end{lemma}
\begin{proof}
Since $x_1\triangleleft$ and $\triangleright x_0$ are derivations and for any
$l\in\Lyn X-X$, the polynomial $\check S_l$ belongs to $x_0\QX x_1$ then
it follows the expected results.
\end{proof}

\begin{theorem}[On representative series]\label{representativeseries}
The following properties are equivalent for any series $S\in\QXX$~:
\begin{enumerate}
\item The left $\C$-module
$Res_{g}(S)=\text{span}\{w\resg S\bv w\in X^*\}$ is finite
dimensional.
\item The right $\C$-module
$Res_{d}(S)=\text{span}\{S\resd w\bv w\in X^*\}$ is finite
dimensional.
\item There are matrices $\lambda\in \llm M_{1,n}(\Q)$,
$\eta\in\llm M_{n,1}(\Q)$ and a representation of $X^*$ in $\llm
M_{n,n}$, such that
\begin{eqnarray*}
S=\sum_{w\in X^*}[\lambda\mu(w)\eta]\;w
=\lambda\biggl(\prod_{l\in\Lyn X}^{\searrow}e^{\mu(S_l)\;\check S_l}\biggr)\eta.
\end{eqnarray*}
\end{enumerate}
\end{theorem}

A series that safisfies the items of this theorem will be called
{\em representative series}. This concept
can be found in \cite{abe,hochschild,DT2}. The two first items are in
\cite{fliess0,hespel}. The third item can be deduced from \cite{ChariPressley,Duchamp}
for example and it was used to factorize first time, by Lyndon words,
the output of bilinear and analytical dynamical systems respectively in
\cite{IMACS0,hoangjacoboussous} and to study polylogarithms, hypergeometric functions
and associated functions in \cite{FPSAC95,FPSAC96,orlando}.
The dimension of $Res_g(S)$ is equal to that of $Res_d(S)$, and to the minimal dimension
of a representation satisfying the third point of
Theorem \ref{representativeseries}. This rank is then equal to the rank of
the Hankel matrix of $S$, that is the infinite matrix
$(\langle S\bv uv\rangle)_{u,v\in X}$ indexed by $X^*\times X^*$ and
is also called {\it Hankel rank} of $S$ \cite{fliess0,hespel}~:
\begin{definition}[\cite{fliess0,hespel}]
The {\em Hankel rank} of a formal power series $S\in\CXX$ is the dimension of the vector space
\begin{eqnarray*}
\{S\resd\Pi\bv\Pi\in\CX\},&(\mbox{resp.}&\{\Pi\resg S\bv\Pi\in\CX\}.
\end{eqnarray*}
\end{definition}
The triplet $(\lambda,\mu,\eta)$ is called a {\em linear representation} of $S$.
We define the minimal representation\footnote{It can be shown that
all minimal representations are isomorphics (see \cite{berstel}).} of $S$ as being a representation of
$S$ of minimal dimension.

For any proper series $S$, the following power series is called ``star of $S$''
\begin{eqnarray}
S^*&=&1+S+S^2+\ldots +S^n+\ldots.
\end{eqnarray}

\begin{definition}[\cite{berstel,schutz}]
A series $S$ is called {\em rational} if it belongs to the closure
in $\QXX$ of the noncommutative polynomial algebra by sum,
product and star operation of {\em proper}\footnote{A series $S$
is said to be proper if $\langle S\bv\epsilon\rangle=0$.} elements.
The set of rational power series will be denoted by $\QratXX$.
\end{definition}

\begin{lemma}
For any noncommutative rational series (resp. polynomial) $R$ and for any polynomial $P$,
the left and right residuals of $R$ by $P$ are rational (resp. polynomial).
\end{lemma}

\begin{theorem}[Sch\"utzenberger, \cite{berstel,schutz}]
Any noncommutative power series is representative if and only if it is rational.
\end{theorem}

\subsubsection{Continuity and indiscernability}

\begin{definition}[\cite{these,cade}]\label{indiscernability}
Let $\calH$ be a class of formal power series over $X$ and let $S\in\CXX$.
\begin{enumerate}
\item $S$ is said to be {\em continuous}\footnote{See \cite{these,cade},
for a convergence criterion and an example of continuous generating series.}
over $\calH$ if for any $\Phi\in\calH$, the following sum, denoted $\langle S\bbv\Phi\rangle$,  is convergent in norm
\begin{eqnarray*}
\sum_{w\in X^*}\langle S\bv w\rangle\langle\Phi\bv w\rangle.
\end{eqnarray*}
The set of continuous power series over $\calH$ will be denoted by $\CcXX$.
\item $S$ is said to be {\em indiscernable}\footnote{Here, we adapt this notion
developped in \cite{these} via the residual calculus.} over $\calH$ if and only if
\begin{eqnarray*}
\forall \Phi\in \calH,&&\langle S\bbv\Phi\rangle=0.
\end{eqnarray*}
\end{enumerate}
\end{definition}

Let $\rho$ be the monoid morphism verifying $\rho(x_0)=x_1$ and $\rho(x_1)=x_0$
and  let $\hat w=\rho(\tilde w)$, where $\tilde w$ is the mirror of $w$.
\begin{lemma}\label{antipode}
Let $S\in\CcXX$. If $\langle S\bbv Z_{\minishuffle}\rangle=0$
then $\langle\hat S\bbv Z_{\minishuffle}\rangle=0$, where
\begin{eqnarray*}
\hat S&:=&\Sum_{w\in X^*}\langle S\bv w\rangle\;\hat w.
\end{eqnarray*}\end{lemma}
\begin{proof}
For any $w\in x_0X^*x_1$, by ``duality relation",  one has (see \cite{hoffman0,zagier,FPSAC99})
\begin{eqnarray*}
\zeta(\hat w)=\zeta(w),
&\mbox{or equivalently}&Z_{\minishuffle}=\hat Z_{\minishuffle}
:=\sum_{w\in X^*}\langle Z_{\minishuffle}\bv w\rangle\;\hat w.
\end{eqnarray*}
Using the fact
\begin{eqnarray*}
\langle\hat S\bbv Z_{\minishuffle}\rangle
=\Sum_{\hat w\in X^*}\langle S\bv\hat w\rangle\langle Z_{\minishuffle}\bv\hat w\rangle
=\Sum_{w\in X^*}\langle S\bv w\rangle\langle Z_{\minishuffle}\bv w\rangle,
\end{eqnarray*}
one gets finally the expected result.
\end{proof}

\begin{lemma}\label{indiscernablelemma}
Let $\calH$ be a monoid containing  $\{e^{t\;x}\}_{x\in X}^{t\in\C}$.
Let $S\in\CcXX$ being indiscernable over $\calH$.
Then for any $x\in X$, $x\triangleleft S$ and $S\triangleright x$
belong to $\CcXX$ and they are indiscernable over $\calH$.
\end{lemma}
\begin{proof}
Let us calculate
$\langle x\triangleleft S\bbv\Phi\rangle=\langle S\bbv\Phi x\rangle$
and $\langle S\triangleright x\bbv\Phi\rangle =\langle S\bbv x\Phi\rangle$.
Since
\begin{eqnarray*}
\lim_{t\rightarrow0}\frac{e^{t\;x}-1}{t}=x &\mbox{and}&
\lim_{t\rightarrow0}\frac{e^{t\;x}-1}{t}=x
\end{eqnarray*}
then, for any $\Phi\in\calH$, by uniform convergence, one has
\begin{eqnarray*}
\langle S\bbv\Phi x\rangle =\langle
S\bbv\lim_{t\rightarrow0}\Phi\frac{e^{t\;x}-1}{t}\rangle
=\lim_{t\rightarrow0}\langle S\bbv\Phi\frac{e^{t\;x}-1}{t}\rangle,\\
\langle S\bbv x\Phi\rangle =\langle
S\bbv\lim_{t\rightarrow0}\frac{e^{t\;x}-1}{t}\Phi\rangle
=\lim_{t\rightarrow0}\langle S\bbv\frac{e^{t\;x}-1}{t}\Phi\rangle.
\end{eqnarray*}
Since $S$ is indiscernable over $\calH$ then
\begin{eqnarray*}
\langle S\bbv\Phi x\rangle =\lim_{t\rightarrow0}\frac1t\langle
S\bbv\Phi e^{t\;x}\rangle -\lim_{t\rightarrow0}\frac1t\langle
S\bbv\Phi\rangle
=0,\\
\langle S\bbv x\Phi\rangle =\lim_{t\rightarrow0}\frac1t\langle S\bbv
e^{t\;x}\Phi\rangle -\lim_{t\rightarrow0}\frac1t\langle
S\bbv\Phi\rangle =0.
\end{eqnarray*}
\end{proof}

\begin{proposition}\label{nulle}
Let $\calH$ be a monoid containing  $\{e^{t\;x}\}_{x\in X}^{t\in\C}$.
The formal power series $S\in\CcXX$ is indiscernable over $\calH$
if and only if $S=0$.
\end{proposition}
\begin{proof}
If $S=0$ then it is immediate that $S$ is indiscernable over
$\calH$. Conversely, if $S$ is indiscernable over $\calH$ then by
Lemma \ref{indiscernablelemma}, for any word $w\in X^*$, by
induction on the length of $w$, $w\triangleleft S$ is indiscernable
over $\calH$ and then in particular,
\begin{eqnarray*}
\langle w\triangleleft S\bbv\mathrm{Id}_{\calH}\rangle=\langle S\bv w\rangle=0.
\end{eqnarray*}
In other words, $S=0$.
\end{proof}

\section{Group of associators~: polynomial relations among convergent polyz\^etas and  identification of local coordinates}
\subsection{Generalized Euler constants and global regularization of polyz\^etas}\label{regularization}

\subsubsection{Three regularizations of divergent polyz\^etas}

\begin{theorem}[\cite{SLC44}]\label{reg1}
Let $\zeta_{\ministuffle}:(\QY,\stuffle)\rightarrow(\R,.)$ be the morphism verifying the following properties
\begin{itemize}
\item for $u,v\in Y^*,\zeta_{\ministuffle}(u\ministuffle v)
=\zeta_{\ministuffle}(u)\zeta_{\ministuffle}(v)$,
\item for all convergent word $w\in Y^*- y_1Y^*,\zeta_{\ministuffle}(w)=\zeta(w)$,
\item $\zeta_{\ministuffle}(y_1)=0$.
\end{itemize}
Then
\begin{eqnarray*}
\sum_{w\in X^*}\zeta_{\ministuffle}(w)\;w&=&Z_{\ministuffle}.
\end{eqnarray*}
\end{theorem}

\begin{corollary}[\cite{SLC44}]\label{zetareg1}
For any $w\in X^*,\zeta_{\ministuffle}(w)$ belongs to the algebra $\calZ$.
\end{corollary}

\begin{theorem}[\cite{SLC44}]\label{reg2}
Let $\zeta_{\minishuffle}:(\QX,\shuffle)\rightarrow(\R,.)$ be the morphism verifying the following properties
\begin{itemize}
\item for $u,v\in X^*,\zeta_{\minishuffle}(u\minishuffle v)
=\zeta_{\minishuffle}(u)\zeta_{\minishuffle}(v)$,
\item for all convergent word $w\in x_0X^*x_1,\zeta_{\minishuffle}(w)=\zeta(w)$,
\item $\zeta_{\minishuffle}(x_0)=\zeta_{\minishuffle}(x_1)=0$.
\end{itemize}
Then
\begin{eqnarray*}
\sum_{w\in X^*}\zeta_{\minishuffle}(w)\;w&=&Z_{\minishuffle}.
\end{eqnarray*}
\end{theorem}

\begin{corollary}[\cite{SLC44}]\label{zetareg2}
For any $w\in Y^*,\zeta_{\minishuffle}(w)$ belongs to the algebra $\calZ$.
\end{corollary}

\begin{definition}
For any $w\in Y^*$, let $\gamma_{w}$ be the constant part\footnote{{\it i.e.} $\gamma_w$
is the  Euler-Mac Laurin constante of $\H_{w}(n)$.} of the asymptotic expasion,  on the comparison scale
$\{n^{a}\log^{b}(n)\}_{a\in\Z,b\in\N}$, of $\H_w(n)$.

Let $Z_{\gamma}$ be the noncommutative generating series of $\{\gamma_w\}_{w\in Y^*}$~:
\begin{eqnarray*}
Z_{\gamma}&:=&\sum_{w\in Y^*}\gamma_w\;w.
\end{eqnarray*}
\end{definition}

\begin{definition}\label{mono}
We set
\begin{eqnarray*}
B(y_1):=\exp\biggl(-\sum_{k\ge1}\gamma_{y_k}\frac{(-y_1)^k}{k}\biggr)
&\mbox{and}&B'(y_1):=e^{-\gamma y_1}B(y_1).
\end{eqnarray*}
\end{definition}
The power series $B'(y_1)$ corresponds in fact to the mould\footnote{
The readers can see why we have introduced the power series
${\rm Mono}(z)$ in Proposition \ref{C3}.} ${\rm Mono}$
in \cite{ecalle} and to the $\Phi_{\rm corr}$ in \cite{racinet}
(see also \cite{boutet,cartier2}). While the power series $B(y_1)$
corresponds to the Gamma Euler function with its product expansion,
\begin{eqnarray}
B(y_1)=\Gamma(y_1+1),&&
\frac1{\Gamma(y_1+1)}=e^{\gamma y_1}\prod_{n\ge1}\biggl(1+\frac{y_1}{n}\biggr)e^{-\gamma/n}.
\end{eqnarray}

\begin{lemma}[\cite{cade}]
Let $b_{n,k}(t_1,\ldots,t_{n-k+1})$ be the (exponential) partial Bell
polynomials in the variables $\{t_l\}_{l\ge1}$ given by the exponential generating series
\begin{eqnarray*}
\exp\biggl(u\Sum_{l=0}^{\infty}t_l\frac{v^l}{l!}\biggr)
&=&\Sum_{n,k=0}^{\infty}b_{n,k}(t_1,\ldots,t_{n-k+1})\frac{v^nu^k}{n!}.
\end{eqnarray*}
For any $m\ge1$, let $t_m=(-1)^{m}(m-1)!\gamma_{y_m}$. Then
\begin{eqnarray*}
B(y_1)&=&1+\sum_{n\ge1}\biggl(\sum_{k=1}^nb_{n,k}
(\gamma,-\zeta(2),2\zeta(3),\ldots)\biggr)\frac{(-y_1)^n}{n!}.
\end{eqnarray*}
\end{lemma}
Since the ordinary generating series of the finite parts of coefficients of
$\mathrm{Const}(N)$ is nothing else but the power series $B(y_1)$,
taking the constant part on  either side of
$\H(N)\;{}_{\widetilde{N \rightarrow \infty}}\;\mathrm{Const}(N)\pi_YZ_{\minishuffle}$
(see Proposition \ref{C3}), yields

\begin{theorem}[\cite{cade}]\label{regulation}
We have $Z_{\gamma}=B(y_1)\pi_Y Z_{\minishuffle}$.
\end{theorem}
Identifying the coefficients of $y_1^kw$  on  either side using the identity\footnote{
By the Convolution Theorem \cite{IMACS}, this is equivalent to
\begin{eqnarray*}
\forall u\in X^*,\quad\alpha_0^z(x_1^kx_0u)
&=&\int_0^z\frac{[\log(1-s)-\log(1-z)]^k}{k!}\alpha_0^s(u)\frac{ds}s\\
&=&\sum_{l=0}^k\frac{[-\log(1-z)]^l}{l!}\int_0^z\frac{\log^{k-l}(1-s)}{(k-l)!}
\alpha_0^s(u)\frac{ds}s.
\end{eqnarray*}
This theorem induces {\it de facto} the algebra morphism of regularization to $0$
with respect to the shuffle product, as shown the Theorem \ref{reg2}.}
\cite{SLC44}
\begin{eqnarray}
\forall u\in X^*x_1,\quad
x_1^kx_0u&=&\sum_{l=0}^kx_1^l\shuffle(x_0[(-x_1)^{k-l}\shuffle u])
\end{eqnarray}
and applying the morphism $\zeta_{\minishuffle}$ given in Theorem \ref{reg2},
we get \cite{SLC44}
\begin{eqnarray}
\forall u\in X^*x_1,\quad\zeta_{\minishuffle}(x_1^kx_0u)&=&\zeta(x_0[(-x_1)^k\shuffle u]).
\end{eqnarray}
\begin{corollary}[\cite{cade}]\label{generalizedgamma}
For $w\in x_0X^*x_1$, {\em i.e.} $w=x_0u$ and $\pi_Y w\in Y^*- y_1Y^*$,
and for $k\ge0$, the constant $\gamma_{\ministuffle}(x_1^kw)$
associated to the divergent polyz\^eta $\zeta(x_1^kw)$ is a polynomial
of degree $k$ in $\gamma$ and with coefficients in $\calZ$~:
\begin{eqnarray*}
\gamma_{x_1^kw}
&=&\sum_{i=0}^k\frac{\zeta(x_0[(-x_1)^{k-i}\minishuffle u])}{i!}
\biggl(\sum_{j=1}^ib_{i,j}(\gamma,-\zeta(2),2\zeta(3),\ldots)\biggr).
\end{eqnarray*}
Moreover, for $l=0,..,k$, the coefficient of $\gamma^l$
is of weight $\bv w\bv +k-l$.

In particular, for $s>1$, the constant $\gamma_{y_1y_s}$ associated to
$\zeta(y_1y_s)$ is linear in $\gamma$ and with coefficients in
$\Q[\zeta(2),\zeta(2i+1)]_{0<i\le(s-1)/2}$.
\end{corollary}

\begin{corollary}[\cite{cade}]
The constant $\gamma_{x_1^k}$ associated to the divergent polyz\^eta
$\zeta(x_1^k)$ is a polynomial of degree $k$ in $\gamma$
with coefficients in $\Q[\zeta(2),\zeta(2i+1)]_{0<i\le(k-1)/2}$~:
\begin{eqnarray*}
\gamma_{x_1^k}&=&\sum_{\myover{s_1,\ldots,s_k\ge0}{s_1+\ldots+ks_k=k+1}}
\frac{(-1)^k}{s_1!\ldots s_k!}
(-\gamma)^{s_1}\biggl(-\frac{\zeta(2)}{2}\biggr)^{s_2}\ldots
\biggl(-\frac{\zeta(k)}{k}\biggr)^{s_k}.
\end{eqnarray*}
Moreover, for $l=0,..,k$, the coefficient of $\gamma^l$ is of weight $k-l$.
\end{corollary}

We thereby obtain the following algebra morphism, denoted by $\gamma_{\bullet}$, for the regularization
to $\gamma$ with respect to the quasi-shuffle product {\it independently}
to the regularization with respect to the shuffle product\footnote{In \cite{boutet,cartier2,IKZ,waldschmidt},
the authors suggest the {\it simultaneous} regularizations, with respect to the shuffle product and the quasi-shuffle product, to $T$ and then to set $T=0$.} and then by applying the tensor product of morphisms $\gamma_{\bullet}\otimes\mathrm{Id}$ on the diagonal series, over $Y$, we get (see Annexe A)

\begin{theorem}\label{reg3}
The mapping $\gamma_{\bullet}$ realizes the morphism from
$(\QY,\stuffle)$ to $(\R,.)$ verifying the following properties
\begin{itemize}
\item for any word $u,v\in Y^*,\gamma_{u\ministuffle v}=\gamma_{u}\gamma_{v}$,
\item for any convergent word $w\in Y^*- y_1Y^*,\gamma_{w}=\zeta(w)$
\item $\gamma_{y_1}=\gamma$.
\end{itemize}
Then $Z_{\gamma}=e^{\gamma y_1}Z_{\ministuffle}$.
\end{theorem}

\subsubsection{Identities of noncommutative generating series of polyz\^etas}\label{doubleregulariszation}

\begin{corollary}\label{zigzig}
With the notations of Definition \ref{mono}, we have
\begin{eqnarray*}
Z_{\gamma}=B(y_1)\pi_YZ_{\minishuffle}&\iff&Z_{\ministuffle}=B'(y_1)\pi_YZ_{\minishuffle},\\
\pi_YZ_{\minishuffle}=B^{-1}(x_1)Z_{\gamma}&\iff&Z_{\minishuffle}=B'^{-1}(x_1)\pi_XZ_{\ministuffle}.
\end{eqnarray*}
\end{corollary}

Roughly speaking, for the quasi-shuffle product, the regularization to $\gamma$
is ``equivalent" to the regularization to $0$.

Note also that the constant $\gamma_{y_1}=\gamma$ is obtained as the finite part of
the asymptotic expansion of $\H_1(n)$ in the comparison scale  $\{n^{a}\log^{b}(n)\}_{a\in\Z,b\in\N}$.

In the same way, since $n$ and $\H_1(n)$ are algebraically independent,
as arithmetical functions (see Proposition \ref{isomorphisms}),
then $\{n^{a}\H_1^{b}(n)\}_{a\in\Z,b\in\N}$ constitutes a new comparison scale for asymptotic expansions.

Hence, the constants $\zeta_{\minishuffle}(x_1)=0$ and $\zeta_{\ministuffle}(y_1)=0$
can be interpreted as the finite part of the asymptotic expansions of
$\Li_1(z)$ and $\H_1(n)$ respectively in the comparison scales $\{(1-z)^{a}\log(1-z)^{b}\}_{a\in\Z,b\in\N}$
and $\{n^{a}\H_1^{b}(n)\}_{a\in\Z,b\in\N}$.

\begin{definition}[\cite{SLC44}]\label{motconvergent}
Let $C_1:=\Q\e\oplus x_0\QX x_1,C_2:=\Q\e\oplus(Y-\{y_1\})\QY$.
\end{definition}

\begin{lemma}[\cite{SLC43,SLC44}]\label{congruence}
We get $(C_1,\shuffle)\cong(C_2,\stuffle)$.
\end{lemma}
Using a theorem of Radford \cite{reutenauer} and its analogous over $Y$ (see Annexe A), we get
\begin{proposition}[\cite{SLC43,SLC44}]
\begin{eqnarray*}
(\QX,\shuffle)\quad\cong&(\Q[\Lyn X],\shuffle)&=\quad C_1[x_0,x_1],\\
(\QY,\stuffle)\quad\cong&(\Q[\Lyn Y],\stuffle)&=\quad C_2[y_1],
\end{eqnarray*}
\end{proposition}
This insures the effective way to get the finite part of the asymptotic expansions,
in the comparison scales  $\{(1-z)^{a}\log(1-z)^{b}\}_{a\in\Z,b\in\N}$
and $\{n^{a}\H_1^{b}(n)\}_{a\in\Z,b\in\N}$,
of $\{\Li_w(z)\}_{w\in Y^*}$ and $\{\H_w(N)\}_{w\in Y^*}$ respectively.

\begin{proposition}[\cite{SLC43,SLC44}]\label{coincide}
The restrictions of $\zeta_{\minishuffle}$ and $\zeta_{\ministuffle}$
over $(C_1,\shuffle)$ and $(C_2,\stuffle)$ respectively coincide
with the following {\em surjective} algebra morphism
\begin{eqnarray*}
\zeta\quad:\quad{(C_2,\stuffle)\atop(C_1,\shuffle)}&\longrightarrow&(\R,.)\\
{y_{r_1}\ldots y_{r_k}\atop x_0x_1^{r_1-1}\ldots x_0x_1^{r_k-1}}
&\longmapsto&
\sum_{n_1>\ldots>n_k>0}\Frac1{n_1^{r_1}\ldots n_k^{r_k}},
\end{eqnarray*}
\end{proposition}
In Section \ref{kerzeta} we will give the complete description of the kernel $\ker\zeta$.

With the double regularization\footnote{This double regularization is deduced from of the noncommutative generating series $Z_{\minishuffle}$ and $Z_{\ministuffle}$ on the definitions \ref{psikz} and \ref{phikz} (see the theorems \ref{reg1} and \ref{reg2}).} to zero \cite{boutet,cartier2,SLC44,racinet}, the Drindfel'd associator $\Phi_{KZ}$ corresponds then
to $Z_{\minishuffle}$ (obtained with only convergent polyz\^etas) as being the unique group-like element satisfying \cite{FPSAC98,SLC43}
\begin{eqnarray}
\langle Z_{\minishuffle}\bv x_0\rangle=\langle Z_{\minishuffle}\bv x_1\rangle=0
&\mbox{and}&
\forall x\in x_0X^*x_1,\quad\langle Z_{\minishuffle}\bv w\rangle=\zeta(w).
\end{eqnarray}
As consequence of Proposition \ref{logL}, one has
\begin{proposition}[\cite{orlando}]\label{logZ}
\begin{eqnarray*}
\log Z_{\minishuffle}
&=&\sum_{w\in X^*}\zeta_{\minishuffle}(w)\;\pi_1(w),\\
&=&\sum_{k\ge1}\frac{(-1)^{k-1}}k\sum_{u_1,\ldots,u_k\in X^*-\{\epsilon\}}
\zeta_{\minishuffle}(u_1\shuffle\ldots\shuffle u_k)\;u_1\ldots u_k.
\end{eqnarray*}
\end{proposition}
The associator $\Phi_{KZ}$ can be also graded in the adjoint basis of $\calU(\LQX)$ as follows
\begin{proposition}[\cite{orlando}]\label{graded}
For any $l\in\N$ and $P\in\CX$, let  $\circ$ denotes the composite operation defined by
$x_1x_0^l\circ P=x_1(x_0^l\shuffle P)$. Then
\begin{eqnarray*}
Z_{\minishuffle}&=&\Sum_{k\ge0}\Sum_{l_1,\cdots,l_k\ge0}
\zeta_{\minishuffle}(x_1x_0^{l_1}\circ\cdots\circ x_1x_0^{l_k})
\Prod_{i=0}^k\ad_{x_0}^{l_i}x_1,
\end{eqnarray*}
where $\ad_{x_0}^{l}x_1$ is iterated Lie bracket
$\ad_{x_0}^{l}x_1=[x_0,\ad_{x_0}^{l-1}x_1]$ and $\ad_{x_0}^0x_1=x_1$.
\end{proposition}
Using the following expansion \cite{bourbaki3}
\begin{eqnarray}
\ad^n_{x_0}x_1&=&\sum_{i=0}^n{i\choose n}x_0^{n-i}x_1x_0^i,
\end{eqnarray}
one deduces then, via the regularization process of Theorem \ref{reg2},
the expression of the Drindfel'd associator $\Phi_{KZ}$
given by L\^e and Murakami \cite{lemurakami}.

\subsection{Action of differential Galois group of polylogarithms on their asymptotic expansions}\label{galois}

\subsubsection{Group of associators theorem}

Let $A$ a be a commutative $\Q$-algebra.

Since the polyz\^etas satisfy (\ref{shuffle}), then by
the Friedrichs criterion we can state the following

\begin{definition}\label{dma}
Let $dm(A)$ be the set of $\Phi\in\AXX$ such that\footnote{$\Delta_{\minishuffle}$
denotes the co-product of the shuflle product.}
$\langle\Phi\bv \epsilon\rangle=1,\allowbreak\langle\Phi\bv x_0\rangle=\langle\Phi\bv x_1\rangle=0,
\Delta_{\minishuffle}\Phi=\Phi\otimes\Phi$ and such that, for
\begin{eqnarray*}
\Psi=B'(y_1)\pi_Y\Phi&\in&\AYY
\end{eqnarray*}
then\footnote{$\Delta_{\ministuffle}$ denotes the co-product of the quasi-shuflle product.}
$\Delta_{\ministuffle}\Psi=\Psi\otimes\Psi$.
\end{definition}

\begin{proposition}[\cite{orlando}]\label{sol}
If $G(z)$ and $H(z)$ are exponential solutions of (DE)
then there exsists a Lie series $C\in\LXX$ such that $G(z)=H(z)\exp(C)$.
\end{proposition}
\begin{proof}
Since $H(z)H(z)^{-1}=1$ then by differentiating, we have
\begin{eqnarray*}
d[H(z)]H(z)^{-1}&=&-H(z)d[H(z)^{-1}].
\end{eqnarray*}
Therefore if $H(z)$ is solution of Drinfel'd equation then
\begin{eqnarray*}
d[H(z)^{-1}]&=&-H(z)^{-1}[dH(z)]H(z)^{-1}\\
  &=&-H(z)^{-1}[x_0\omega_0(z)+x_1\omega_1(z)],\\
  d[H(z)^{-1}G(z)]&=&H(z)^{-1}[dG(z)]+[dH(z)^{-1}]G(z)\\
  &=&H(z)^{-1}[x_0\omega_0(z)+x_1\omega_1(z)]G(z)\\
  &-&H(z)^{-1}[x_0\omega_0(z)+x_1\omega_1(z)]G(z).
\end{eqnarray*}
By simplification, we deduce then $H(z)^{-1}G(z)$ is a constant formal power series.
Since the inverse and the product of group like elements is group like then we get the expected result.
\end{proof}

The differential $\calC$-module $\calC\{\Li_w\}_{w\in X^*}$ is the
universal Picard-Vessiot extension of every linear differential
equations, with coefficients in $\calC$ and admitting $\{0,1,\infty\}$ as
regular singularities. The universal differential Galois group,
noted by $\mathrm{Gal}(\LI_{\calC})$, is the set of differential
$\calC$-automorphisms of $\calC\{\Li_w\}_{w\in X^*}$ ({\em i.e} the
automorphisms of $\calC\{\Li_w\}_{w\in X^*}$ that let $\calC$
point-wise fixed and that commute with derivation).
The action of an automorphism of $\mathrm{Gal}(\LI_{\calC})$ can
be determined by its action on $\Li_w$, for $w\in X^*$. It can be
resumed as its action on the noncommutative generating series $\L$ \cite{orlando}~:

Let $\sigma\in\mathrm{Gal}(\LI_{\calC})$. Then
\begin{eqnarray}
\Sum_{w\in X^*}\sigma\Li_w\;w
&=&\Prod_{l\in\Lyn X}^{\searrow}e^{\sigma\Li_{{\check S}_l}S_l}.
\end{eqnarray}

Since $d\sigma\Li_{x_i}=\sigma d\Li_{x_i}=\omega_i$ then by integrating the two memmbers,
we obtain $\sigma\Li_{x_i}=\Li_{x_i}+c_{x_i}$, where $c_{x_i}$ is a constant of integration.
More generally, for any Lyndon word $l=x_il_1^{i_1}\cdots l_k^{i_k}$ with $l_1>\cdots>l_k$,
one has
\begin{eqnarray}
\sigma\Li_{{\check S}_l}&=&\int\omega_{x_i}
\Frac{\sigma\Li_{{\check S}_{l_1}}^{i_1}}{i_1!}\cdots
\frac{\sigma\Li_{{\check S}_{l_k}}^{i_k}}{i_k!}+c_{{\check S}_l},
\end{eqnarray}
where $c_{\check S_l}$ is a constant of integration. For example,
\begin{eqnarray}
\sigma\Li_{x_0x_1}&=&\Li_{x_0x_1}+c_{x_1}\Li_{x_0}+c_{x_0x_1},\\
\sigma\Li_{x_0^2x_1}&=&
\Li_{x_0^2x_1}+\frac{c_{x_1}}2\Li^2_{x_0}+c_{x_0x_1}\Li_{x_0}+c_{x_0^2x_1},\\
\sigma \Li_{x_0x_1^2}&=&
\Li_{x_0x_1^2}+c_{x_1}\Li_{x_0x_1}+\frac{c_{x_1}^2}2\Li_{x_0}+c_{x_0x_1^2}.
\end{eqnarray}
Consequently,
\begin{eqnarray}
\Sum_{w\in X^*}\sigma\Li_w\;w=\L e^{C_{\sigma}}
&\mbox{where}&
e^{C_{\sigma}}:=\Prod_{l\in\Lyn X}^{\searrow}e^{c_{{\check S}_l}S_l}.
\end{eqnarray}
The action of $\sigma\in\mathrm{Gal}(\LI_{\calC})$ over
$\{\Li_w\}_{w\in X^*}$ is then equivalent to the action of the Lie exponential
$e^{C_{\sigma}}\in\mathrm{Gal}(DE)$ over the exponential solution $\L$. So,

\begin{theorem}[\cite{orlando}]
We have
$\mathrm{Gal}(\LI_{\calC})=\{e^C\;\bv\;C\in\LXX\}$.
\end{theorem}
Typically, since $\L(z_0)^{-1}$ is group-like then $S_{z_0\path z}=\L(z)\L(z_0)^{-1}$
is an other solution of (\ref{drinfeld}) as already saw in (\ref{chenpolylog}).

\begin{theorem}[Group of associators theorem]\label{associatorgp}
Let $\Phi\in\AXX$ and $\Psi\in\AYY$ be group-like elements, for the co-products  $\Delta_{\minishuffle},\Delta_{\ministuffle}$ respectively,  such that $\Psi=B(y_1)\pi_Y\Phi$. There exists an unique $C\in\LAXX$ such that
$\Phi=Z_{\minishuffle}e^C$ and $\Psi=B(y_1)\pi_Y(Z_{\minishuffle}e^C)$.
\end{theorem}

\begin{proof}
If $C\in\LAXX$ then $\L'=\L e^C$ is group-like, for the co-product  $\Delta_{\minishuffle}$, and $e^C\in\mathrm{Gal}(DE)$.
Let $\H'$ be the noncommutative generating series of the Taylor coefficients,
belonging to the harmonic algebra, of $\{(1-z)^{-1}\langle\L'\bv w\rangle\}_{w\in Y^*}$.
Then $\H'(N)$ is also group-like, for the co-product $\Delta_{\ministuffle}$.
By the asymptotic expansion of $\L$, we have
${\L'(z)}\;{}_{\widetilde{\eps\to1}}\;e^{-x_1\log(1-z)}Z_{\minishuffle}e^C$ \cite{FPSAC99,SLC43}.
We put then $\Phi:=Z_{\minishuffle}e^C$ and we deduce that
\begin{eqnarray*}
\frac{\L'(z)}{1-z}\;{}_{\widetilde{z\to1}}\;\mathrm{Mono}(z)\Phi
&\mbox{and}&
\H'(N)\;{}_{\widetilde{N\to\infty}}\;\mathrm{Const}(N)\pi_Y\Phi,
\end{eqnarray*}
where the expressions of $\mathrm{Mono}(z)$ and $\mathrm{Const}(N)$
are given on (\ref{Mono}) and (\ref{Const}) respectively.
Let $\kappa_w$ be the constant part of $\H'_w(N)$. Then
\begin{eqnarray*}
\sum_{w\in Y^*}\kappa_w\;w&=&B(y_1)\pi_Y\Phi.
\end{eqnarray*}
We put then $\Psi:=B(y_1)\pi_Y\Phi$ (and also $\Psi':=B'(y_1)\pi_Y\Phi$).
\end{proof}

\begin{corollary}\label{associators}
We have
\begin{eqnarray*}
dm(A)=\{Z_{\minishuffle}e^C\bv C\in\LAXX&\mbox{and }
\langle e^C\bv\epsilon\rangle=1,\langle e^C\bv x_0\rangle=\langle e^C\bv x_1\rangle=0\}.
\end{eqnarray*}
\end{corollary}

\begin{proof}
On the one hand, $\langle\Phi\bv x_0\rangle=\langle Z_{\minishuffle}\bv x_0\rangle=0$,
$\langle\Phi\bv x_1\rangle=\langle Z_{\minishuffle}\bv x_1\rangle=0$
and on the other, $\langle\Phi\bv\epsilon\rangle=\langle Z_{\minishuffle}\bv\epsilon\rangle=1$,
 the result follows.
\end{proof}

Note also that if $\calZ\subset A$ then $dm(A)$ forms a group and with the notations of Corollary \ref{zigzig}, we obtain
\begin{corollary}\label{phiphi}
For any associator $\Phi=Z_{\minishuffle}e^C\in dm(A)$,
let $\Psi=B(y_1)\pi_Y\Phi$ and let $\Psi'=B'(y_1)\pi_Y\Phi$. Then
\begin{eqnarray*}
\Psi=B(y_1)\pi_Y\Phi&\iff&\Psi'=B'(y_1)\pi_Y\Phi.
\end{eqnarray*}
\end{corollary}
\begin{proof}
Since $\Psi$ is group like and since
$\langle\Phi\bv x_1\rangle=\langle\Psi'\bv y_1\rangle=0$ and
$\langle\Psi\bv y_1\rangle=\gamma$ then, using the factorization by Lyndon words,
we get the expected result.
\end{proof}

\begin{lemma}\label{generators}
Let $\Phi=Z_{\minishuffle}e^C\in dm(A)$ and let $\Psi=B(y_1)\pi_Y(Z_{\minishuffle}e^C)$.
The local coordinates (of second kind) of $\Phi$ (resp. $\Psi$) are polynomials
on $\{\zeta_{\minishuffle}(\check S_l)\}_{l\in\Lyn X}$
(resp. $\{\zeta_{\ministuffle}(\check\Sigma_l)\}_{l\in\Lyn Y}$) of $\calZ$ (resp. $\calZ'$).
While $C$ describes $\LAXX$, these coordinates
describe $A[\{\zeta_{\minishuffle}(\check S_l)\}_{l\in\Lyn X}]$
(resp. $A[\{\zeta_{\ministuffle}(\check\Sigma_l)\}_{l\in\Lyn Y}]$).
\end{lemma}
\begin{proof}
Let $\Phi\in dm(A)$. By Corollary \ref{associators}, there exists $P\in\LAXX$ verifying
$\langle e^P\bv\epsilon\rangle=1,\langle e^P\bv x_0\rangle=\langle e^P\bv x_1\rangle=0$
such that $\Phi=Z_{\minishuffle}e^{P}$.
Using the factorization forms by Lyndon words, we get
\begin{eqnarray*}
\prod_{l\in\Lyn X-X}^{\searrow}e^{\phi(\check S_l)\;S_l}
&=&\biggl(\prod_{l\in\Lyn X-X}^{\searrow}e^{\zeta(\check S_l)\;S_l}\biggr)
\biggl(\prod_{l\in\Lyn X-X}^{\searrow}e^{p_{\check S_l}\;S_l}\biggr).
\end{eqnarray*}
Expanding the Hausdorff product and identifying the local coordinates in the PBW-Lyndon basis there exists
$I_l\subset\{\lambda\in\Lyn X-X\mbox{ s.t. }|\lambda|\le|l|\}$, for $l\in\Lyn X-X$,
and the coefficients $\{p'_{\check S_u}\}_{u\in I_l}$ belonging to $A$ such that
\begin{eqnarray*}
\phi(\check S_l)&=&\sum_{u\in I_l}p'_{\check S_u}\;\zeta(\check S_u).
\end{eqnarray*}
This belongs to $A[\{\zeta(\check S_l)\}_{l\in\Lyn X-X}]$
and holds for any $P\in\LAXX$.
\end{proof}

With the notations of Definition \ref{mono} and by Corollary \ref{phiphi}, we get in particular
\begin{lemma}
For any $\Phi\in dm(A)$, by identifying the local coordinates (of second kind)
on two members of the identities $\Psi=B(y_1)\pi_Y\Phi$, or equivalently of
$\Psi'=B'(y_1)\pi_Y\Phi$, we get polynomial relations, of
coefficients in $A$,  among generators of the $A$-algebra of convergent polyz\^etas.
\end{lemma}

Therefore,

\begin{theorem}\label{alginpdt}
While $\Phi$ describes $dm(A)$, the identities
$\Psi=B(y_1)\pi_Y\Phi$ describe the ideal of polynomial
relations, of coefficients in $A$, among generators of the $A$-algebra of convergent polyz\^etas.
Moreover, if the Euler constant, $\gamma$, does not belong to $A$
then these relations are algebraically independent on $\gamma$.
\end{theorem}

Simplyfied computations on Section \ref{kerzeta} is an example of such identities.
Some consequences of Theorem \ref{alginpdt} will be drawn in Section  \ref{gamma}.

\subsubsection{Concatenation of Chen generating series}\label{monodromy}

As an example of the action of the differential Galois group
of polylogarithms on their asymptotic expansions,
we are interrested on the action of their monodromy group
which is contained in $\mathrm{Gal}(DE)$.

The monodromies at $0$ and $1$ of $\L$ are given respectively by \cite{FPSAC98,SLC43}
\begin{eqnarray}\label{monodromies}
\calM_0\L=\L e^{2\mathrm{i}\pi\mathfrak{m}_0}&\mbox{and}&
\calM_1\L=\L Z_{\minishuffle}^{-1}e^{-2\mathrm{i}\pi x_1}Z_{\minishuffle}
=\L e^{2\mathrm{i}\pi\mathfrak{m}_1},\\
\mbox{where}\quad
\mathfrak{m}_0=x_0&\mbox{and}&
\mathfrak{m}_1=\Prod_{l\in\Lyn X-X}^{\searrow}
e^{-\zeta(\check S_l)\ad_{S_l}}(-x_1).
\end{eqnarray}
\begin{itemize}
\item If $C=2\mathrm{i}\pi\mathfrak{m}_0$ then
\begin{eqnarray}
\Phi&=&Z_{\minishuffle}e^{2\mathrm{i}\pi x_0},\label{m_0}\\
\Psi&=&\exp\biggl(\gamma y_1
-\Sum_{k\ge2}\zeta(k)\Frac{(-y_1)^k}{k}\biggr)\pi_YZ_{\minishuffle}\\
&=&Z_{\ministuffle}.
\end{eqnarray}
The monodromy at $0$ consists in the multiplication on the right of $Z_{\minishuffle}$
by $e^{2\mathrm{i}\pi x_0}$ and does not modify $Z_{\ministuffle}$.
\item If $C=2\mathrm{i}\pi\mathfrak{m}_1$ then
\begin{eqnarray}
\Phi&=&e^{-2\mathrm{i}\pi x_1}Z_{\minishuffle},\label{m_1}\\
\Psi&=&\exp\biggl((\underbrace{\gamma-2\mathrm{i}\pi}_{T:=})y_1
-\Sum_{k\ge2}\zeta(k)\Frac{(-y_1)^k}{k}\biggr)\pi_YZ_{\minishuffle}\\
&=&e^{-2\mathrm{i}\pi y_1}Z_{\ministuffle}.
\end{eqnarray}
The monodromy at $1$ consists in the multiplication on left of $Z_{\minishuffle}$
and of $Z_{\ministuffle}$ by $e^{-2\mathrm{i}\pi x_1}$ and $e^{-2\mathrm{i}\pi y_1}$ respectively.
\end{itemize}

\begin{remark}
\begin{enumerate}
\item The monodromies around singularities of $\L$ could not
allow, in this case, neither to introduce the factor $e^{\gamma x_1}$
on the left of $Z_{\minishuffle}$ nor to eliminate the left factor
$e^{\gamma y_1}$ in $Z_{\gamma}$ (by putting\footnote{Why ?} $T=0$, for example).

\item By Proposition \ref{chenregularization}, we already saw that $Z_{\minishuffle}$
is the concatenation of Chen generating series \cite{chen}
$e^{x_0\log\varepsilon}$ and then $S_{\eps\path1-\eps}$ and finally, $e^{x_1\log\varepsilon}$~:
\begin{eqnarray}
Z_{\minishuffle}&{}_{\widetilde{\eps\rightarrow0^+}}&
e^{x_1\log\eps}\;S_{\eps\path1-\eps}\;e^{x_0\log\eps}.
\end{eqnarray}
From (\ref{m_0}) and (\ref{m_1}), the action of the monodromy group gives
\begin{eqnarray}
e^{x_1\;2k_1\mathrm{i}\pi}Z_{\minishuffle}e^{x_0\;2k_0\mathrm{i}\pi}
&{}_{\widetilde{\eps\rightarrow0^+}}&
e^{x_1(\log\eps+2k_1\mathrm{i}\pi)}\;
S_{\eps\path1-\eps}\;e^{x_0(\log\eps+2k_0\mathrm{i}\pi)},
\end{eqnarray}
as being the concatenation of the Chen generating series
$e^{x_0(\log\varepsilon+2k_0\mathrm{i}\pi)}$
(along circular path turning $k_0$ times around $0$),
then the Chen generating series $S_{\eps\path1-\eps}$ and finally,
the Chen generating series $e^{x_1(\log\varepsilon+2k_1\mathrm{i}\pi)}$
(along circular path turning $k_1$ times around $1$).

\item More generally, by Corollary \ref{associators}, the action of the Galois differential group
of polylogarithms states, for any Lie series $C$,
the associator $\Phi=Z_{\minishuffle}e^C$
is the concatenation of some Chen generating series $e^C$
and $e^{x_0\log\varepsilon}$ and then the Chen generating series $S_{\eps\path1-\eps}$
and finally, $e^{x_1\log\varepsilon}$~:
\begin{eqnarray}\label{concatenationofchenseries}
\Phi&{}_{\widetilde{\eps\rightarrow0^+}}&
e^{x_1\log\eps}\;S_{\eps\path1-\eps}\;e^{x_0\log\eps}\;e^C.
\end{eqnarray}
\end{enumerate}
\end{remark}

By construction (see Theorem \ref{associatorgp}) the associator $\Phi$ is then the noncommutative generating series of the finite parts of the coefficients  of the Chen generating series $S_{z_0\path1-z_0}e^C$, for $z_0=\eps\rightarrow0^+$.
Hence,
\begin{corollary}\label{concatenationofassociators}
Let $\Phi\in dm(A)$. For any differential produced formal power series $S$ over $X$,  there exists\footnote{See Corollary \ref{polynomialrealisation} of Annexe B.} a {\em differential representation} $(\calA,f)$ such that~:
\begin{eqnarray*}
\langle\Phi\bbv S\rangle=\sum_{w\in X^*}\langle\Phi\bv w\rangle\;\calA(w)\circ f_{|_0}
=\prod_{l\in\Lyn X-X}^{\searrow}e^{\langle\Phi\bv\check S_l\rangle\;\calA(S_l)}\circ f_{|_0}.
\end{eqnarray*}
\end{corollary}

\subsection{Algebraic combinatorial studies of polynomial relation among poly\-z\^eta via a group of associators}\label{kerzeta}

Here, $\bar Y=\{y_1\}\cup\{\bar y_k\}_{k\ge2}$. With the factorization of the monoids $X^*$ and $\bar Y^*$ by Lyndon words, let $\{\hat l\}_{l\in\Lyn X}$ and $\{\hat l\}_{l\in\Lyn\bar Y}$ be the dual of the Lyndon basis over $X$ and $\bar Y$.

\subsubsection{Preliminary study}
As in Definition \ref{motconvergent}, let
\begin{eqnarray}
A_1=A\e\oplus x_0\AX x_1&\mbox{and}&A_2=A\e\oplus(\bar Y-\{y_1\})\AYb.
\end{eqnarray}
For $\Phi\in dm(A)$, let $\Psi=B'(y_1)\pi_{\bar Y}\Phi$. Let us introduce two algebra morphisms
\begin{eqnarray}
\begin{array}{ccc}
\phi:(A_1,\shuffle)&\longrightarrow&A,\\
u&\longmapsto&\langle\Phi\bv u\rangle,
\end{array}\qquad
\begin{array}{ccc}
\psi:(A_2,\stuffle)&\longrightarrow&A,\\
v&\longmapsto&\langle\Psi\bv v\rangle,
\end{array}
\end{eqnarray}
verifying respectively $\phi(\epsilon)=1,\phi(x_0)=\phi(x_1)=0$ and $\psi(\epsilon)=1,\psi(y_1)=0$.

\begin{lemma}\label{coeff}
For any $\Phi\in dm(A)$, let $\Psi=B'(y_1)\pi_Y\Phi$. Then
\begin{eqnarray*}
\forall w\in\bar Y^*-y_1\bar Y^*,&&\psi(w)=\phi(\pi_Xw),\\
\mbox{or equivalently},\qquad\forall w\in x_0X^*x_1,&&\phi(w)=\psi(\pi_{\bar Y}w).
\end{eqnarray*}
\end{lemma}

\begin{lemma}\label{phipsi}
We have
\begin{eqnarray*}
\Phi=\sum_{u\in X^*}\phi(u)\;u
=\prod_{l\in\Lyn X-X}^{\searrow}e^{\phi(l)\;\hat l}&\mbox{and}&
\Psi=\sum_{v\in\bar  Y^*}\psi(u)\;u
=\prod_{l\in\Lyn\bar  Y-\{y_1\}}^{\searrow}e^{\psi(l)\;\hat l}.
\end{eqnarray*}
\end{lemma}

With the notations in Lemma \ref{phipsi}, we can state the following

\begin{definition}
We put
\begin{eqnarray*}
\calR:=\bigcap_{\Phi\in dm(A)}\ker\phi&(\mbox{resp.}&
\bigcap_{\Psi=B'(y_1)\pi_{\bar Y}\Phi\atop\Phi\in dm(A)}\ker\psi).
\end{eqnarray*}
\end{definition}

\begin{lemma}
For any $\Phi\in dm(A)$, let $\Psi=B'(y_1)\pi_{\bar Y}\Phi$. Let $Q\in\Q[\Lyn X]$ (resp. $\Q[\Lyn\bar  Y]$). Then
\begin{eqnarray*}
\langle Q\bbv\Phi\rangle=0\iff Q\in\ker\phi&
(\mbox{resp.}&\langle Q\bbv\Psi\rangle=0\iff Q\in\ker\psi).
\end{eqnarray*}
Or equivalently (see Definition  \ref{indiscernability}),
\begin{eqnarray*}
Q\in\calR&\iff&Q\mbox{ is indiscernable over }dm(A).
\end{eqnarray*}
\end{lemma}

Let $\Phi_1,\Phi_2\in dm(A)$. By Corollary \ref{associators}, for $i=1$ or $2$,
there exists an unique $P_i\in\LAXX$ such that $e^{-P_i}$
is well defined and
\begin{eqnarray}
\Phi_i=Z_{\minishuffle}e^{P_i},&\mbox{or equivalently,}&Z_{\minishuffle}=\Phi_1e^{-P_1}=\Phi_2e^{-P_2}.
\end{eqnarray}
Then, we get $\Phi_1=\Phi_2e^{P_1-P_2}$ and $\Phi_2=\Phi_1e^{P_2-P_1}$. By Lemma \ref{generators}, it follows

\begin{lemma}\label{Hausdorffproduct}
Let $\Phi_1$ and $\Phi_2\in dm(A)$. For any convergent Lyndon word, $l$,
there exists a finite set
$I_l\subset\{\lambda\in\Lyn X-X\mbox{ s.t. }|\lambda|\le|l|\}$
and the coefficients $\{p'_{i,u}\}_{u\in I_l}$ and $\{p''_{i,u}\}_{u\in I_l}$,
for $i=1$ or $2$, belonging to $A$ such that
\begin{eqnarray*}
\phi_i(l)=\sum_{u\in I_l}p'_{i,u}\;\zeta(u),&\mbox{or equivalently,}&
\zeta(l)=\sum_{u\in I_l}p''_{i,u}\;\phi_i(u).
\end{eqnarray*}
There also exists the coefficients $\{p'_u\}_{u\in I_l}$
and $\{p''_u\}_{u\in I_l}$ belonging to $A$ such that
\begin{eqnarray*}
\phi_1(l)=\sum_{u\in I_l}p'_{u}\;\phi_2(u),&\mbox{or equivalently,}&
\phi_2(l)=\sum_{u\in I_l}p''_{u}\;\phi_1(u).
\end{eqnarray*}
\end{lemma}
Therefore, the $\{\phi_i(l)\}_{l\in\Lyn X-X}$
(resp. $\{\psi_i(l)\}_{l\in\Lyn\bar  Y-\{y_1\}}$), for $i=1$ or $2$, are also generators
of the $A$-algebra generated by convergent polyz\^etas.

\subsubsection{Description of polynomial relations among coefficients of associator and irreducible polyz\^etas}

Since the identities of Corollary \ref{phiphi} (see also Corollary \ref{zigzig})
hold for any pair of bases, in duality, compatible with factorization
of the monoid $X^*$ (resp. $\bar Y^*$) then, by Corollary \ref{phiphi}, one gets

\begin{theorem}\label{repetita}
For any $\Phi\in dm(A)$, let $\Psi=B'(y_1)\pi_{\bar Y}\Phi$. We have
\begin{eqnarray*}
\Prod_{l\in\Lyn\bar Y-{y_1}}^{\searrow}e^{\psi(l)\;\hat l}
&=&\exp\biggl(\Sum_{k\ge2}\zeta(k)\Frac{(-y_1)^k}{k}\biggr)
\pi_{\bar Y}\Prod_{l\in\Lyn X-X}^{\searrow}e^{\phi(l)\;\hat l}.
\end{eqnarray*}
\end{theorem}

If $\Phi=Z_{\minishuffle}$ and $\Psi=Z_{\ministuffle}$ then,
for $\ell\in\Lyn X-X$ (resp. $\Lyn\bar Y-{y_1}$), one has
$\zeta(l)=\phi(l)$ (resp. $\psi(l)$).
Hence, one ontains (see also Corollary \ref{zigzig})

\begin{theorem}[Bis repetita]\label{Bisrepetita}
\begin{eqnarray*}
\Prod_{l\in\Lyn\bar Y-{y_1}}^{\searrow}e^{\zeta(l)\;\hat l}
&=&\exp\biggl(\Sum_{k\ge2}\zeta(k)\Frac{(-y_1)^k}{k}\biggr)
\pi_{\bar Y}\Prod_{l\in\Lyn X-X}^{\searrow}e^{\zeta(l)\;\hat l}.
\end{eqnarray*}
\end{theorem}

\begin{corollary}\label{ideal}
For  any $\ell\in\Lyn\bar Y-{y_1}$ (resp. $\Lyn X-X$), let
$P_\ell\in\calU(\LQX)$ (resp. $\calU(\LQYb)$) be the decomposition
of the polynomial $\pi_X\hat\ell\in\QX$ (resp. $\pi_{\bar Y}\hat\ell\in\QYb$)
in the PBW basis, induced by $\{\hat l\}_{l\in\Lyn X}$ (resp. $\{\hat l\}_{l\in\Lyn\bar Y}$),
and let $\check P_\ell\in\Q[\Lyn X-X]$ (resp. $\Q[\Lyn\bar Y-{y_1}]$) be its dual.
Then one obtains
\begin{eqnarray*}
\pi_X\ell-\check P_\ell\in\ker\phi&(\mbox{resp.}
&\pi_{\bar Y}\ell-\check P_\ell\in\ker\psi).
\end{eqnarray*}
In particular, for $\phi=\zeta$ (resp. $\psi=\zeta$) then one also obtains
\begin{eqnarray*}
\pi_X\ell-\check P_\ell\in\ker\zeta&(\mbox{resp.}
&\pi_{\bar Y}\ell-\check P_\ell\in\ker\zeta).
\end{eqnarray*}
Moreover, for any $\ell\in\Lyn\bar Y-{y_1}$ (resp. $\Lyn X-X$),
the homogenous polynomial $\pi_X\ell-\check P_\ell\in\QX$ (resp. $\QYb$)
is  of degree equal $\bv\ell\bv\ge2$.
\end{corollary}

\begin{proof}
Since
\begin{eqnarray*}\label{perrinlemma}
\ell\in\Lyn\bar Y&\iff&\pi_X\ell\in\Lyn X-\{x_0\}
\end{eqnarray*}
then identifying the local coordinates (of second kind) on the two members
of each identity in Theorem \ref{repetita}, one obtains
\begin{eqnarray*}
\forall\ell\in\Lyn\bar Y-{y_1}\subset Y^*-y_1Y^*,&&\psi(\ell)=\phi(\check P_\ell),\\
(\mbox{resp. }
\forall\ell\in\Lyn X-X\subset x_0X^*x_1,&&\phi(\ell)=\psi(\check P_\ell)).
\end{eqnarray*}
By Lemma \ref{coeff}, we get the expected result.
\end{proof}

With the notations of Corollary \ref{ideal}, we get the following

\begin{definition}\label{RXRYLirr}
Let $Q_\ell$ be the decomposition of the proper polynomial
$\pi_{\bar Y}\ell-\check P_\ell$ (resp. $\pi_X\ell-\check P_\ell$)
in $\Lyn\bar Y$ (resp. $\Lyn X$). Let
\begin{eqnarray*}
\calR_{\bar Y}:=\{Q_\ell\}_{\ell\in\Lyn\bar Y-{y_1}}&\mbox{and}&
\calR_X:=\{Q_\ell\}_{\ell\in\Lyn X-X},\\
\Lirr\bar Y:=\{\ell\in\Lyn\bar Y-{y_1}\bv Q_\ell=0\}&\mbox{and}&
\Lirr X:=\{\ell\in\Lyn X-X\bv Q_\ell=0\}.
\end{eqnarray*}
\end{definition}

It follows that
\begin{lemma}\label{oplus1}
We have
\begin{eqnarray*}
(\Q[\Lyn\bar Y-{y_1}],\stuffle)&=&(\calR_{\bar Y},\stuffle)\oplus(\Q[\Lirr\bar Y],\stuffle),\\
(\Q[\Lyn X-X],\shuffle)&=&(\calR_X,\shuffle)\oplus(\Q[\Lirr X],\shuffle).
\end{eqnarray*}
\end{lemma}

Then we can state the following
\begin{definition}\label{irreducible}
Any word $w$ is said to be {\em irreducible} if and only if $w$ belongs to $\Lirr\bar Y$ (resp. $\Lirr X$).
In this case, the polyz\^eta $\zeta(w)$ est said to be $\Q$-{\em irreducible}.
\end{definition}

For any $P\in\Q[\Lirr X]$, there exists\footnote{See Corollary \ref{polynomialrealisation} of Annexe B.} a {\it differential representation} $(\calA,f)$ such that $P$ can be {\em finitely} factorized (see also Corollary \ref{concatenationofassociators})~:
\begin{eqnarray}\label{factorized}
P=\sigma f_{|_0}=\sum_{w\in X^*_{irr}}\calA(w)\circ f\;w
=\prod_{\ell\in\Lirr X,\mathrm{finite}}^{\searrow}e^{\calA(\hat\ell)\;\ell}\circ f,
\end{eqnarray}
where $X^*_{irr}$ denotes the set of words obtaining by shuffling on $\Lirr X$.

\begin{lemma}\label{indiscernablesimple}
Any proper polynomial $P\in(\Q[\Lirr X],\shuffle)$ (resp. $(\Q[\Lirr\bar Y],\stuffle)$)
is indiscernable over Chen generating series $\{e^{t\;x}\}_{x\in X}^{t\in\C}$~: 
\begin{eqnarray*}
\langle P\bbv e^{t\;x_0}\rangle=\langle P\bbv e^{t\;x_1}\rangle=0
&(\mbox{resp.}&\langle P\bbv e^{ty_1}\rangle=0).
\end{eqnarray*}
\end{lemma}

\begin{proof}
By construction, $x_0$ and $x_1\notin\Lirr X$ (resp. $y_1\notin\Lirr X$). For any $n>1,x_0^n$ and $x_1^n$ (resp. $y_1^n$) are not Lyndon words then they do not belong to $\Lirr X$ (resp. $\Lirr X$). Therefore, for any $n\ge0$, one has
$\langle P\bv x_0^n\rangle=\langle P\bv x_1^n\rangle=0$ (resp. $\langle P\bv y_1^n\rangle=0$).
Using the expansion of the exponential, we find the expected result.
\end{proof}

\begin{lemma}\label{monoid2}
Let $\Phi\in dm(A)$ and let $t\in\C,x\in X$. For any proper polynomial
$P\in(\Q[\Lirr X],\shuffle)$, if $\langle P\bbv\Phi\rangle=0$ then
$\langle P\bbv\Phi e^{t\;x}\rangle=0$ and $\langle P\bbv e^{t\;x}\Phi\rangle=0$.
\end{lemma}

\begin{proof}
Since $P\in(\Q[\Lirr X],\shuffle)$ and $P$ is proper then,
by Lemma \ref{indiscernablesimple}, for any $t\in\C$ and for any $x\in X$,
we have $\langle P\bbv e^{t\;x}\rangle=0$ and then
$\langle P\bbv\Phi e^{t\;x}\rangle=0$.

Since $\supp(P)\subset x_0X^*x_1$  then $\langle P\bbv e^{t\;x_0}\Phi\rangle=\langle P\resd e^{t\;x_0}\bbv\Phi\rangle=0$.

Next, for $\Phi\in dm(A)$, there exists $e^C$ such that $e^{t\;x_1}\Phi=e^{t\;x_1}Z_{\minishuffle}e^C$
and, by Proposition \ref{chenregularization}, we get
\begin{eqnarray*}
e^{t\;x_1}\Phi&{}_{\widetilde{\eps\rightarrow0^+}}&
e^{x_1(t+\log\eps)}\;S_{\eps\path1-\eps}\;e^{x_0\log\eps}\;e^C.
\end{eqnarray*}
Hence, there exists a Chen generating series
$C_{z\path1-z_0}$ and $S_{z_0\path1-z_0}$ such that
we get the following asymptotic behaviour (see Section \ref{monodromy})
\begin{eqnarray*}
e^{t\;x_1}\Phi&{}_{\widetilde{\eps\rightarrow0^+}}&
C_{z\path1-z_0}S_{z_0\path z}\;e^C
\end{eqnarray*}
and the following concatenation holds \cite{chen} (see Formula (\ref{concate}))
\begin{eqnarray*}
C_{z\path1-z_0}S_{z_0\path z}&=&S_{z_0\path1-z_0},\\
\iff\qquad
C_{z\path1-z_0}S_{z_0\path z}e^C&=&S_{z_0\path1-z_0}e^C.
\end{eqnarray*}
Since $P\in\Q[\Lirr X]$ then by (\ref{factorized}), applying $\langle\sigma f_{|_0}\bbv\bullet\rangle$ to the two sides
of the previous equality, one has
\begin{eqnarray*}
\langle \sigma f_{|_0}\bbv C_{z\path1-z_0}S_{z_0\path z}e^C\rangle
&=&\langle\sigma f_{|_0}\bbv S_{z_0\path1-z_0}e^C\rangle.
\end{eqnarray*}
Thus, for $z_0=\eps\rightarrow0^+$, one obtains
\begin{eqnarray*}
\langle\sigma f_{|_0}\bbv e^{t\;x_1}\Phi\rangle
&{}_{\widetilde{\eps\rightarrow0^+}}&\langle\sigma f_{|_0}\bbv\Phi\rangle.
\end{eqnarray*}
Since $\langle\sigma f_{|_0}\bbv\Phi\rangle=\langle P\bbv\Phi\rangle=0$ then we get the expected result.
\end{proof}

\begin{lemma}\label{subset1}
For any $\Phi\in dm(A)$, let $\Psi=B'(y_1)\pi_{\bar Y}\Phi$. We have
$\calR_{\bar Y}\subseteq\ker\psi$ and $\calR_X\subseteq\ker\phi$.
In particular,
$\calR_{\bar Y}\subseteq\ker\zeta$ and $ \calR_X\subseteq\ker\zeta$.
\end{lemma}

\begin{proposition}\label{RX}
We have $\calR_X\subseteq\calR$ (resp. $\calR_{\bar Y}\subseteq\calR$).
\end{proposition}

\begin{proposition}\label{kerphi}
For any proper polynomial $Q\in\Q[\Lirr X]$ (resp. $\Q[\Lirr\bar Y]$),
\begin{eqnarray*}
Q\in\calR&\iff&Q=0.
\end{eqnarray*}
\end{proposition}

\begin{proof}
If $Q=0$ then since, for $\Phi\in dm(A),\phi$ is an algebra homorphism then $\phi(Q)=0$.
Hence, $Q\in\ker\phi$ and then $Q\in\calR$.

Conversely, if $Q\in\calR$ then, for $\Phi\in dm(A)$,
we get $\langle Q\bbv\Phi\rangle=0$. That means $Q$ is indiscernable over $dm(A)$.
Let $\calH$ be the monoid generated by $dm(A)$
and by the Chen generating series $\{e^{t\;x}\}_{x\in X}^{t\in\C}$.
By Lemma \ref{exponentiallycontinuous}, $Q$ is continuous over $\calH$
and by Lemma \ref{monoid2}, it is indiscernable over $\calH$.
By Proposition \ref{nulle}, the expected result follows.
\end{proof}

Therefore, by the propositions \ref{RX} and \ref{kerphi}, we obtain

\begin{theorem}
We have $\calR=\calR_X$ (resp. $\calR_{\bar Y}$).
\end{theorem}

\begin{proposition}\label{injectivite}
For any $\Phi\in dm(A)$, let $\Psi=B'(y_1)\pi_{\bar Y}\Phi$. Let $Q\in(\Q[\Lirr X],\shuffle)$ (resp. $(\Q[\Lirr\bar Y],\stuffle)$) such that $\langle\Phi\bbv Q\rangle=0$ (resp. $\langle\Psi\bbv Q\rangle=0$). Then $Q=0$.
\end{proposition}

\begin{proof}
Let $\calH$ defined as being the monoid generated by $\Phi$ and by Chen generating series
$\{e^{t\;x}\}_{x\in X}^{t\in\C}$. By assumption, $\langle\Phi\bbv Q\rangle=0$ and by Lemma \ref{monoid2}, $Q$ is then indiscernable over $\calH$. Finally, by Proposition \ref{nulle}, it follows that $Q=0$. 
\end{proof}

\begin{proposition}\label{noyeaudephi}
For any $\Phi\in dm(A)$, let $\Psi=B'(y_1)\pi_{\bar Y}\Phi$. We get
$\ker\phi=\calR_X$ (resp. $\ker\psi=\calR_{\bar Y}$).
In particular, $\ker\zeta=\calR_X$ (resp. $\ker\zeta=\calR_{\bar Y}$).
\end{proposition}

\begin{proof}
By Lemma \ref{subset1}, $\calR_X$ and $\calR_{\bar Y}$ are included in
$\ker\phi$ and $\ker\psi$ respectively.

Conversely, two cases can occur (see Lemma \ref{oplus1})~:
\begin{enumerate}
\item Case $Q\notin\Q[\Lirr X]$ (resp. $\Q[\Lirr\bar Y]$). By Lemma \ref{oplus1}, $Q\equiv_{\calR_X}Q_1$ (resp. $Q\equiv_{\calR_{\bar Y}}Q_1$) such that $Q_1\in\Q[\Lirr X]$ (resp. $\Q[\Lirr\bar Y]$) and $\phi(Q_1)=0$ (resp. $\psi(Q_1)=0$). This case is then reduced to the following

\item Case $Q\in\Q[\Lirr X]$ (resp. $\Q[\Lirr\bar Y]$).
Using Proposition \ref{injectivite}, we have $Q\equiv_{\calR_X}0$ (resp. $Q\equiv_{\calR_{\bar Y}}0$). 
\end{enumerate}
Then, $\calR_X$ (resp. $\calR_{\bar Y}$) contains $\ker\phi$ (resp. $\ker\psi$).
\end{proof}

For any $Q\in(\Q[\Lirr X],\shuffle)$ (resp. $(\Q[\Lirr\bar Y],\stuffle)$), $\zeta(Q)$ is then a polynomial on $A$-irreducible polyz\^etas (see Definition \ref{irreducible}). Moreover,

\begin{proposition}\label{generator}
The $\Q$-algebra $\calZ$ is generated by the family of $A$-irreducible polyz\^etas  $\{\zeta(\ell)\}_{\ell\in\Lirr\bar Y}$ (resp. $\{\zeta(\ell)\}_{\ell\in\Lirr X}$).
\end{proposition}

\begin{proof}
By Radford's theorem \cite{reutenauer}, one just needs to prove  for Lyndon words~:

Let $\ell\in\Lyn\bar Y-{y_1}$. If $\pi_X\ell=\check P_\ell$
then the result follows else one has $\pi_X\ell-\check P_\ell\in\ker\zeta$.
Hence, $\zeta(\ell)=\zeta(\check P_\ell)$.

Since $\check P_\ell\in\Q[\Lyn X-X]$ then $\check P_\ell$ is polynomial on Lyndon words, over $X$, of degree less or equal $\bv \ell\bv$. For each Lyndon word does appear in this decomposition of $\check P_\ell$, after applying $\pi_{\bar Y}$, one uses the same recurcive procedure until getting Lyndon words in $\Lirr\bar Y$.

The same treatment works for any $\ell'\in\Lyn X-X$.
\end{proof}

For any $\Phi\in dm(A)$, by Proposition \ref{noyeaudephi}, one also has
$\ker\phi=\ker\zeta=\calR_X$.
That means, for any irreducible Lyndon words $l\neq l'$,
\begin{eqnarray}
\phi(l)=\phi(l')&\iff&\zeta(l)=\zeta(l').
\end{eqnarray}

Let us state then the following
\begin{lemma}\label{varphi}
Let $\Phi\in dm(A)$. Let us define the map $$\varphi:\calZ\longrightarrow A$$
as follows
\begin{eqnarray*}
\forall l\in\Lirr X,&&\varphi(\zeta(l)):=\phi(l).
\end{eqnarray*}
Then $\varphi$ is an algebra homomorphism and $\{\varphi(\zeta(l))\}_{l\in\Lirr X}$ are generators of $A$.
\end{lemma}

Thus, for any $\theta\in\calZ$ there exist the coefficients $\{\alpha_{l_1,\ldots,l_n}\}^{n\in\N}_{l_1,\ldots,l_n\in\Lirr X}$ in $A$ such that (see Proposition \ref{generator} and Lemme \ref{varphi})
\begin{eqnarray}
\varphi(\theta)&=&\sum_{n\ge0}\sum_{l_1,\ldots,l_n\in\Lirr X}
\alpha_{l_1,\ldots,l_n}\;\varphi(\zeta(l_1))\ldots\varphi(\zeta(l_n))\label{alphabis}.
\end{eqnarray}
In particular, since for any $w\in X^*$, $\zeta_{\minishuffle}(w)$ belongs to
$\calZ$ (see Corollary \ref{zetareg2}) then $\varphi(\zeta_{\minishuffle}(w))$
is well defined and $\varphi(\zeta_{\minishuffle}(w))$ can be expressed as polynomial
on convergent polyz\^etas with coefficients in $A$~:
\begin{lemma}
With the notations in Lemma \ref{varphi}, one has
\begin{eqnarray*}
\forall w\in X^*,\quad\varphi(\zeta_{\minishuffle}(w))&=&\sum_{u,v\in X^*\atop uv=w}
\langle e^C\bv v\rangle\;\zeta_{\minishuffle}(u).
\end{eqnarray*}
\end{lemma}

\begin{proof}
The expected result follows by identifying coefficients in $\Phi=Z_{\minishuffle}e^C$.
\end{proof}

Finally, we can state the following
\begin{theorem}\label{associatormorphism}
For any $\Phi\in dm(A)$, there exists an unique algebra homomorphism
$$\varphi:\calZ\longrightarrow A$$
such that $\Phi$ is computed from $Z_{\minishuffle}$
by applying $\varphi$ to each coefficient~:
\begin{eqnarray*}
\Phi=\sum_{w\in X^*}\varphi(\zeta_{\minishuffle}(w))\;w=\prod_{l\in\Lyn X-X}^{\searrow}e^{\varphi(\zeta(l))\;\hat l}.
\end{eqnarray*}
\end{theorem}

\begin{remark}\label{rk2}
\begin{enumerate}
\item In this work, neither the question deciding any real number belongs to $\calZ$ or not nor the question expliciting $\{\alpha_{l_1,\ldots,l_n}\}^{n\in\N}_{l_1,\ldots,l_n\in\Lirr X}$ in (\ref{alphabis}), are considered.

\item\label{rk2p2} Now, by considering the commutative indeterminates $t_1,t_2,t_3,\ldots$,
let $A$ be the $\Q$-algebra obtained by specializing $\Q[t_1,t_2,t_3,\ldots]$
at $t_1=\mathrm{i}\pi$~:
\begin{eqnarray}
A=\Q[\mathrm{i}\pi][t_2,t_3,\ldots].
\end{eqnarray}
Neither the Lie exponential series $e^{\mathrm{i}\pi x_0}$ nor $e^{\mathrm{i}\pi x_1}$ does belong to $dm(A)$ but it belongs to $\mathrm{Gal}(DE)$. In particular, it figures in the modromies (see Section \ref{monodromy}) or in the functional relations (see (\ref{sigma2}) and (\ref{sigma3})) of polylogarithms  and in the hexagonal relation of polyz\^etas (see Proposition \ref{atomiser}).

\item\label{point3} Applying Baker--Campbell--Hausdorff formula \cite{bourbaki3} to Proposition \ref{atomiser} we get, at orders 2 and 3 as examples, the famous Euler's formula saying $\zeta(2)$ is an algebraic number over
$A=\Q[\mathrm{i}\pi]$~:
\begin{eqnarray}
\zeta(2)+\frac{(\mathrm{i}\pi)^2}6&=&0\quad\mbox{(order 2)},\\
\zeta(3)-\zeta(2,1)&=&0\quad\mbox{(order 3, imaginary part)}.\label{zeta21}
\end{eqnarray}
Therefore, the first comming in mind homomorphism
$$\varphi:\calZ\longrightarrow A$$
maps, at least
$\zeta(2)$ to $\varphi(\zeta(2))=\pi^2/6$.

\item For this reason, in \cite{HDR}, we have to consider the $\Q$-algebra
generated by $\mathrm{i}\pi$ and by other $A$-irreducible polyz\^etas obtained
in \cite{SLC43,FPSAC97,bigotte,elwardi} (and such algebra is denoted in this work by $A$).
This algebra came up from the studies of monodromies \cite{SLC43,FPSAC98},
as already shown in (\ref{monodromies}), and the Kummer type functional equations
of polylogarithms \cite{SLC43,FPSAC99}, as already shown in (\ref{sigma1})--(\ref{sigma3}).
In particular, by (\ref{sigma3}), we get for example \cite{FPSAC99,SLC43},
\begin{eqnarray}
 \Li_{2,1}{1\over t}&=&-{(\mathrm{i}\pi)^2\over 2}\log t
 +\mathrm{i}\pi(\zeta(2)-{\log^2t\over 2}-\Li_2t)\cr
 &&-\Li_{2,1}t+\Li_3t -\log t\Li_2t+\zeta(3)-{\log^3t\over 6}.
\end{eqnarray}
specializing $t=1$, the real part of this leads again to the Euler's indentity (\ref{zeta21}).
\end{enumerate}
\end{remark}

\section{Concluding remarks~: complete description of $\ker\zeta$ and structure of polyz\^etas}

For the same raison as already said in Remark \ref{rk2}(\ref{rk2p2}), let us consider now\footnote{We do not consider in any case $A=\Q$ as in previous versions.} the commutative indeterminates $t_1,t_2,t_3,\ldots$.

Let $A$ be the $\Q$-algebra obtained by specializing $\Q[t_1,t_2,t_3,\ldots]$
at $t_1=\mathrm{i}\pi$~:
\begin{eqnarray}
A=\Q[\mathrm{i}\pi][t_2,t_3,\ldots].
\end{eqnarray}

\subsection{A conjecture by Pierre Cartier}\label{Cartierconj}

\begin{definition}[\cite{cartier2,racinet}]
Let $DM(A)$ denotes the set of $\Phi\in\AXX$ such that
\begin{eqnarray*}
\langle\Phi\bv \epsilon\rangle=1,&\langle\Phi\bv x_0\rangle=\langle\Phi\bv x_1\rangle=0,
&\Delta_{\minishuffle}\Phi=\Phi\otimes\Phi
\end{eqnarray*}
and such that, for
\begin{eqnarray*}
\bar\Psi=\exp\biggl(-\Sum_{n\ge2}
\langle\pi_Y\Phi\bv y_n\rangle\Frac{(-y_1)^{n}}{n}\biggr)\pi_Y\Phi\in\AYY,
\end{eqnarray*}
then $\Delta_{\ministuffle}\bar\Psi=\bar\Psi\otimes\bar\Psi$.
\end{definition}

Since $DM(A)$ contains already $Z_{\minishuffle}$ then for $\Phi\in DM(A)$, by Theorem \ref{associatorgp}, there exists $C\in\LAXX$ verifying
\begin{eqnarray*}
\langle e^C\bv\epsilon\rangle=1&\mbox{and}&\langle e^C\bv x_0\rangle=\langle e^C\bv x_1\rangle=0
\end{eqnarray*}
such that
\begin{eqnarray}
\Phi&=&Z_{\minishuffle}e^C
\end{eqnarray}
and such that
\begin{eqnarray}
&&\Psi=B'(y_1)\pi_Y\Phi
=\exp\biggl(-\sum_{k\ge2}\zeta(k)\frac{(-y_1)^k}{k}\biggr)\pi_Y\Phi,\\
&&\bar\Psi=\exp\biggl(-\Sum_{n\ge2}\langle\pi_Y\Phi\bv y_n\rangle\Frac{(-y_1)^{n}}{n}\biggr)\pi_Y\Phi.
\end{eqnarray}

By construction (see Definition \ref{dma} and Theorem \ref{associatorgp}), such $\Phi$  and $\Psi$ are group-like (for the co-products $\Delta_{\minishuffle}$ and $\Delta_{\ministuffle}$ respectively) and here, $\bar\Psi$ must be also group-like (for the co-product $\Delta_{\ministuffle}$). If such a Lie series $C$ exists then it is unique, due to the fact that $e^C=\Phi Z_{\minishuffle}^{-1}$, and it is group-like (for the co-product $\Delta_{\minishuffle}$).

\begin{corollary}[conjectured by Cartier, \cite{cartier2}]
For any $\Phi\in DM(A)$, there exists an unique algebra homomorphism\footnote{
See Remark \ref{rk2}(\ref{point3}) to have an example of $\bar\varphi$.}
$$\bar\varphi:\calZ\longrightarrow A$$
such that $\Phi$ is computed from $Z_{\minishuffle}$
by applying $\bar\varphi$ to each coefficient.
\end{corollary}
\begin{proof}
By Theorem \ref{associatormorphism}, use the fact $$DM(\Q)\subseteq DM(A)\subseteq dm(A).$$
\end{proof}

\subsection{Arithmetical nature of $\gamma$}\label{gamma}

By Theorem \ref{alginpdt}, under the assumption that the Euler constant, $\gamma$,
does not belong to a commutative $\Q$-algebra $A$ then $\gamma$ does not verify
any polynomial with coefficients in $A$ among the convergent polyz\^etas.
It follows then,

\begin{corollary}\label{crutial}
If $\gamma\notin A$ then it is transcendental over the $A$-algebra generated
by the convergent polyz\^etas.
\end{corollary}

Or equivalently, by contraposition,

\begin{corollary}\label{contrario}
If there exists a polynomial relation  with coefficients in $A$ among the Euler constant,
$\gamma$, and the convergent polyz\^etas then $\gamma\in A$.
\end{corollary}

Therefore,

\begin{corollary}\label{notalgebraic}
If the Euler constant, $\gamma$, does not belong to $A$
then $\gamma$ is not algebraic over $A$.
\end{corollary}

\begin{remark}
\begin{enumerate}
\item In the same spirit of Theorem \ref{reg3}, let $\zeta_{\ministuffle}^{T}$
be the regularization morphism\footnote{This is a {\it symbolic} regularization and does not
yet have an analytical justification as it is done, separately,
for $\zeta_{\minishuffle}$ and $\zeta_{\ministuffle}$ in Section \ref{doubleregulariszation}
as finite parts of the asymptotic expansions,
in different scales of comparison, of $\Li_{x_1}(z)$, for $z\path1$, and $\H_{y_1}(N)$,
for $N\rightarrow\infty$, respectively.}
from $(\QY,\ministuffle)$ to $(\R,.)$ mapping $y_1$ to $T$.
Let $Z_{\ministuffle}^{T}$ be the noncommutative generating series
of polyz\^etas regulariszed with respect to $\zeta_{\ministuffle}^{T}$.
Thus, as in Theorem \ref{reg3} and by infinite factorization by Lyndon words, we also get
\begin{eqnarray}
Z_{\ministuffle}^{T}:=\sum_{w\in X^*}\zeta_{\ministuffle}^{T}(w)\;w=e^{Ty_1}Z_{\ministuffle}.
\end{eqnarray}

\item Now let us consider the regularization, for $N\rightarrow+\infty$ and with respect to $\zeta_{\ministuffle}^{T}$, of the power series $\mathrm{Const}(N)$ given in (\ref{Const}) as
\begin{eqnarray}
B^{T}(y_1)&=&e^{Ty_1}B'(y_1)
\end{eqnarray}

As in Corollary \ref{zigzig}, we always get
\begin{eqnarray}
Z_{\ministuffle}^{T}=B^{T}(y_1)\pi_YZ_{\minishuffle}
&\iff&Z_{\ministuffle}=B'(y_1)\pi_YZ_{\minishuffle}.
\end{eqnarray}
Hence, roughly speaking, for the quasi-shuffle product,
the symbolic regularization to $T$ is also ``equivalent" to the regularization to $0$.

\item Again, as in Corollary \ref{crutial}, if $T\notin A$ then $T\notin\bar A$.

{\it A contrario}, as in Corollary \ref{contrario}, if there exists a polynomial
relation with coefficients in $A$ among $T$ and convergent polyz\^etas then $T\in A$.
\end{enumerate}
\end{remark}

\subsection{Structure and arithmetical nature of polyz\^etas}

Once again, let us consider
\begin{eqnarray}
(A_1,\shuffle)&=&(A\e\oplus x_0\AX x_1,\shuffle)\\
&\cong&A[\Lyn X-X],\shuffle)\\
&=&(\calR_X,\shuffle)\oplus(A[\Lirr X],\shuffle),\\
(A_2,\stuffle)&=&(A\e\oplus(\bar Y-\{y_1\})\AYb,\stuffle)\\
&\cong&(A[\Lyn\bar Y-{y_1}],\stuffle)\\
&=&(\calR_{\bar Y},\stuffle)\oplus(A[\Lirr\bar Y],\stuffle)
\end{eqnarray}
(see Definition \ref{motconvergent}, Lemma \ref{oplus1}, Definition \ref{RXRYLirr}).
Then $(A_1,\shuffle)\cong(A_2,\stuffle)$ \cite{SLC43,SLC44}.

Let us consider again  the following algebra morphism (see Proposition \ref{coincide})
\begin{eqnarray}\label{zeta}
\zeta\quad:\quad{(A_2,\stuffle)\atop(A_1,\shuffle)}&\longrightarrow&(\R,.)\\
{y_{r_1}\ldots y_{r_k}\atop x_0x_1^{r_1-1}\ldots x_0x_1^{r_k-1}}
&\longmapsto&
\sum_{n_1>\ldots>n_k>0}\Frac1{n_1^{r_1}\ldots n_k^{r_k}}.
\end{eqnarray}

\begin{lemma}\label{image}
The image of the algebra morphism $\zeta$ is $\calZ$.
\end{lemma}

Let us make precise the structure of $\calZ$ and the arithmetical nature of polyz\^etas~: 

As consequences of the propositions \ref{injectivite}, \ref{noyeaudephi} and \ref{generator}, by taking $\Phi=Z_{\minishuffle}$, we have
\begin{eqnarray}
\mbox{Im}\;\zeta=\zeta(\A[\Lirr\bar Y])&\mbox{and}&\ker\zeta=\calR_{\bar Y}\label{imagenoyau}\\
(\mbox{resp.}\quad
\mbox{Im}\;\zeta=\zeta(\A[\Lirr X])&\mbox{and}&\ker\zeta=\calR_X).\label{imagenoyaubis}
\end{eqnarray}

By Corollary \ref{ideal}, $\ker\zeta$ is an ideal generated by the homogenous polynomials of degree $\ge2$. Hence, the quotien $A_1/\calR_X$ or $A_2/\calR_{\bar Y}$ (the source by the kernel of $\zeta$) is graded \cite{bourbaki3} and it is isomorphic to $\mbox{Im}\zeta$.

Therefore, by Lemma \ref{image} and  Proposition  \ref{generator}, we obtain respectively the following direct consequences

\begin{theorem}[Structure of polyz\^etas]\label{polyzetastructure}
The $A$-algebra $\calZ$ is
\begin{enumerate}
\item isomorphic to the graded algebra $(A_1/\calR_X,\shuffle)$, or equivalently, $(A_2/\calR_{\bar Y},\stuffle)$.
\item {\it freely} generated by the $A$-irreducible polyz\^etas $\{\zeta(l)\}_{l\in\Lirr\bar Y}$ (resp. $\{\zeta(l)\}_{l\in\Lirr X}$).
\end{enumerate}
\end{theorem}

For any $p\ge2$, let
\begin{eqnarray}
\calZ_p&=&\mathrm{span}_{\Q}\{\zeta(w)\;|\;w\in x_0X^*x_1,\abs{w}=p\}.
\end{eqnarray}
 By definition of graded algebra \cite{bourbaki3}, Theorem  \ref{polyzetastructure} means also that
\begin{eqnarray}
\calZ&=&A\oplus\bigoplus_{p\ge2}\calZ_p
\end{eqnarray}
and there is no linear relation among elements of different $\calZ_p$ (\cite{cartier2} conjecture C1, \cite{waldschmidt}).

Thus, if $\theta$ is a ($A$-irreducible) polyz\^eta verifying the following algebraic equation
\begin{eqnarray}\label{equation}
\theta^n+a_{n-1}\theta^{n-1}+\ldots+a_0&=&0
\end{eqnarray}
then $\theta=0$ because $\calZ_{p_1}\calZ_{p_2}\subset\calZ_{p_1+p_2}$, for $p_1,p_2\ge2$, and each monomial in (\ref{equation}) is then of different weight. By consequence,

\begin{corollary}
Any ($A$-irreducible) polyz\^eta $\theta$ is a transcendental over $\Q$.
\end{corollary}

\begin{remark}
In this work, neither the study of $\dim\calZ_p$  \cite{zagier} (see also  \cite{cartier2}, conjecture C2) nor the estimate of the number of $A$-irreducible polyz\^etas generating $\calZ_p$, are discussed knowing the $A$-irreducible polyz\^etas form transcendence basis of the $A$-algebra $\calZ$.
\end{remark}

\section{Annexe A : pair of bases in duality and proof of Theorem \ref{Hfactorization}}

\subsection{Preliminary results}
Let $\QY$ be equipped the concatenation and the quasi-shuffle, $\stuffle$, defined by
\begin{eqnarray*}
\forall y_i,y_j\in Y=\{y_i\}_{i\ge1},\forall u,v\in Y^*,&&y_iu\stuffle y_jv=y_i(u\stuffle y_jv)+y_{i+j}(y_iu\stuffle v),\\
\forall w\in Y^*,&&w\stuffle 1_{Y^*}=1_{Y^*}\stuffle w=w,
\end{eqnarray*}
or by its associated co-product, $\Delta_{\ministuffle}$,  defined by
\begin{eqnarray*}
\forall y_k\in Y,&&\Delta_{\ministuffle}(y_k)=y_k\otimes1+1\otimes y_k+\Sum_{i+j=k}y_i\otimes y_j.
\end{eqnarray*}
satistying, for any $u,v,w\in Y^*$, $\scal{u\otimes v}{\Delta_{\ministuffle}(w)}=\scal{u\stuffle v}{w}$.

\begin{lemma}\label{L3}
Let $S_1,\ldots,S_n$ be {\em proper} formal power series in $\QYY$.
Let $P_1,\ldots,P_m$ be primitive elements\footnote{{\it i.e.}, for any $i=1,..,m,\Delta_{\ministuffle}(P_i)=1\otimes P_i+P_i\otimes1$.} in $\QY$, for the co-product $\stuffle$.
\begin{enumerate}
\item\label{1}  If $n>m$ then  $\scal{S_1\stuffle\ldots\stuffle S_n}{P_1\ldots P_m}=0.$
\item\label{2}  If $n=m$ then
\begin{eqnarray*}
\scal{S_1\stuffle\ldots\stuffle S_n}{P_1\ldots P_n}&=&
\sum_{\sigma\in\mathfrak{S}_n}\prod_{i=1}^n\scal{S_i}{P_{\sigma(i)}}.
\end{eqnarray*}
\item\label{3} If $n<m$ then, by considering the  language $\mathcal{M}$ over $\calA=\{P_1,\ldots,P_m\}$
\begin{eqnarray*}
\mathcal{M}&=&\{w\in\calA^*|w=P_{j_1}\ldots P_{j_{|w|}},j_1<\ldots<j_{|w|},|w|\ge1\}
\end{eqnarray*} 
and the morphism $\mu:\Q\langle\calA\rangle\longrightarrow\QY$, one has~:
\begin{eqnarray*}
\scal{S_1\stuffle\ldots\stuffle S_n}{P_1\ldots P_m}
&=&\sum_{{w_1,\ldots,w_m\in\mathcal{M}\atop|w_1|+\ldots+|w_n|=m}
\atop\forall i,j=1,..m,{\rm alp}(w_i)\cap{\rm alp}(w_j)=\emptyset}
\prod_{i=1}^n\scal{S_i}{\mu(w_i)}.
\end{eqnarray*}
\end{enumerate}
\end{lemma}

\begin{proof}
On one hand, since the $P_i$'s are primitive then
\begin{eqnarray*}
\Delta_{\ministuffle}^{(n-1)}(P_i)&=&\sum_{p+q=n-1}1^{\otimes p}\otimes P_i\otimes1^{\otimes q}.
\end{eqnarray*}
On the other hand, $\scal{S_1\stuffle\ldots\stuffle S_n}{P_1\ldots P_m}
=\scal{S_1\otimes\ldots\otimes S_n}{\Delta_{\ministuffle}^{(n-1)}(P_1\ldots P_m)}$
and $\Delta_{\ministuffle}^{(n-1)}(P_1\ldots P_m)
=\Delta_{\ministuffle}^{(n-1)}(P_1)\ldots\Delta_{\ministuffle}^{(n-1)}(P_m)$.
Hence,
\begin{eqnarray*}
\scal{S_1\stuffle\ldots\stuffle S_n}{P_1\ldots P_m}
&=&\scal{\bigotimes_{i=1}^nS_i}{\prod_{i=1}^m\sum_{p+q=n-1}1^{\otimes p}\otimes P_i\otimes1^{\otimes q}}.
\end{eqnarray*}
\begin{enumerate}
\item  For $n>m$, by expanding
$\Delta_{\ministuffle}^{(n-1)}(P_1)\ldots\Delta_{\ministuffle}^{(n-1)}(P_m)$,
one obtains a sum of tensors contening at least one factor equal to $1$. 
For $j=1,..,n$, the formal power series $S_j$ is proper and the result follows immediatly.

\item  For $n=m$, since
\begin{eqnarray*}
\prod_{i=1}^n\Delta_{\ministuffle}^{(n-1)}(P_i)
&=&\sum_{\sigma\in\mathfrak{S}_n}\bigotimes_{i=1}^nP_{\sigma(i)}+Q,
\end{eqnarray*}
where $Q$ is sum of tensors contening at least one factor equal to $1$
and the $S_j$,'s are proper  then
$\scal{S_1\otimes\ldots\otimes S_n}{Q}=0$.
Thus, the result follows.
\item  For $n<m$, en tenant compte que, for $j=1,..,n$, formal power series $S_j$ is proper, the expected follows by expanding the product
\begin{eqnarray*}
\prod_{i=1}^m\Delta_{\ministuffle}^{(n-1)}(P_i)&=&
\prod_{i=1}^m\sum_{p+q=n-1}1^{\otimes p}\otimes P_i\otimes1^{\otimes q}.
\end{eqnarray*}
\end{enumerate}
\end{proof}

\begin{proposition}\label{pi_1}
\begin{enumerate}
\item  We have
\begin{eqnarray*}
\log\biggl(\sum_{w\in Y^*}w\otimes w\biggr)
=\sum_{w\in Y^+}w\otimes\pi_1(w)=\sum_{w\in Y^+}\pi_1^*(w)\otimes w,
\end{eqnarray*}
where $\pi_1^*$ is the adjoint of  $\pi_1$ and they are given by 
\begin{eqnarray*}
\pi_1(w)&=&\sum_{k\ge1}\frac{(-1)^{k-1}}{k}
\sum_{u_1,\ldots,u_k\in Y^+}\langle w\bv u_1\stuffle\ldots\stuffle u_k\rangle u_1\ldots u_k,\\
\pi_1^*(w)&=&\sum_{k\ge1}\frac{(-1)^{k-1}}{k}
\sum_{u_1,\ldots,u_k\in Y^+}\langle w\bv u_1\ldots u_k\rangle u_1\stuffle\ldots\stuffle u_k.
\end{eqnarray*}
In particular, for any $y_k\in Y$, one has
\begin{eqnarray*}
\pi_1(y_k)=y_k+\sum_{l\ge2}\frac{(-1)^{l-1}}{l}\sum_{j_1,\ldots,j_l\ge1\atop j_1+\ldots+j_l=k}y_{j_1}\ldots y_{j_l}
&\mbox{and}&
\pi_1^*(y_k)=y_k.
\end{eqnarray*}

\item  For any $w\in Y^*$, we  have
\begin{eqnarray*}
w&=&\sum_{k\ge0}\frac1{k!}\sum_{u_1,\ldots,u_k\in Y^*}
\langle w\bv u_1\stuffle\ldots\stuffle u_k\rangle\pi_1(u_1)\ldots\pi_1(u_k),\\
&=&\sum_{k\ge0}\frac1{k!}\sum_{u_1,\ldots,u_k\in Y^*}
\langle w\bv u_1\ldots u_k\rangle\pi_1^*(u_1)\stuffle\ldots\stuffle\pi_1^*(u_k).
\end{eqnarray*}
\end{enumerate}
\end{proposition}

\begin{proof}
\begin{enumerate}
\item  Expanding the logarithm, we have
\begin{eqnarray*}
\log\biggl(\sum_{w\in Y^*}w\otimes w\biggr)
&=&\sum_{k\ge1}\frac{(-1)^{k-1}}{k}\biggl(\sum_{w\in Y^+}w\otimes w\biggr)^k\\
&=&\sum_{k\ge1}\frac{(-1)^{k-1}}{k}\sum_{u_1,\ldots,u_k\in Y^+}
(u_1\stuffle\ldots\stuffle u_k)\otimes u_1\ldots u_k\\
&=&\sum_{w\in Y^+}w\otimes\biggl(\sum_{k\ge1}\frac{(-1)^{k-1}}{k}\\
&&
\sum_{u_1,\ldots,u_k\in Y^+}\langle w\bv u_1\stuffle\ldots\stuffle u_k\rangle u_1\ldots u_k\biggr).
\end{eqnarray*}
In the same way,
\begin{eqnarray*}
\log\biggl(\sum_{w\in Y^*}w\otimes w\biggr)
&=&\sum_{w\in Y^+}\biggl(\sum_{k\ge1}\frac{(-1)^{k-1}}{k}\\
&&\sum_{u_1,\ldots,u_k\in Y^+}\langle w\bv u_1\ldots u_k\rangle u_1\stuffle\ldots\stuffle u_k\biggr)\otimes w.
\end{eqnarray*}
Thus, the expressions of $\pi_1(w)$ and $\pi_1^*(w)$ follow immediatly.

\item  Since $\exp$ and $\log$ are mutually inverse then, by the previous results, one  has
\begin{eqnarray*}
\sum_{w\in Y^*}w\otimes w
&=&\sum_{k\ge0}\frac1{k!}\biggl(\sum_{w\in Y^+}w\otimes\pi_1(w)\biggr)^k\\
&=&\sum_{k\ge0}\frac1{k!}\sum_{u_1,\ldots,u_k\in Y^+}
(u_1\stuffle\ldots\stuffle u_k)\otimes(\pi_1(u_1)\ldots\pi_1(u_k))\\
&=&\sum_{w\in Y^+}w\otimes\biggl(\sum_{k\ge1}\frac1{k!}\cr
&&\sum_{u_1,\ldots,u_k\in Y^+}
\langle w\bv u_1\stuffle\ldots\stuffle u_k\rangle \pi_1(u_1)\ldots\pi_1(u_k)\biggr).
\end{eqnarray*}
In the same way,
\begin{eqnarray*}
\sum_{w\in Y^*}w\otimes w
&=&\sum_{k\ge0}\frac1{k!}\sum_{u_1,\ldots,u_k\in Y^+}
(\pi_1^*(u_1)\stuffle\ldots\stuffle\pi_1^*(u_k))\otimes(u_1\ldots u_k)\\
&=&\sum_{w\in Y^+}\biggl(\sum_{k\ge0}\frac1{k!}\cr
&&\sum_{u_1,\ldots,u_k\in Y^+}
\langle w\bv u_1\ldots u_k\rangle\pi_1^*(u_1)\stuffle\ldots\stuffle\pi_1^*(u_k)\biggr)\otimes w.
\end{eqnarray*}
\end{enumerate}
It follows then the expected result.
\end{proof}

\subsection{Pair of bases in duality}

\begin{definition}\label{Sigma}
Let $\{\Sigma_l\}_{l\in \Lyn Y}$ be the family of $\LQY$ obtained as follows
$$\begin{array}{cccll}
\Sigma_{y_k}&=&\pi_1(y_k)&\mbox{for}&k\ge1,\\
\Sigma_{l}&=&[\Sigma_s,\Sigma_r]&\mbox{for}&l\in\Lyn X,\mbox{ standard factorization of }l=(s,r),
\end{array}$$
and the family $\{\Sigma_w\}_{w\in Y^*}$ of $\calU(\LQY)$ (viewed as a $\Q$-module) obtained as follows
$$\begin{array}{cccll}
\Sigma_{l}&=&1&\mbox{for}&l=1_{Y^*},\\
\Sigma_{w}&=&\Sigma_{l_1}^{i_1}\ldots \Sigma_{l_k}^{i_k}
&\mbox{for}&w=l_1^{i_1}\ldots l_k^{i_k},l_1>\ldots>l_k,l_1\ldots,l_k\in\Lyn Y.
\end{array}$$
Let $\{\check\Sigma_w\}_{w\in Y^*}$ be the family  of the quasi-shuffle algebra  (viewed as a $\Q$-module) obtained by duality with $\{\Sigma_w\}_{w\in Y^*}$~:
\begin{eqnarray*}
\forall u,v\in Y^*,&&\scal{\check\Sigma_{v}}{\Sigma_u}=\delta_{u,v}.
\end{eqnarray*}
\end{definition}

\begin{proposition}\label{Ptriangulaire}
\begin{enumerate}
\item  For $l\in\Lyn Y$, the polynomial $\Sigma_l$ is upper triangular~:
\begin{eqnarray*}
\Sigma_l&=&l+\sum_{v>w,(v)=(l)}c_vv.
\end{eqnarray*}
\item The families $\{\Sigma_w\}_{w\in Y^*}$ and $\{\check\Sigma_w\}_{w\in Y^*}$ are upper and lower triangular respectively. On other words,  for any $w\in Y^+$, one has
\begin{eqnarray*}
\Sigma_w= w+\sum_{v>w,(v)=(w)}c_vv&\mbox{and}&\check\Sigma_w= w+\sum_{v<w,(v)=(w)}d_vv,
\end{eqnarray*}
\end{enumerate}
\noindent where, for any $y_k\in Y$ and $w\in Y^*$, $(w)$ denotes the degree of $w$ and $(y_k)=\deg(y_k)=k$.
\end{proposition}

\begin{proof}
\begin{enumerate}
\item  Let us prove it by induction on the length of $l$~:
\begin{itemize}
\item The result is immediat for $l\in Y$.
\item The result is suppose verified for any $l\in\Lyn Y\cap Y^k$ and $0\le k\le N$.
\item At $N+1$, by the standard factorization $(l_1,l_2)$ of $l$,
one has, by definition, $\Sigma_{l}=[\Sigma_{l_1},\Sigma_{l_2}]$ and $l_2l_1>l_1l_2=l$.
By induction hypothesis,
\begin{eqnarray*}
\Sigma_{ l_1}=l_1+\sum_{v> l_1\atop(v)=(l_1)}c_vv
&\mbox{and}&
\Sigma_{ l_2}=l_2+\sum_{u> l_2\atop(v)=(l_2)}d_uu,\\
&\Rightarrow&\Sigma_{l}=l+\sum_{w> l\atop(w)=(l)}e_ww,
\end{eqnarray*}
getting $e_w$'s from $c_v$'s and $d_u$'s. Actually, the Lie bracket gives
\begin{eqnarray*}
\Sigma_{l}&=&[l_1,l_2]\cr
&+&\sum_{u> l_2\atop(v)=(l_2)}d_ul_1u+\sum_{v> l_1, u> l_2\atop(v)=(l_1),(u)=(l_2)}c_vd_uvu\cr
&-&\sum_{v> l_1\atop(v)=(l_1)}c_vl_2v-\sum_{v> l_1, u> l_2\atop(v)=(l_2),(u)=(l_1)}c_vd_uuv\cr
&=&[l_1,l_2]\cr
&+&\sum_{u>l_1 l_2\atop(v)=(l_1l_2)}d'_uu+\sum_{vu> l_1 l_2\atop(vu)=(l_1l_2)}c_vd_uvu\cr
&-&\sum_{v>l_2 l_1\atop(v)=(l_2l_1)}c'_vv-\sum_{uv>l_2 l_1\atop(uv)=(l_2l_1)}c_vd_uuv\cr
&=&[l_1,l_2]\cr
&+&\sum_{u>l\atop(v)=(l)}d'_uu+\sum_{vu> l\atop(vu)=(l)}c_vd_uvu\cr
&-&\sum_{v>l_2 l_1> l\atop(v)=(l)}c'_vv-\sum_{uv>l_2 l_1> l\atop(uv)=(l)}c_vd_uuv.
\end{eqnarray*}
\end{itemize}
Hence, the conclusion follows.
\item  Let $w=l_1\ldots l_k$, with $l_1>\ldots>l_k$ and $l_1,\ldots,l_k\in\Lyn Y$.
By (ii),
\begin{eqnarray*}
\Sigma_{l_i}=l_i+\sum_{v>l_i\atop(v)=(l_i)}c_{i,v}v
&\mbox{and}&
\Sigma_w=l_1\ldots l_k+\sum_{u>w\atop(v)=(w)}d_uu,
\end{eqnarray*}
where the $d_u$'s are obtained from the $c_{i,v}$'s. Hence, the family $\{\Sigma_w\}_{w\in Y^*}$  is upper triangular and, by duality, the family $\{\check\Sigma_w\}_{w\in Y^*}$ is lower triangular.
\end{enumerate}
\end{proof}

\begin{theorem}
\begin{enumerate}
\item The family $\{\Sigma_l\}_{l\in\Lyn Y}$ forms a basis of the free Lie algebra.
\item The family $\{\Sigma_w\}_{w\in Y^*}$ forms a basis of the free associative algebra $\QY$.
\item The family $\{\check\Sigma_w\}_{w\in Y^*}$ generate freely the quasi-shuffle algebra.
\item  The family $\{\check\Sigma_l\}_{l\in\Lyn Y}$ forms a transcendence basis of the quasi-shuffle algebra.
\end{enumerate}
\end{theorem}

\begin{proof}
The family $\{\Sigma_l\}_{l\in\Lyn Y}$ of upper triangular polynomials is free and then, by a theorem of Viennot, we get the first result. The second is a direct consequence of the Poincar\'e-Birkhoff-Witt theorem. By the Cartier-Quillen-Milnor-Moore theorem, we get the third one and the last one is also obtained as consequence of the constructions of the families $\{\check\Sigma_l\}_{l\in\Lyn Y}$ and $\{\check\Sigma_w\}_{w\in Y^*}$ of lower triangular polynomials.
\end{proof}

Now, let us clarify the basis $\{\check\Sigma_w\}_{w\in Y^*}$ and then the transcendence basis $\{\check\Sigma_l\}_{l\in\Lyn Y}$ of the quasi-shuffle algebra $(\QY,\stuffle)$ as follows

\begin{theorem}\label{checkSigma}
We have
\begin{enumerate}
\item  For $w=1_{Y^*}$, $\check\Sigma_{w}=1$.
\item  For any $w=l_1^{i_1}\ldots l_k^{i_k}$, with $l_1,\ldots,l_k\in\Lyn Y$ and $l_1>\ldots>l_k$,
\begin{eqnarray*}
\check\Sigma_w&=&\Frac{\check\Sigma_{l_1}^{\stuffle i_1}\stuffle\ldots\stuffle\check\Sigma_{l_k}^{\stuffle i_k}}{i_1!\ldots i_k!}.
\end{eqnarray*}
\item  For any $y\in Y,\check\Sigma_{y}=\pi_1^*(y)$.
\item  For any $l=yu\in\Lyn Y-Y$, $\check\Sigma_l=y\check\Sigma_u$.
\end{enumerate}
\end{theorem}

\begin{proof}
\begin{enumerate}
\item  Since $\Sigma_{1_{Y^*}}=1$ then $\check\Sigma_{1_{Y^*}}=1$.

\item  Let $u=u_1\ldots u_n=l_1^{i_1}\ldots l_k^{i_k},v=v_1\ldots v_m=h_1^{j_1}\ldots h_p^{j_p}$ with
$l_1\ldots,l_k,\allowbreak h_1,\ldots,h_p,\allowbreak u_1,\ldots,u_n,\allowbreak v_1,\ldots,v_m\in\Lyn Y,\allowbreak l_1>\ldots>l_k,\allowbreak h_1>\ldots>h_p,\allowbreak u_1\ge\ldots\ge u_n,\allowbreak v_1\ge\ldots\ge v_m$ 
and  $i_1+\ldots+i_k=n,\allowbreak j_1+\ldots+j_p=m$. Hence, if $m\ge2$ (resp. $n\ge2$) then $v\notin\Lyn Y$ (resp. $u\notin\Lyn Y$).

Since
\begin{eqnarray*}
\scal{\check\Sigma_{u_1}\stuffle\ldots\stuffle\check\Sigma_{u_n}}{\prod_{i=1}^n\Sigma_{u_i}}
&=&\scal{\check\Sigma_{u_1}\otimes\ldots\otimes\check\Sigma_{u_n}}{\Delta_{\ministuffle}^{(n-1)}(\Sigma_{v_1}\ldots\Sigma_{v_m})}
\end{eqnarray*}
then many cases occur~:
\begin{enumerate}
\item Case $n>m$. By Lemma \ref{L3}(\ref{1}),
$\scal{\check\Sigma_{u_1}\stuffle\ldots\stuffle\check\Sigma_{u_n}}{\Sigma_{v_1}\ldots\Sigma_{v_m}}=0.$

\item Case $n=m$. By Lemma \ref{L3}(\ref{2}), one has
\begin{eqnarray*}
\scal{\check\Sigma_{u_1}\stuffle\ldots\stuffle\check\Sigma_{u_n}}{\prod_{i=1}^n\Sigma_{v_i}}
&=&\sum_{\sigma\in\check\Sigma_n}\prod_{i=1}^n\scal{\check\Sigma_{u_i}}{\Sigma_{v_{\sigma(i)}}}\cr
&=&\sum_{\sigma\in\check\Sigma_n}\prod_{i=1}^n\delta_{\check\Sigma_{u_i},\Sigma_{v_{\sigma(i)}}}.
\end{eqnarray*}
Thus, if $u\neq v$ then $(u_1,\ldots,u_n)\neq(v_1,\ldots,v_n)$ then the second member is vanishing else,
{\it i.e.} $u=v$, the second member equals $1$ because the factorization by Lyndon words is unique.

\item Case $n<m$. By Lemma \ref{L3}(\ref{3}), let us consider the following language over the alphabet
$\calA=\{\Sigma_{v_1},\ldots,\Sigma_{v_m}\}$~:
\begin{eqnarray*}
\mathcal{M}&=&\{w\in\calA^*|w=\Sigma_{v_{j_1}}\ldots\Sigma_{v_{j_{|w|}}},j_1<\ldots<j_{|w|},|w|\ge1\},
\end{eqnarray*}
and the morphism $\mu:\Q\langle\calA\rangle\longrightarrow\QY$. We get~:
\begin{eqnarray*}
\scal{\stuffle_{i=1}^n\check\Sigma_{u_i}}{\prod_{i=1}^n\Sigma_{u_i}}
&=&\sum_{{{w_1},\ldots,{w_m}\in\mathcal{M}\atop|{w_1}|+\ldots+|{w_n}|=m}
\atop\forall i,j=1,..m,{\rm alp}({w_i})\cap{\rm alp}({w_j})=\emptyset}
\prod_{i=1}^n\scal{\check\Sigma_{u_i}}{\mu(w_i)}\\
&=&0.
\end{eqnarray*}
Because in this product, on one hand, there exists at least one $w_i\in\mathcal{M}$, $|w_i|\ge2$, corresponding to
$\Sigma_{v_{j_1}}\ldots\Sigma_{v_{j_{|w_i|}}}=\mu(w_i)$ such that
$v_{j_1}\ge\ldots\ge v_{j_{|w_i|}}$ and on other hand, 
$\nu_i:=v_{j_1}\ldots v_{j_{|w_i|}}\notin\Lyn Y$ and $u_i\in\Lyn Y$.
\end{enumerate}
By consequent,
\begin{eqnarray*}
\scal{\check\Sigma_{u}}{\Sigma_v}
=\scal{\frac{\check\Sigma_{l_1}^{\stuffle i_1}\stuffle\ldots\stuffle\check\Sigma_{l_k}^{\stuffle i_k}}{i_1!\ldots i_k!}}
{\Sigma_{h_1}^{j_1}\ldots\Sigma_{h_p}^{j_p}}
=\delta_{u,v}.
\end{eqnarray*}
\item For any $l\in Y,\Sigma_{l}=\pi_1(l),\check\Sigma_{l}=\pi_1^*(l)$ and $\pi_1,\pi_1^*$ are mutually adjoint.
\item By decomposing $bu$ in the PBW-Lyndon basis, we get on one hand,
\begin{eqnarray*}
bu&=&\sum_{w\in Y^*}\scal{\check\Sigma_{w}}{bu}\Sigma_w+\sum_{aw\in\Lyn Y}\scal{\check\Sigma_{aw}}{bu}\Sigma_{aw}\\
&+&{\mbox{sum of decreasinge products of length $\ge2$, of Lie polynomials}.}
\end{eqnarray*}
Since for any $aw=l_1^{i_1}\ldots l_k^{i_k}\notin\Lyn Y$ ($k\ge2$) with $l_1,\ldots,l_k\in\Lyn Y$ and $l_1>\ldots>l_k$, then $\Sigma_{aw}=\Sigma_{l_1}\ldots\Sigma_{l_k}$. Since (see Proposition \ref{Ptriangulaire})
\begin{eqnarray*}
\Sigma_{w}=w+\sum_{v>w,(v)=(w)}c_vv
&\mbox{and}&
\Sigma_{bw}= bw+\sum_{v>bw,(v)=(bw)}d_vv
\end{eqnarray*}
then, by decomposing $u$ in the PBW-Lyndon basis and then by multiplying by $b$, we get on other hand,
\begin{eqnarray*}
bu&=&\sum_{w\in Y^*}\scal{\check\Sigma_w}{u}b\Sigma_w\cr
&=&\sum_{w\in Y^*}\scal{\check\Sigma_w}{u}\biggl(\Sigma_{bw}-\sum_{v>bw,(v)=(bw)}d_vv\biggr)\cr
&=&\sum_{w\in Y^*}\scal{\check\Sigma_w}{u}\Sigma_{bw}-
\sum_{w\in Y^*}\sum_{v>bw,(v)=(bw)}d_v\sum_{w'\in Y^*}\scal{\check\Sigma_{w'}}{v}\Sigma_{w'}\cr
&&(\mbox{by decomposing $v$ in PBW-Lyndon basis})\cr
&=&\sum_{w\in Y^*}\scal{\check\Sigma_w}{u}\Sigma_{bw}
-\sum_{w\in Y^*}\sum_{w'\in Y^*}\sum_{v>bw,(v)=(bw)}d_v\scal{\check\Sigma_{w'}}{v}\Sigma_{w'}\cr
&=&\sum_{bw\in\Lyn Y}\scal{\check\Sigma_w}{u}\Sigma_{bw}-
\sum_{w'\in\Lyn Y}\sum_{w\in Y^*}\sum_{v>bw,(v)=(bw)}d_v\scal{\check\Sigma_{w'}}{v}\Sigma_{w'}\cr
&+&{\mbox{sum of decreasinge products, of length $\ge2$, of Lie polynomials}.}\label{eq2}
\end{eqnarray*}
After splitting these two sums on two disjoint supports, one has
\begin{itemize}
\item for any $bw=l_1^{i_1}\ldots l_n^{i_n}\notin\Lyn Y$ ($n\ge2$) with  $l_1,\ldots,l_n\in\Lyn Y$
verifying $l_1>\ldots>l_n$ and  $\Sigma_{aw}=\Sigma_{l_1}^{i_n}\ldots\Sigma_{l_n}^{i_n}$.
\item for any $w'=\lambda_1^{j_1}\ldots\lambda_m^{j_m}\notin\Lyn Y$ ($m\ge2)$ with 
$\lambda_1,\ldots,\lambda_m\in\Lyn Y$ verifying $\lambda_1>\ldots>\lambda_m$ and $\Sigma_{w'}=\Sigma_{\lambda_1}^{j_1}\ldots\Sigma_{\lambda_m}^{j_m}$.
\end{itemize}
In the second sum, since each word $v$ is great than the Lyndon word $bw$ then the Lie polynomial  $\Sigma_{bw}$ does not appear in the decomposition, in the PBW-Lyndon basis, of $v$. More precisely,
(see Proposition \ref{Ptriangulaire})
\begin{eqnarray*}
\check\Sigma_{w'}&=&w'+\sum_{w'>v',(w')=(v')}e_{v'}v'\quad\mbox{with }e_{v'}\ge0,\\
\Rightarrow\quad
\scal{\check\Sigma_{w'}}{v}&=&\scal{w'}{v}+\sum_{w'>v',(w')=(v')}e_{v'}\scal{v'}{v}.
\end{eqnarray*}
In particular (for $w'=bw\in\Lyn Y$), the coefficient of the Lie polynomial $\Sigma_{bw}$ in the decomposition of $v$ ($>bw$) is vanishing~:
\begin{eqnarray*}
\scal{\check\Sigma_{bw}}{v}
=\scal{bw}{v}+\sum_{v>bw>v',(bw)=(v')}e_{v'}\scal{v'}{v}
=0.
\end{eqnarray*}
Thus, by identifying the coefficients in these two expressions of Lyndon word $bu$, one has
$\scal{\check\Sigma_{aw}}{bu}=\delta_{a,b}\scal{\check\Sigma_{w}}{u}.$ In other words, $\check\Sigma_{bw}=b\check\Sigma_{w}$.
\end{enumerate}
\end{proof}

\begin{corollary}\label{corollary}
\begin{enumerate}
\item For $w\in Y^+$, the polynomial $\check\Sigma_w$ is proper and homogenous of degree $|w|$, for $\deg(y_i)=i$, and of rational positive coefficients.
\item 
\begin{eqnarray*}
\sum_{w\in Y^*}w\otimes w=\sum_{w\in Y^*}\check\Sigma_w\otimes\Sigma_w
=\prod_{l\in\Lyn Y}^{\searrow}\exp(\check\Sigma_l\otimes\Sigma_l).
\end{eqnarray*}
\item  The family $\Lyn Y$ forms a transcendence basis\footnote{This result is an analogous of a Radford theorem (see \cite{reutenauer}). Thus the bases $\Lyn Y$ and $\{\check\Sigma_l\}_{l\in\Lyn Y}$ belong to the class of Radford bases, {\it i.e.} the class of trancensdence bases, of the quasi-shuffle algebra, as well as the bases $\Lyn X$ and $\{S_l\}_{l\in\Lyn X}$ belong to the class of Radford bases  of the shuffle algebra.} of the quasi-shuffle algebra
and the family of proper polynomials of rational positive coefficients defined by,
for any $w=l_1^{i_1}\ldots l_k^{i_k}$ with $l_1>\ldots>l_k$ and $l_1,\ldots,l_k\in\Lyn Y$,
\begin{eqnarray*}
\chi_w&=&\frac{l_1^{\stuffle i_1}\stuffle\ldots\stuffle l_k^{\stuffle i_k}}{i_1!\ldots i_k!}
\end{eqnarray*}
forms a basis of the quasi-shuffle algebra.
\item  Let $\{\xi_w\}_{w\in Y^*}$ be the basis of the envelopping algebra $\calU(\LQX)$ obtained by duality with the basis $\{\chi_w\}_{w\in Y^*}$~:
\begin{eqnarray*}
\forall u,v\in Y^*,&&\scal{\chi_{v}}{\xi_u}=\delta_{u,v}.
\end{eqnarray*}
Then the family $\{\xi_l\}_{l\in\Lyn Y}$ forms  a  basis of the free Lie algebra $\LQY$.
\end{enumerate}
\end{corollary}

\begin{proof}
\begin{enumerate}
\item The proof can be done by induction on the length of $w$ using the fact that the product $\stuffle$ conserve the property, l'homogenity and rational positivity of the coefficients.

\item  Expressing $w$ in  the basis $\{\check\Sigma_w\}_{w\in Y^*}$ of the quasi-shuffle algebra and then in the basis $\{\Sigma_w\}_{w\in Y^*}$ of the envelopping algebra, we obtain successively
\begin{eqnarray*}
\sum_{w\in Y^*}w\otimes w
&=&\sum_{w\in Y^*}\biggl(\sum_{u\in X^*}\scal{\Sigma_u}{w}\check\Sigma_u\biggr)\otimes w\cr
&=&\sum_{u\in Y^*}\check\Sigma_u\otimes\biggl(\sum_{w\in X^*}\scal{\Sigma_u}{w}w\biggr)\cr
&=&\sum_{u\in Y^*}\check\Sigma_u\otimes\Sigma_u\cr
&=&\sum_{l_1>\ldots>l_k\atop i_1,\ldots,i_k\ge1}
\Frac{\check\Sigma_{l_1}^{\shuffle i_1}\stuffle\ldots\stuffle\check\Sigma_{l_k}^{\shuffle i_k}}{i_1!\ldots i_k!}
\otimes\Sigma_{l_1}^{i_1}\ldots\Sigma_{l_k}^{i_k}\cr
&=&\prod_{l\in\Lyn Y}^{\searrow}\sum_{i\ge0}\frac{\check\Sigma_l^{\stuffle i}}{i!}\otimes\Sigma_l^{i}\cr
&=&\prod_{l\in\Lyn Y}^{\searrow}\exp(\check\Sigma_l\otimes\Sigma_l).
\end{eqnarray*}

\item  For  $w=l_1^{i_1}\ldots l_k^{i_k}$ with $l_1,\ldots,l_k\in\Lyn Y$ and $l_1>\ldots>l_k$, by Proposition \ref{Ptriangulaire}, the  polynomial  of rational positive coefficients $\check\Sigma_w$ is lower triangular~:
\begin{eqnarray*}
\check\Sigma_w
=\frac{\check\Sigma_{l_1}^{\stuffle i_1}\stuffle\ldots\stuffle\check\Sigma_{l_k}^{\stuffle i_k}}{i_1!\ldots i_k!}
=w+\sum_{v<w,(v)=(w)}c_vv.
\end{eqnarray*}
In particular, for any $l_j\in\Lyn Y$, $\check\Sigma_{l_j}$ is lower triangular~:
\begin{eqnarray*}
\check\Sigma_{l_j}&=&l_j+\sum_{v<l_j,(v)=(l_j)}c_vv.
\end{eqnarray*}
Hence, $\check\Sigma_w=\chi_w+\chi'_w$,
where $\chi'_w$ is a proper polynomial of $\QY$ of rational positive coefficients.
We deduce then the support of $\chi_w$ contains words which are less than $w$ and $\scal{\chi_w}{w}=1$.
Thus, the proper polynomial $\chi_w$ of rational positive coefficients is lower triangular~:
\begin{eqnarray*}
\chi_w&=&w+\sum_{v<w,(v)=(w)}c_vv,\\
\Rightarrow\quad\forall l\in\Lyn Y,\quad\chi_l&=&l+\sum_{v<l,(v)=(l)}c_vv.
\end{eqnarray*}
It follows then expected results.

\item  By duality, for $w\in Y^*$, the proper polynomial  $\xi_w$ is upper triangular.
In particular, for any $l\in\Lyn Y$, the proper polynomial $\xi_l$ is upper triangular~:
\begin{eqnarray*}
\xi_l&=&l+\sum_{v>l,(v)=(l)}d_vv.
\end{eqnarray*}
Hence, the family  $\{\xi_l\}_{l\in\Lyn Y}$ is free and its elements verify an analogous of the generalized criterion of Friedrichs~:
\begin{itemize}
\item for $w\in\Lyn Y$, one has $\scal{\chi_w}{\xi_l}=\delta_{w,l}$,
\item for $w\notin\Lyn Y$, $w=l_1\ldots l_n$ with  $l_1,\ldots,l_n\in\Lyn Y$ and $l_1>\ldots>l_n$, one has
$\scal{\chi_w}{\xi_l}=\scal{\chi_{l_1}\stuffle\ldots\stuffle\chi_{l_n}}{\xi_l}=0$.
\end{itemize}
Moreover, the polynomials $\xi_l$'s are primitive~: by Corollary \ref{corollary}(3), one has
\begin{eqnarray*}
\Delta_{\ministuffle}(\xi_l)
&=&\sum_{u,v\in Y^*}\scal{u\stuffle v}{\xi_l}u\otimes v\cr
&=&\sum_{u\in Y^+}\scal{u\stuffle1_{Y^*}}{\xi_l}u\otimes1_{Y^*}
+\sum_{v\in Y^+}\scal{1_{Y^*}\stuffle v}{\xi_l}1_{Y^*}\otimes v\cr
&+&\sum_{u,v\in Y^+}\scal{u\stuffle v}{\xi_l}u\otimes v+\scal{1_{Y^*}\stuffle1_{Y^*}}{\xi_l}1_{Y^*}\otimes1_{Y^*}\cr
&=&\xi_l\otimes1_{Y^*}+1_{Y^*}\otimes\xi_l.
\end{eqnarray*}
Because, after decomposing $u$ and $v$ on the basis $\{\chi_l\}_{l\in\Lyn Y}$ and by the previous criterion, the third term is vanishing. The last one is also vanishing since the $\xi_l$'s are proper. By a theorem of Viennot, we obtain then the expected result.
\end{enumerate}
\end{proof}

\subsection{Proof of Theorem \ref{Hfactorization}}
Applying the tensor product of isomorphisms $\H\otimes\mathrm{Id}$  (Proposition \ref{isomorphisms}) on the diagonal series (Corollary \ref{corollary}(ii)), the infinite factorization, by Lyndon words, of the noncommutative generating series of harmonic sums follows\footnote{This proof omitted in previous versions uses mainly the results presented in this annexe that have not been published earlier but have already been presented at various workshops.
It is an analogous way to obtain the infinite factorization, by Lyndon words over the alphabet $X$, of the noncommutative generating series of polylogarithms (see Theorem \ref{factorisationL}) by applying the tensor product of isomorphisms $\Li\otimes\mathrm{Id}$  (see Proposition \ref{isomorphisms}) on the diagonal series, over $X$.}~:
\begin{eqnarray}
\H(N)=\sum_{w\in Y^*}\H_w(N)\;w=\prod_{l\in\Lyn Y}^{\searrow}\exp(\H_{\check\Sigma_l}(N)\;\Sigma_l).
\end{eqnarray}

\section{Annexe B~: differential realization}
To facilitate reading, the following results are placed in this Annex which can be skipped by readers already familiar with the techniques developed by Fliess (and adapted by us for studies in this paper).
\subsection{Polysystem and convergence criterion}\label{convergence}

\subsubsection{Serial estimates from above}
Here,  generalizing a little, $\K$ is supposed a $\C$-algebra
and a complete normed vector space equipped with a norm denoted by $\absv{.}$. 

For any $n\in\N, X^{\ge n}$ denotes the set of words over $X$ of length greater than or equal to $n$.
The set of formal power series (resp. polynomials) on $X$, is denoted by $\KXX$ (resp. $\KX$).

\begin{definition}[\cite{these,cade}]\label{xi-em}\label{chi-gc}
Let $\xi,\chi$ be real positive functions over $ X^*$.
Let $S\in\KXX$.
\begin{enumerate}
\item $S$ will be said {\em $\xi-$exponentially bounded from above} if it verifies
\begin{eqnarray*}
\exists K\in\R_+,\exists n\in\N,\forall w\in X^{\ge n},&&
\absv{\langle S\bv w\rangle}\le K{\xi(w)}/{\abs w!}.
\end{eqnarray*}

We denote by $\K^{\xi-\mathrm{em}}\ser{X}$ the set of formal power series
in $\KXX$ which are $\xi-$exponentially bounded from above.
\item $S$ verifies the {\em $\chi-$growth condition} if it satisfies
\begin{eqnarray*}
\exists K\in\R_+,\exists n\in\N,\forall w\in X^{\ge n},&&
\absv{\langle S\bv w\rangle}\le K\chi(w)\abs w!.
\end{eqnarray*}
We denote by $\K^{\chi-\mathrm{gc}}\ser{X}$ the set of formal power series
in $\KXX$ verifying the $\chi-$growth condition.
\end{enumerate}
\end{definition}

\begin{lemma}\label{R}
We have
\begin{eqnarray*}
R=\sum_{w\in X^*}\abs{w}!\;w
&\Rightarrow&
\langle R^{\minishuffle 2}\bv w\rangle
=\sum_{u,v\in X^*\atop\supp(u\shuffle v)\ni w}\abs{u}!\abs{v}!\le2^{\abs{w}}\abs{w}!.
\end{eqnarray*}
\end{lemma}

\begin{proof}
One has
\begin{eqnarray*}
\sum_{u,v\in X^*\atop\supp(u\shuffle v)\ni w}\abs{u}!\abs{v}!
&=&\sum_{k=0}^{\abs w}\sum_{\abs u=k,\abs v=\abs w-k\atop\supp(u\shuffle v)\ni w}
k!(\abs{w}-k)!\\
&=&\sum_{k=0}^{\abs w}{\abs w\choose k}k!(\abs{w}-k)!\\
&=&\sum_{k=0}^{\abs w}\abs w!\\
&=&(1+\abs w)\abs w!.
\end{eqnarray*}
By induction on the length of $w$, one has $1+|w|\le2^{|w|}$. It follows the expected result.
\end{proof}

\begin{proposition}\label{shufflegrowth}
Let $S_1$ and $S_2$ verifying the growth condition.
Then $S_1+S_2$ and $S_1\shuffle S_2$ also verifies  the growth condition.
\end{proposition}

\begin{proof}
The proof for $S_1+S_2$ is immediate.

Next, since $\absv{\langle S_i\bv w\rangle}\le K_i\chi_i(w)\abs{w}!$,
for $i=1$ or $2$ and for $w\in X^*$, then\footnote{$\langle S_1\shuffle S_2\bv w\rangle$
is the coefficient of the word $w$ in the power series $ S_1\shuffle S_2$.}
\begin{eqnarray*}
\langle S_1\shuffle S_2\bv w\rangle&=&\sum_{\supp(u\shuffle v)\ni w}
\langle S_1\bv u\rangle\langle S_2\bv v\rangle,\\
\Rightarrow\quad\absv{\langle S_1\shuffle S_2\bv w\rangle}
&\le&K_1K_2\sum_{u,v\in X^*\atop\supp(u\shuffle v)\ni w}(\chi_1(u)\abs{u}!)(\chi_2(v)\abs{v}!).
\end{eqnarray*}
Let $K=K_1K_2$ and let $\chi$ be a real positive function over $ X^*$ such that
\begin{eqnarray*}
\forall w\in X^*,&&
\chi(w)=\max\{\chi_1(u)\chi_2(v)\bv u,v\in X^*\ \mbox{and}\ \supp(u\shuffle v)\ni w\}.
\end{eqnarray*}
With the notations in Lemma \ref{R}, we get
\begin{eqnarray*}
\absv{\langle S_1\shuffle S_2\bv w\rangle}&\le&
K\chi(w)\langle R^{\minishuffle 2}\bv w\rangle.
\end{eqnarray*}
Hence, $S_1\shuffle S_2$ verifies the $\chi'$-growth condition with $\chi'$ defined as $\chi'(w)=2^{\abs{w}}\chi(w)$.
\end{proof}

\begin{definition}[\cite{these,cade}]
Let $\xi$ be a real positive function defined over $X^*$, $S$ will be said {\em $\xi$-exponentially continuous} if it is
continuous over $\K^{\xi-\mathrm{em}}\ser{X}$. The set of formal power series which are $\xi$-exponentially continuous is denoted by $\K^{\xi-ec}\ser{X}$ .
\end{definition}

\begin{lemma}[\cite{these,cade}]\label{exponentiallycontinuous}
For any real positive function $\xi$ defined over $ X^*$, we have
$\KX\subset\K^{\xi-ec}\ser{X}$.
Otherwise, for $\xi=0$, we get
$\KX=\K^{0-ec}\ser{X}$.
Hence, any polynomial is $0-$exponentially continuous.
\end{lemma}

\begin{proposition}[\cite{these,cade}]
Let $\xi,\chi$ be a real positive functions over $ X^*$ and let $P\in\KX$.
\begin{enumerate}
\item Let $S\in\K^{\xi-\mathrm{em}}\ser{X}$.
The right residual of $S$ by $P$ belongs to $\K^{\xi-\mathrm{em}}\ser{X}$.
\item Let $R\in\K^{\chi-\mathrm{gc}}\ser{X}$.
The concatenation $SR$ belongs to $\K^{\chi-\mathrm{gc}}\ser{X}$.
\end{enumerate}
\end{proposition}

\begin{proof}
\begin{enumerate}
\item Since $S\in\K^{\xi-\mathrm{em}}\ser{X}$ then 
\begin{eqnarray*}
\exists K\in\R_+,\exists n\in\N,\forall w\in X^{\ge n},&&
\absv{\langle S\bv w\rangle}\le K{\xi(w)}/{\abs w!}.
\end{eqnarray*}
If $u\in\supp(P):=\{w\in X^*\bv\langle P\bv w\rangle\neq0\}$
then, for any $w\in X^*$, one has $\langle S\triangleright u\bv w\rangle=\langle S\bv uw\rangle$
and $S\triangleright u$ belongs to $\K^{\xi-\mathrm{em}}\ser{X}$~:
\begin{eqnarray*}
\exists K\in\R_+,\exists n\in\N,\forall w\in X^{\ge n},&&
\absv{\langle S\triangleright u\bv w\rangle}\le[K\xi(u)]{\xi(w)}/{\abs w!}.
\end{eqnarray*}
It follows then $S\triangleright P$ is $\K^{\xi-\mathrm{em}}\ser{X}$ by taking
$K_1=K\max_{u\in\supp(P)}\xi(u)$.
\item Since $R\in\K^{\chi-\mathrm{gc}}\ser{X}$ then
\begin{eqnarray*}
\exists K\in\R_+,\exists n\in\N,\forall w\in X^{\ge n},&&
\absv{\langle S\bv w\rangle}\le K\chi(w)\abs w!.
\end{eqnarray*}
Let $v\in\supp(P)$ such that $v\neq\epsilon$.
Since, for any $w\in X^*$, $Rv$ belongs to $\K^{\chi-\mathrm{gc}}\ser{X}$
and one has $\langle Rv\bv w\rangle=\langle R\bv v\resg w\rangle$~:
\begin{eqnarray*}
\exists K\in\R_+,\exists n\in\N,\forall w\in X^{\ge n},\quad
\absv{\langle R\bv v\resg w\rangle}&\le&K\chi(v\resg w)(\abs w-\abs v)!\\
&\le&{K}\abs w{\chi(w)}/{\chi(v)}.
\end{eqnarray*}
Note that if $v\resg w=0$ then $\langle Rv\bv w\rangle=0$ and the previous conclusion holds.
It follows then $RP$ is $\K^{\chi-\mathrm{gc}}\ser{X}$ by taking
$K_2=K\min_{v\in\supp(P)}\chi(v)^{-1}$.
\end{enumerate}
\end{proof}

\begin{proposition}[\cite{these,cade}]\label{convergencecriterion}
Two real positive morphisms over $X^*$, $\xi$ and $\chi$ are assumed to verify the condition
\begin{eqnarray*}
\sum_{x\in X}\chi(x)\xi(x)&<&1.
\end{eqnarray*}
Then for any $F\in\K^{\chi-\mathrm{gc}}\ser{X}$, $F$ is continuous over $\K^{\xi-\mathrm{em}}\ser{X}$.
\end{proposition}

\begin{proof}
If $\xi,\chi$ verify the upper bound condition then the following power series
\begin{eqnarray*}
\sum_{w\in X^*}\chi(w)\xi(w)&=&\biggl(\sum_{x\in X}\chi(x)\xi(x)\biggr)^*
\end{eqnarray*}
is well defined. If $F\in\K^{\chi-\mathrm{gc}}\ser{X}$ and $C\in\K^{\xi-\mathrm{em}}\ser{X}$ then there
exists $K_i\in\R_+$ and $n_i\in\N$ such that for any $w\in X^{\ge n_i},i=1,2$,
one has
\begin{eqnarray*}
\absv{\langle F\bv w\rangle}\le K_1\chi(w)\abs w!&\mbox{and}&
\absv{\langle C\bv w\rangle}\le K_2{\xi(w)}/{\abs w!}.
\end{eqnarray*}
Hence,
\begin{eqnarray*}
&&\forall w\in X^*,\abs w\ge\max\{n_1,n_2\},\quad
\absv{\langle F|w\rangle\langle C|w\rangle}\le K_1K_2\chi(w)\xi(w),\\
\Rightarrow&&\sum_{w\in X^*}\absv{\langle F|w\rangle\langle C|w\rangle}
\le K_1K_2\sum_{w\in X^*}\chi(w)\xi(w)=K_1K_2\biggl(\sum_{x\in X}\chi(x)\xi(x)\biggr)^*.
\end{eqnarray*}
\end{proof}

\subsubsection{Upper bounds {\it \`a la} Cauchy}

Let $q_1,\ldots,q_n$ be commutative indeterminates over $\C$.
The algebra of formal power series (resp. polynomials) over 
$\{q_1,\ldots,q_n\}$ with coefficients in $\C$ is denoted
by $\C[\![q_1,\ldots,q_n]\!]$ (resp. $\C[q_1,\ldots,q_n]$).

\begin{definition}[\cite{these,cade}]
Let
\begin{eqnarray*}
f=\sum_{i_1,\ldots,i_n\ge0}
f_{i_1,\ldots,i_n}q_1^{i_1}\ldots q_n^{i_n}\in\C[\![q_1,\ldots,q_n]\!].
\end{eqnarray*}
We set
\begin{eqnarray*}
E(f)
 &:=&\{\rho\in\R_+^n:\exists C_f\in\R_+\mbox{ s.t. }
 \forall i_1,\ldots,i_n\ge0,\abs{f_{i_1,\ldots,i_n}}\rho_1^{i_1}\ldots\rho_n^{i_n}\le C_f\}\cr
 {\breve E}\!(f)&:&\mbox{interior of $E(f)$ in }\R^n.\cr
\Conv(f)&:=&\{q\in\C^n:(\abs{q_1},\ldots,\abs{q_n})\in{\breve E}\!(f)\}\quad:\quad
\mbox{convergence domain of }f.
 \end{eqnarray*}
The power series $f$ is to be said {\em convergent} if
$\Conv(f)\ne\emptyset$. Let $\calU$ be an open domain in $\C^n$ and
let $q\in\C^n$. The power series $f$ is to be said convergent on
$q$ (resp. over $\calU$) if $q\in\Conv(f)$ (resp. $\calU\subset \Conv(f)$). We set
\begin{eqnarray*}
\C^{\rm cv}[\![q_1,\ldots,q_n]\!]&=&\{f\in\C[\![q_1,\ldots,q_n]\!]:\Conv(f)\ne\emptyset\}.
 \end{eqnarray*}

Let $q\in\Conv(f)$. There exists the constants $C_f,\rho$ and ${\bar\rho}$ such that
$\abs{q_1}<{\bar\rho}<\rho,\ldots,\abs{q_n}<{\bar\rho}<\rho$
and, for $i_1\ldots,i_n\ge0$,
$\abs{f_{i_1,\ldots,i_n}}\rho_1^{i_1}\ldots\rho_n^{i_n}\le C_f$.
The {\em convergence modulus} of $f$ at $q$ is $(C_f,\rho,{\bar\rho})$.
\end{definition}

Suppose that $\Conv(f)\ne\emptyset$ and let $q\in\Conv(f)$.
If $(C_f,\rho,{\bar\rho})$ is a convergence modulus of $f$ at $q$ then
$\abs{f_{i_1,\ldots,i_n}q_1^{i_1}\ldots q_n^{i_n}}\le
C_f({\bar\rho_1}/{\rho_1})^{i_1}\ldots({\bar\rho_1}/{\rho_1})^{i_n}.$
Hence, at $q$, the power series $f$ is majored termwise by
\begin{eqnarray}
C_f\prod_{k=0}^m\biggl(1-\frac{\bar\rho_k}{\rho_k}\biggr)^{-1}.
 \end{eqnarray}
Hence, $f$ is uniformly absolutely convergent in
$\{q\in\C^n:\abs{q_1}<\;{\bar\rho},\ldots,\abs{q_n}<{\bar\rho}\}$
which is an open domain in $\C^n$. Thus, $\Conv(f)$ is an open domain in $\C^n$.
Since the partial derivation of order $j_1,\ldots,j_n\ge0$
of $f$ is estimated by 
\begin{eqnarray}
\absv{D^{j_1}_1\ldots D^{j_n}_nf}&\le&C_f
\frac{\partial^{j_1+\ldots+j_n}}{\partial{\bar\rho}^{j_1+\ldots+j_n}}
\prod_{k=0}^m\biggl(1-\frac{\bar\rho_k}{\rho_k}\biggr)^{-1}.
\end{eqnarray}

\begin{proposition}[\cite{these}]
We have $\Conv(f)\subset\Conv(D^{j_1}_1\ldots D^{j_n}_nf)$.
\end{proposition}

Let $f\in\C^{\rm cv}[\![q_1,\ldots,q_n]\!]$ and let
$\{A_i\}_{i=0,1}$ be a polysystem defined as follows
\begin{eqnarray}\label{vectorfield}
A_i(q)=\sum_{j=1}^nA_i^j(q)\frac{\partial}{\partial q_j},
&\mbox{with}&A_i^j(q)\in\C^{\rm cv}[\![q_1,\ldots,q_n]\!],j=1,\ldots,n.
\end{eqnarray}

\begin{lemma}[\cite{fliessrealisation}]\label{coefficients}
For $i=0,1$ and $j=1,..,n$, one has 
$A_i\circ q_j=A_i^j(q)$.
Thus,
\begin{eqnarray*}
\forall i=0,1,\quad A_i(q)&=&\sum_{j=1}^n(A_i\circ q_j)\frac{\partial}{\partial q_j}.
\end{eqnarray*}
\end{lemma}

\goodbreak

Let $(\rho,\bar\rho,C_f),\{(\rho,\bar\rho,C_i)\}_{i=0,1}$
be respectively the convergence modulus at
\begin{eqnarray}
q&\in&\Conv(f)\bigcap_{i=0,1\atop j=1,..,n}\Conv(A_i^j)
\end{eqnarray}
of $f$ and $\{A_i^j\}_{j=1,..,n}$. Let us consider the following monoid morphisms
\begin{eqnarray}
\calA(\epsilon)=\mbox{identity}&\mbox{and}&C(\epsilon)=1,\label{calA1}\\
\forall w=vx_i,x_i\in X,v\in X^*,\quad 
\calA(w)=\calA(v)A_i&\mbox{and}&C(w)=C(v)C_i\label{calA2}.
\end{eqnarray}

\begin{lemma}[\cite{fliess1}]
For any word $w$, $\calA(w)$ is continuous over $\C^{\rm cv}[\![q_1,\ldots,q_n]\!]$
and, for any $f,g\in\C^{\rm cv}[\![q_1,\ldots,q_n]\!]$, one has
\begin{eqnarray*}
\calA(w)\circ(fg)&=&\sum_{u,v\in X^*}
\langle u\shuffle v\bv w\rangle\;(\calA(u)\circ f)(\calA(v)\circ g).
\end{eqnarray*}
\end{lemma}
These notations are extended, by linearity, to $\KX$ and  we will denote
$\calA(w)\circ f_{|q}$ the evaluation of $\calA(w)\circ f$ at $q$.

\begin{definition}[\cite{fliess1}]\label{fliess}
Let $f\in\C^{\rm cv}[\![q_1,\ldots,q_n]\!]$.
The generating series of the polysystem $\{A_i\}_{i=0,1}$ and of the observation $f$ is given by
\begin{eqnarray*}
\sigma f&:=&\sum_{w\in X^*}\calA(w)\circ f\;w\quad\in\quad\serie{\C^{\rm cv}[\![q_1,\ldots,q_n]\!]}{X}.\\
\sigma f_{|_{q}}&:=&\sum_{w\in X^*}\calA(w)\circ f_{|q}\;w\quad\in\quad\serie{\C}{X}.
\end{eqnarray*}
The last generating series is called {\em Fliess generating series}
of the polysystem $\{A_i\}_{i=0,1}$ and of the observation $f$ at $q$.
\end{definition}

\begin{lemma}[\cite{fliess1}]\label{sigmamorphism}
Let $\{A_i\}_{i=0,1}$ be a polysystem.
Then, the map
\begin{eqnarray*}
\sigma:(\C^{\rm cv}[\![q_1,\ldots,q_n]\!],.)&\longrightarrow&
(\serie{\C^{\rm cv}[\![q_1,\ldots,q_n]\!]}{X},\shuffle),
\end{eqnarray*}
is an algebra morphism, {\em i.e.} for any $f,g\in\C^{\rm cv}[\![q_1,\ldots,q_n]\!]$ and $\mu,\nu\in\C$, one has
$\sigma(\nu f+\mu h)=\nu\sigma f+\mu\sigma g$ and $\sigma(fg)=\sigma f\;\shuffle\;\sigma g$.
\end{lemma}

\begin{lemma}[\cite{fliessrealisation}]\label{fields}
Let $\{A_i\}_{i=0,1}$ be a polysystem and $f\in\C^{\rm cv}[\![q_1,\ldots,q_n]\!]$. Then
\begin{eqnarray*}
\forall x_i\in X,\quad\sigma(A_i\circ f)
=x_i\triangleleft \sigma f&\in&\serie{\C^{\rm cv}[\![q_1,\ldots,q_n]\!]}{X}\\
\forall w\in X^*,\quad\sigma (\calA(w)\circ f)=
w\triangleleft \sigma f&\in&\serie{\C^{\rm cv}[\![q_1,\ldots,q_n]\!]}{X}.
\end{eqnarray*}
\end{lemma}

\begin{lemma}[\cite{these}]\label{growth}
Let
$\tau=\min_{1\le k\le n}\rho_k$ and $r=\max_{1\le k\le n}{\bar\rho_k}/{\rho_k}$.
We have
\begin{eqnarray*}
\absv{\calA(w)\circ f}&\le&C_f\frac{(n+1)}{(1-r)^n}\Frac{C(w)\abs w!}{{{n+\abs{w}-1}\choose{\abs w}}}
\biggl[\frac{n}{\tau(1-r)^{n+1}}\biggr]^{\abs{w}}\\
&\le&C_f\frac{(n+1)}{(1-r)^n}C(w)
\biggl[\frac{n}{\tau(1-r)^{n+1}}\biggr]^{\abs{w}}\abs w!.
\end{eqnarray*}
\end{lemma}

\begin{theorem}[\cite{these}]\label{growthcondition}
Let $K=C_f{(n+1)}{(1-r)^{-n}}$
and let $\chi$ be the real positive function defined over $X^*$ by
\begin{eqnarray*}
\forall i=0,1,\quad\chi(x_i)&=&\frac{C_in}{\tau(1-r)^{(n+1)}}.
\end{eqnarray*}
Then the generating series $\sigma f$ of the polysystem $\{A_i\}_{i=0,1}$
and of the observation $f$ satisfies the $\chi-$growth condition.
\end{theorem}

It is the same for the Fliess generating series $\sigma f_{|q}$
of the polysystem $\{A_i\}_{i=0,1}$ and of the observation $f$ at $q$.

\subsection{Polysystems and nonlinear differential equation}

\subsubsection{Nonlinear differential equation (with three singularities)}\label{system}

Let us consider the following singular inputs\footnote{These singular inputs are not
included in the studies of Fliess motivated, in particular, by the renormalization of $y(z)$ at $+\infty$ \cite{fliess1,fliessrealisation}.}
\begin{eqnarray}
u_0(z):=z^{-1}&\mbox{and}&u_1(z):=(1-z)^{-1},
\end{eqnarray}
and the following nonlinear dynamical system\footnote{Any differential equation with singularities in $\{a,b,c\}$, via homogra\-phic transformation
${(z-a)(c-b)}{(z-b)^{-1}(c-a)^{-1}},$
can be changed into a differential equation with singularities in $\{0,1,+\infty\}$ (the singularities of homographic transformations belonging to the group $\mathcal{G}$).}
\begin{eqnarray}\label{nonlinear}
\left\{\begin{array}{lcl}
y(z)&=&f(q(z)),\\
\dot q(z)&=&A_0(q)\;u_0(z)+A_1(q)\;u_1(z),\\
q(z_0)&=&q_0,
\end{array}\right.
\end{eqnarray}
where, the state $q=(q_1,\ldots,q_n)$ belongs to the complex analytic manifold of dimension $n$,
$q_0$ is the initial state, the observation $f$ belongs to $\C^{\rm cv}[\![q_1,\ldots,q_n]\!]$
and $\{A_i\}_{i=0,1}$ is the polysystem defined on (\ref{vectorfield}).

\begin{definition}[\cite{hoangjacoboussous}]
The following power series is called {\em transport operator}
of the polysystem $\{A_i\}_{i=0,1}$ and of the observation $f$
\begin{eqnarray*}
\calT&:=&\sum_{w\in X^*}\alpha_{z_0}^z(w)\;\calA(w).
\end{eqnarray*}
\end{definition}
By the factorization of the monoid by Lyndon words, we have \cite{hoangjacoboussous}
\begin{eqnarray}
\calT=(\alpha_{z_0}^z\otimes\calA)\;\biggl(\sum_{w\in X^*}w\otimes w\biggr)
=\prod_{l\in\Lyn X}\exp[\alpha_{z_0}^z(S_l)\;\calA(\check S_l)].
\end{eqnarray}

Let us consider again the Chen generating series $S_{z_0\path z}$ given in (\ref{chen}) of the diffferential forms  involed in $(DE)$ of Example \ref{KZ}, {\it i.e.}
$\omega_0(z)=u_0(z)\;dz$ {and} $\omega_1(z)=u_1(z)\;dz$,
verifying the upper bound conditions given on (\ref{majoration}).

\subsubsection{Asymptotic behaviour of the successive differentiation of the output via extended Fliess fundamental formula}

\begin{theorem}[\cite{cade}]\label{fondamentalformula}
The Fliess fundamental formula can be extended as follows 
\begin{eqnarray*}
y(z)=\calT\circ f_{|_{q_0}}=\sum_{w\in X^*}\langle S_{z_0\path z}\bv w\rangle\langle \calA(w)\circ f_{|_{q_0}}\bv w\rangle
=\langle\sigma f_{|_{q_0}}\bbv S_{z_0\path z}\rangle.
\end{eqnarray*}
\end{theorem}
By the factorization of the Lie exponential series $\L$,
it follows the expansions of the output $y$ of nonlinear dynamical system with singular inputs,
\begin{corollary}[\cite{cade}]\label{output}
\begin{eqnarray*}
y(z)&=&\sum_{w\in X^*}g_w(z)\;\calA(w)\circ f_{|q_0},\\
&=&\sum_{k\ge0}\sum_{n_1,\ldots,n_k\ge0}g_{x_0^{n_1}x_1\ldots x_0^{n_k}x_1}(z)\;
\ad_{A_0}^{n_1}A_1\ldots\ad_{A_0}^{n_k}A_1e^{\log zA_0}\circ f_{|q_0},\\
&=&\prod_{l\in\Lyn X}\exp\biggl(g_{S_l}(z)\;\calA(\check S_l)\circ f_{|q_0}\biggr),\\
&=&\exp\biggl(\sum_{w\in X^*}g_w(z)\;\calA(\pi_1(w))\circ f_{|_{q_0}}\biggr),
\end{eqnarray*}
where, for any word $w$ in $X^*,g_w$ belongs to the polylogarithm algebra.
\end{corollary}

Since $S_{z_0\path z}=\L(z)\L(z_0)^{-1}$ and since $\sigma f_{|_{q_0}}$
and $\L(z_0)^{-1}$ are invariant by $\partial=d/dz$ then
$\partial^ly(z)=\langle\sigma f_{|_{q_0}}\bbv\partial^lS_{z_0\path z}\rangle
=\langle\sigma f_{|_{q_0}}\bbv\partial^l\L(z)\L(z_0)^{-1}\rangle$, for $l\ge0$.

With the notations of Proposition \ref{lem:derivL}, we get
\begin{eqnarray}
\partial^ly(z)=\langle\sigma f_{|_{q_0}}\bbv[P_l(z)\L(z)]\L(z_0)^{-1}\rangle
=\langle\sigma f_{|_{q_0}}\resd P_l(z)\bbv\L(z)\L(z_0)^{-1}\rangle.
\end{eqnarray}

For $z_0=\eps\rightarrow0^+$, the asymptotic behaviour and the renormalization at $z=1$  of $\partial^ly(z)$
(or the asymptotic expansion and the renormalization of its Taylor coefficients at $+\infty$)
are deduced from Proposition \ref{chenregularization} and extend a little  bit the results of \cite{cade} as follows

\begin{corollary}\label{asymptoticofoutput}
For any integer $l$,we have
\begin{eqnarray*}
\partial^ly(1)&{}_{\widetilde{\eps\rightarrow0^+}}&
\langle\sigma f_{|_{q_0}}\resd P_l(1-\eps)\bbv e^{-x_1\log\eps}\;Z_{\minishuffle}\;e^{-x_0\log\eps}\rangle\\
&=&\sum_{w\in X^*}\langle \calA(w)\circ f_{|_{q_0}}\bv w\rangle
\langle P_l(1-\eps)\; e^{-x_1\log\eps}\;Z_{\minishuffle}\;e^{-x_0\log\eps}\bv w\rangle.
\end{eqnarray*}
\end{corollary}

\begin{corollary}
The differentiation of order $l\in\N$ of the output $y$ of the dynamical system (\ref{nonlinear})
is a $\calC$-combination of the elements $g$ belonging to the polylogarithm algebra.
If its ordinary Taylor expansion exists then the coefficients of this expansion
belong to the algebra of harmonic sums and there exists algorithmically computable
coefficients $a_i\in\Z,b_i\in\N$ and $c_i$ belong to the $\C$-algrebra generated
by $\calZ$ and by the Euler's $\gamma$ constant, such that
\begin{eqnarray*}\label{taylor}
\partial^ly(z)=\sum_{n\ge0}y^{(l)}_{n}z^n,&&
y^{(l)}_{n}\;{}_{\widetilde{n\rightarrow\infty}}\;\sum_{i\ge0}c_in^{a_i}\log^{b_i}n.
\end{eqnarray*}
\end{corollary}

\subsection{Differential realization}\label{realization}

\subsubsection{Differential realization}

\begin{definition}
The {\em Lie rank} of a formal power series $S\in\KXX$ is the dimension of the vector space generated by
$\{S\resd\Pi\bv\Pi\in\LKX\}$, or by $\{\Pi\resg S\bv\Pi\in\LKX\}$.
\end{definition}

\begin{definition}\label{annulateur}
Let $S\in\KXX$ and let us put
\begin{eqnarray*}
\Ann(S)&:=&\{\Pi\in\LKX\bv S\resd\Pi=0\},\\
\Ann^{\bot}(S)&:=&\{Q\in(\KXX,\shuffle)\bv Q\resd\Ann(S)=0\}.
\end{eqnarray*}
\end{definition}

It is immediate that $\Ann^{\bot}(S)\ni S$ and it follows that
(see \cite{fliessrealisation,reutenauerrealisation})

\begin{lemma}\label{finitedimension}
Let $S\in\KXX$. If $S$ is of finite Lie rank, $d$,
then the dimension of $\Ann^{\bot}(S)$ equals $d$.
\end{lemma}
By Lemma \ref{residuals}, the residuals
are derivations for shuffle product. Then,
\begin{lemma}\label{stable}
Let $S\in\KXX$. Then~:
\begin{enumerate}
\item For any $Q_1$ and $Q_2\in\Ann^{\bot}(S)$, one has $Q_1\shuffle Q_2\in\Ann^{\bot}(S)$.
\item For any $P\in\KX$ and $Q_1\in\Ann^{\bot}(S)$, one has $P\resg Q_1\in\Ann^{\bot}(S)$.
\end{enumerate}
\end{lemma}

\begin{definition}[\cite{fliessrealisation}]
The formal power series $S\in\KXX$ is {\em differentially produced} if there exists
\begin{itemize}
\item an integer $d$,
\item a power series $f\in\K[\![\bar{q}_1,\ldots,\bar{q}_d]\!]$,
\item a homomorphism $\calA$ from $X^*$ maps to the algebra of differential operators generated by
\begin{eqnarray*}
\calA(x_i)=
\sum_{j=1}^dA_i^j(\bar{q}_1,\ldots,\bar{q}_d)\frac{\partial}{\partial\bar{q}_j},
&\A_i^j(\bar{q}_1,\ldots,\bar{q}_d)\in\K[\![\bar{q}_1,\ldots,\bar{q}_d]\!],
j=1,\ldots,d,
\end{eqnarray*}
such that, for any $w\in X^*$, $\pol{S\bv w}=\calA(w)\circ f_{|_0}$ .
\end{itemize}
The couple $(\calA,f)$ is called {\em differential representation} of $S$ of dimension $d$.
\end{definition}

\begin{proposition}[\cite{reutenauerrealisation}]\label{direct}
Let $S\in\KXX$. If $S$ is differentially produced then 
it verifies the growth condition and its Lie rank is finite.
\end{proposition}

\begin{proof}
Let $(\calA,f)$ be a differential representation of $S$ of dimension $d$.
Then, by the notations of Definition \ref{fliess}, we get
\begin{eqnarray*}
\sigma f_{|_0}=S=\sum_{w\in X^*}(\calA(w)\circ f)_{|_0}\;w.
\end{eqnarray*}
For any $j=1,..,d$, we put
\begin{eqnarray*}
T_j&=&\sum_{w\in X^*}\frac{\partial(\calA(w)\circ f)}{\partial\bar{q}_j}\;w\\
\iff\forall w\in X^*,\quad\langle T_j\bv w\rangle
&=&\frac{\partial(\calA(w)\circ f)}{\partial\bar{q}_j}.
\end{eqnarray*}

Firstly, by Theorem \ref{growthcondition},
the generating series $\sigma f$ verifies the growth condition.
Secondly, for any $\Pi\in\LKX$ and for any $w\in X^*$, one has
\begin{eqnarray*}
\langle\sigma f\resd\Pi\bv w\rangle=\langle\sigma f\bv \Pi w\rangle=\calA(\Pi w)\circ f
=\calA(\Pi)\circ(\calA(w)\circ f).
\end{eqnarray*}
Since $\calA(\Pi)$ is a derivation over $\K[\![\bar{q}_1,\ldots,\bar{q}_d]\!]$~:
\begin{eqnarray*}
\calA(\Pi)
&=&\sum_{j=1}^d(\calA(\Pi)\circ\bar{q}_j)\frac{\partial}{\partial\bar{q}_j},\\
\Rightarrow\quad
\calA(\Pi)\circ(\calA(w)\circ f)
&=&\sum_{j=1}^d(\calA(\Pi)\circ\bar{q}_j)\frac{\partial(\calA(w)\circ f)}{\partial\bar{q}_j}
\end{eqnarray*}
then we deduce that
\begin{eqnarray*}
\forall w\in X^*,\quad\langle\sigma f\resd\Pi\bv w\rangle
&=&\sum_{j=1}^d(\calA(\Pi)\circ\bar{q}_j)\langle T_j\bv w\rangle,\\
\iff\quad\sigma f\resd\Pi
&=&\sum_{j=1}^d(\calA(\Pi)\circ\bar{q}_j)\;T_j
\end{eqnarray*}
That means $\sigma f\resd\Pi$ is $\K$-linear combination 
of $\{T_j\}_{j=1,..,d}$ and the dimension of the vector space
$\mathrm{span}\{\sigma f\resd\Pi\bv\Pi\in\LKX\}$ is less than or equal to $d$.
\end{proof}

\subsubsection{Fliess' local realization theorem}

\begin{proposition}[\cite{reutenauerrealisation}]\label{gcdp}
Let $S\in\KXX$ such that its Lie rank equals $d$. 
Then there exists a basis ${S}_1,\ldots,{S}_d\in\KXX$ of
$(\Ann^{\bot}(S),\shuffle)\cong(\K[\![{S}_1,\ldots,{S}_d]\!],\shuffle)$
such that the ${S}_i$'s are proper and for any $R\in\Ann^{\bot}(S)$, one has
\begin{eqnarray*}
R&=&\sum_{i_1,\ldots,i_d\ge0}\frac{r_{i_1,\ldots,i_n}}{i_1!\ldots i_d!}
{S}_1^{\minishuffle i_1}\shuffle\ldots\shuffle{S}_d^{\minishuffle i_d},
\end{eqnarray*}
where the coefficients $\{r_{i_1,\ldots,i_d}\}_{i_1,\ldots,i_d\ge0}$ belong to $\K$
and $r_{0,\ldots,0}=\langle R\bv\epsilon\rangle$.
\end{proposition}

\begin{proof}
By Lemma \ref{finitedimension}, a such basis exists.
More precisely, since the Lie rank of $S$ is $d$ then there exists
$P_1,\ldots,P_d\in\LKX$ such that
$S\resd P_1,\ldots,S\resd P_d\in(\KXX,\shuffle)$ are $\K$-linearly independent.
By duality, their exists ${S}_1,\ldots,{S}_d\in(\KXX,\shuffle)$ such that
\begin{eqnarray*}
\forall i,j=1,..d,\quad\langle{S}_i\bv P_j\rangle=\delta_{i,j},
&\mbox{and}&
R=\prod_{i=1}^d\exp({S}_i\;P_i).
\end{eqnarray*}
Expending this product, one obtains, via Poincar\'e-Birkhoff-Witt theorem, the expected expression
for the coefficients $r_{i_1,\ldots,i_d}=\langle R\bv P_1^{i_1}\ldots P_d^{i_d}\rangle$.
Hence, $(\Ann^{\bot}(S),\shuffle)$ is generated by ${S}_1,\ldots,{S}_d$.
\end{proof}

With the notations of Proposition \ref{gcdp}, one has respectively

\begin{corollary}\label{polynomialrealization}
If $S\in\K[{S}_1,\ldots,{S}_d]$
then, for any $i=0,1$ and $j=1,..,d$,
one has $x_i\resg S\in\Ann^{\bot}(S)=\K[{S}_1,\ldots,{S}_d]$.
\end{corollary}

\begin{corollary}\label{cns}
The power series $S$ verifies the growth condition if and only if,
for any $i=1,..d$, ${S}_i$ also verifies the growth condition.
\end{corollary}

\begin{proof}
Assume their exists $j\in[1,..,d]$ such that ${S}_j$ does not verify the
growth condition.
Since $S\in\Ann^{\bot}(S)$ then using the decomposition of $S$
on ${S}_1,\ldots,{S}_d$, one obtains a contradiction with the fact
that $S$ verifies the growth condition.

Conservely, using Proposition \ref{shufflegrowth}, we get the expected results.
\end{proof}

\begin{theorem}[\cite{fliessrealisation}]
The formal power series $S\in\KXX$ is differentially produced
if and only if its Lie rank is finite and if it verifies the $\chi$-growth condition.
\end{theorem}

\begin{proof}
By Proposition \ref{direct}, one gets a direct proof.

Conversely, since the Lie rank of $S$ equals $d$ then by Proposition \ref{gcdp}, by putting $\sigma f_{|_0}=S$ and,
for any $j=1,..,d$, $\sigma\bar{q}_i={S}_i$,
\begin{enumerate}
\item we choose the observation $f$ as follows
\begin{eqnarray*}
f(\bar{q}_1,\ldots,\bar{q}_d)&=&\sum_{i_1,\ldots,i_d\ge0}
\frac{r_{i_1,\ldots,i_n}}{i_1!\ldots i_d!}
\bar{q}_1^{i_1}\ldots\bar{q}_d^{i_d}
\quad\in\quad\K[\![\bar{q}_1,\ldots,\bar{q}_d]\!],
\end{eqnarray*}
such that
\begin{eqnarray*}
\sigma f_{|_0}(\bar{q}_1,\ldots,\bar{q}_d)
&=&\sum_{i_1,\ldots,i_d\ge0}\frac{r_{i_1,\ldots,i_n}}{i_1!\ldots i_d!}
(\sigma \bar{q}_1)^{\minishuffle i_1}\shuffle\ldots\shuffle(\sigma \bar{q}_d)^{\minishuffle i_d},
\end{eqnarray*}

\item it follows that, for $i=0,1$ and for $j=1,..,d$, the residuals $x_i\resg\sigma\bar{q}_j$ belongs to
$\Ann^{\bot}(\sigma f_{|_0})$ (see also Lemma \ref{stable}),

\item since $\sigma f$ verifies the $\chi$-growth condition then, by Corollary \ref{cns}, the generating series $\sigma\bar{q}_j$ and $x_i\resg\sigma\bar{q}_j$ (for $i=0,1$ and for $j=1,..,d$) verify also the growth condition.
We then take (see Lemma \ref{fields})
\begin{eqnarray*}
\forall i=0,1,&\forall j=1,..,d,&
\sigma A_j^i(\bar{q}_1,\ldots,\bar{q}_d)=x_i\resg\sigma\bar{q}_j,
\end{eqnarray*}
by expressing $\sigma A_j^i$ on the basis $\{\sigma\bar{q}_i\}_{i=1,..,d}$ of $\Ann^{\bot}(\sigma f_{|_0})$,

\item the homomorphism $\calA$ is then determined as follows
\begin{eqnarray*}
\forall i=0,1,\quad\calA(x_i)
&=&\sum_{j=0}^dA_j^i(\bar{q}_1,\ldots,\bar{q}_d)\frac{\partial}{\partial \bar{q}_j},
\end{eqnarray*}
where, for $i=0,1, j=1,..,d,A_j^i(\bar{q}_1,\ldots,\bar{q}_d)=\calA(x_i)\circ\bar{q}_j$ (see Lemma \ref{coefficients}).
\end{enumerate}
Thus, $(\calA,f)$ provides a differential representation\footnote{
In \cite{fliessrealisation,reutenauerrealisation}, the reader can found the discussion
on the {\it minimal} differential representation.} of dimension $d$ of $S$.
\end{proof}

Moreover, one also has the following

\begin{theorem}[\cite{fliessrealisation}]
Let $S\in\KXX$ supposed to be a differentially produced formal power series.
If $(\calA,f)$ and $(\calA',f')$ are two differential representations
of dimension $n$ of $S$ then there exists a continuous and convergent automorphism
$h$ of $\K$ such that, for $w\in X^*,g\in\K,h(\calA(w)\circ g)=\calA'(w)\circ(h(g))$ and $f'=h(f)$.
\end{theorem}

Since any rational power series (resp. polynomial),
verifies the growth condi\-tion and its Lie rank is less or equal
to its Hankel rank which is finite \cite{fliessrealisation} then

\begin{corollary}\label{polynomialrealisation}
Any rational power series and any polynomial over $X$
with coefficients in $\K$ are differentially produced.
\end{corollary}

\begin{remark}
\begin{enumerate}
\item By Corollary \ref{polynomialrealization}, if $S$ is polynomial then for any $j=1,..,d$, ${S}_j$ is polynomial.
Therefore, for $i=0,1$ and $j=1,..,d$, $x_i\resg S$ is also polynomial over $X$.
In this case, let $(\calA,f)$ be a differential representation of $S$ of dimension $d$.
Then $f$ and $\{A_j^i\}_{j=1,..,d}^{i=0,1}$ are obviously polynomial on $\bar{q}_1,\ldots,\bar{q}_d$
and the Lie algebra generated by $\{\calA(x_i)\}^{i=0,1}$ is nilpotent.

\item Note also that, by Theorem \ref{representativeseries}, if $S$ is rational over $X$
of linear representation $(\lambda,\mu,\eta)$ then the observation
$f(q_1,\ldots,q_n)$ equals $\lambda_1q_1+\ldots+\lambda_nq_n$
and the polysystem $\{\calA(x)\}_{x\in X}$ is obtained by putting
\begin{eqnarray*}
\forall x_i\in X,\quad\calA(x_i)&=&\sum_{j=1}^{n}(\mu(x_i))_j^i\frac{\partial}{\partial q_j}
\end{eqnarray*}
yields {\em linear} representation not necessarily
of minimal dimension \cite{fliessrealisation}.

\item Assume $S\in\K\epsilon\oplus x_0\KXX x_1$ and $S$ is a differentially produced.
If there exists a basis ${S}_1,\ldots,{S}_d$ of
$(\Ann^{\bot}(S),\shuffle)\cong(x_0\KXX x_1,\shuffle)$
such that
\begin{eqnarray}
S=\sum_{i_1,\ldots,i_d\ge0}r_{i_1,\ldots,i_n}
\frac{S_1^{\minishuffle i_1}}{i_1!}\shuffle\ldots\shuffle\frac{S_d^{\minishuffle i_d}}{i_d!}
\in\ (\K[{S}_1,\ldots,{S}_d],\shuffle).\label{decompositionX}
\end{eqnarray}
We put $\Sigma_i:=\pi_YS_i$, for $i=1,..,d$ and then
\begin{eqnarray}
\Sigma:=\sum_{i_1,\ldots,i_d\ge0}r_{i_1,\ldots,i_n}
\frac{\pi_1^{\ministuffle i_1}}{i_1!}
\stuffle\ldots\stuffle\frac{\Sigma_d^{\ministuffle i_d}}{i_d!}
\in(\K[{\Sigma}_1,\ldots,{\Sigma}_d],\stuffle).\label{decompositionY}
\end{eqnarray}
It is a generalization of a Radford's theorem because \cite{FPSAC95,HDR}~:
\begin{itemize}
\item If $S\in\QX$ then (\ref{decompositionX}), (\ref{decompositionY})
are decompositions on Radford bases.
\item If $S$ is rational then these are {\em noncommutative partial decompositions}.
In general one has $\pi_YS\neq\Sigma$ but $\zeta(S_i)=\zeta(\Sigma_i)$ and
\begin{eqnarray}
\zeta(S)=\zeta(\Sigma)=\sum_{i_1,\ldots,i_d\ge0}r_{i_1,\ldots,i_n}
\frac{\zeta(S_1)^{i_1}}{i_1!}\ldots\frac{\zeta(S_d)^{i_d}}{i_d!}.
\end{eqnarray}
Thus, these yield also identities on polyz\^etas at arbitrary weight \cite{words03}.
\end{itemize}
\end{enumerate}
\end{remark}

\end{document}